%% file: Revision2.tex
\documentclass{amsart}
\usepackage[utf8]{inputenc}

\usepackage{a4wide}
\usepackage{mathabx}
\usepackage[english]{babel}
\usepackage{amsthm,mathtools}

\usepackage{xcolor}
\usepackage[colorlinks,
    linkcolor={red!50!black},
    citecolor={blue!50!black},
    urlcolor={blue!80!black}]{hyperref}

\usepackage{tikz}
\usetikzlibrary{cd}    
\usepackage[all]{xy}
\usepackage{amsmath}
\usepackage{amsthm}
\usepackage{amssymb}
\usepackage{amsfonts}
\usepackage{amsopn}
\usepackage[nobysame,alphabetic, initials]{amsrefs}
\usepackage{hyperref}
\usepackage[color=green!30]{todonotes}

\usepackage{subcaption}

\usepackage{import}

\usetikzlibrary{calc,cd,intersections,decorations.markings,patterns,decorations.pathreplacing,matrix,arrows}

\def\Z{\mathbb{Z}}
\def\Q{\mathbb{Q}}
\def\R{\mathbb{R}}
\def\F{\mathbb{F}}

\def\C{\mathbb{C}}
\def\ee{\mathfrak{e}}
\def\ff{\mathfrak{f}}
\def\L#1{{#1}[t^{\pm 1}]}
\def\O#1{#1(t)}
\def\LF{\L{\F}}

\def\LC{\L{\C}}

\def\OC{\O{\C}}

\def\ol#1{\overline{#1}}

\def\ds{\delta\sigma}
\def\hodgep{\mathcal{P}}

\def\basicR{\mathsf{R}}
\def\basicC{\mathsf{C}}
\def\basicF{\mathsf{F}}
\def\makeithash#1{#1^\#}
\def\makeithashT#1{#1^{\#T}}
\def\eps{\epsilon}
\newcommand{\N}{\mathbb{N}}

\newcommand{\id}{\operatorname{id}}
\newcommand{\sign}{\operatorname{sign}}
\newcommand{\im}{\operatorname{im}}

\DeclareMathOperator{\lk}{lk}

\DeclareMathOperator{\Hom}{Hom}

\DeclareMathOperator\iim{Im}
\DeclareMathOperator\re{Re}

\DeclareMathOperator{\Bl}{Bl}
\DeclareMathOperator{\GL}{GL}

\newcommand{\bsm}{\left(\begin{smallmatrix}}
\newcommand{\esm}{\end{smallmatrix}\right)}

\newtheorem{theorem}{Theorem}[section]

\newtheorem{corollary}[theorem]{Corollary}
\newtheorem{lemma}[theorem]{Lemma}
\newtheorem{proposition}[theorem]{Proposition}

\theoremstyle{definition}
\newtheorem{definition}[theorem]{Definition}

\newtheorem{example}[theorem]{Example}
\newtheorem{convention}[theorem]{Convention}
\newtheorem{algorithm}[theorem]{Algorithm}

\newtheorem{remark}[theorem]{Remark}
\newtheorem{construction}[theorem]{Construction}
\theoremstyle{claim}

\newtheorem*{claim*}{Claim}

\newcommand{\purple}{\textcolor{purple}}

\numberwithin{equation}{section}

\begin{document}
 
\title{Twisted Blanchfield pairings and twisted signatures  III: Applications}
\author{Maciej Borodzik}
\address{Institute of Mathematics, University of Warsaw, ul. Banacha 2, 02-097 Warsaw, Poland}
\email{mcboro@mimuw.edu.pl}
\author[A.~Conway]{Anthony Conway}
\address{The University of Texas at Austin, Austin TX 78712}
\email{anthony.conway@austin.utexas.edu}
\author{Wojciech Politarczyk}
\address{Institute of Mathematics, University of Warsaw, ul. Banacha 2, 02-097 Warsaw, Poland.}
\email{wpolitarczyk@mimuw.edu.pl}

\begin{abstract}
  This paper describes how to compute algorithmically certain twisted signature invariants of a knot $K$ using twisted Blanchfield forms.
  An illustration of the algorithm is implemented on $(2,q)$-torus knots.
  Additionally, using satellite formulas for these invariants, we also show how to obstruct the sliceness of certain iterated torus knots.
\end{abstract}

\maketitle

\section{Introduction}

This paper, which is the last in a series of three~\cite{BCP_Alg,BCP_Top}, illustrates how to compute algorithmically certain signature invariants of a knot~$K$, twisted by a representation~$\pi_1(X_K) \to GL_d(\LF)$, where~$X_K$ denotes the exterior of~$K$. 
Before describing this algorithm, we provide some background on twisted Blanchfield forms and twisted signature invariants.

\subsection{Classical knot theory}

Classical knot theory is concerned with knot invariants that are extracted from the algebraic topology of the knot exterior.
Here,  given a knot $K$, customary notation involves using~$\nu(K)$ for an open tubular neighbourhood of~$K$ and  $X_K:=S^3 \setminus \nu(K)$  for the exterior of~$K$.
Textbook examples of classical knot invariants include the Alexander polynomial~$\Delta_K$ and the Levine-Tristram signature~$\sigma_K \colon S^1 \to \Z$~\cite{LickorishIntroduction,Rolfsen}.

Both of these invariants are extracted from the infinite cyclic cover~$X_K^\infty \to X_K$,  and can be calculated using Seifert surfaces, i.e.\  compact, connected, oriented surfaces in~$S^3$ with boundary~$K$.
Indeed,   given a Seifert matrix~$A$ for $K$,
the Alexander polynomial and signature at $\omega \in S^1 $
can be expressed as
\begin{align*}
&\Delta_K(t)=\det(tA-A^T),\\
&\sigma_K(\omega)=\sign((1-\omega)A+(1- \overline{\omega})A^T).
\end{align*}
The same goes for the \emph{Alexander module} $H_1(X_K;\Z[t^{\pm1 }]):=H_1(X_K^\infty)$  which is presented, as a~$\Z[t^{\pm 1}]$-module, by the matrix~$tA-A^T$.

Another classical knot invariant, that appears less frequently in textbooks is a nonsingular, sesquilinear, Hermitian form
$$ \Bl(K) \colon H_1(X_K;\Z[t^{\pm 1}]) \times H_1(X_K;\Z[t^{\pm 1}])  \to \Q(t) /\Z[t^{\pm 1}],$$
known as the \emph{Blanchfield form.}
Sesquilinearity refers to the fact that $\Bl(K)(px,qy)=p\overline{q}\Bl(K)(x,y)$ for every $x,y \in H_1(X_K;\Z[t^{\pm 1}])$ (where, given a rational function $p:=p(t) \in \Q(t)$,  we write $\overline{p}:=p(t^{-1})$) and $\Bl(K)$ being Hermitian means that $\Bl(K)(y,x)=\overline{\Bl(K)(x,y)}$.
The Blanchfield form can also be expressed using Seifert matrices~\cite{KeartonBlanchfieldSeifert,FriedlPowell}.
The data of the pair~$(H_1(X_K;\Z[t^{\pm 1}]),\Bl(K))$ encapsulates both the Alexander polynomial and the signature: $H_1(X_K;\Z[t^{\pm 1}])$ determines~$\Delta_K$ and~$\Bl(K)$ determines~$\sigma_K$.

Applications of classical invariants include the study of the unknotting number,  the $3$-genus, fibredness as well as questions related to $4$-dimensional topology and knot concordance.
Here \emph{knot concordance} refers to the study of (topologically) slice knots i.e.\ knots that bound a locally flat disc in~$D^4$.

 \subsection{Twisted knot invariants}

More recently, knot theorists have taken up the study of 
 invariants of pairs~$(K,\beta)$, where~$K$ is a knot and~$\beta \colon \pi_1(X_K) \to GL_d(\LF)$ is a representation; (in fact,  instead of the exterior~$X_K$, it is often convenient to use the closed~$3$-manifold~$M_K$ obtained from~$S^3$ by~$0$-framed surgery on~$K$).
 The idea is that classical invariants correspond to the case where~$\beta$ is the trivial representation,  whereas nontrivial representations capture more information about the fundamental group. 

The first papers on the topic, such as~\cite{Lin, KirkLivingston}, focused on twisted Alexander polynomials, but the theory has also been leveraged to construct signature invariants~\cite{LevineEta,FriedlEta,KirkLivingston,Nosaka,ChaJEMS} and arguably has its root in earlier work of Casson and Gordon~\cite{CassonGordon1,CassonGordon2}.
Many of the applications and successes
of twisted knot invariants 
are discussed in~\cite{FriedlVidussiSurvey} 
but we simply note that one common area of application again concerns the study of knot concordance; see e.g.~\cite{HeraldKirkLivingston,
KirkLivingstonConcordanceAndMutation,
KirkLivingstonMutation, FriedlEta,
FriedlEtaLink,
LivingstonMeier,MillerPretzel,Miller2Bridge}.

More recently,  taking inspiration from work of Cochran-Orr-Teichner~\cite{CochranOrrTeichner},
Miller and Powell initiated the study of the twisted Blanchfield form $\Bl_\beta(K)$, also with an eye towards applications to knot concordance group~\cite{MillerPowell}.
The absence of computational tools (such as the Seifert matrix in the classical case) has however limited further applications.

The goal of our sequence of papers is to further develop the theory of twisted Blanchfield forms and signatures,  address computational aspects,  and obtain further results in knot concordance.
Our theory was applied in~\cite{ConwayKimPolitarczyk} to study the subgroup of the knot concordance group generated by iterated torus knots; see also~\cite{ConwayNagelFibered}.

 \color{black}

\subsection{Twisted Blanchfield forms and twisted signatures}

Given a knot~$K \subset S^3$ and a unitary acyclic representation~$\beta \colon \pi_1(M_K) \to GL_d(\LC)$ (we call $\beta$ \emph{acyclic} if the $\LC$-module $H_1(M_K;\LC^d_{\beta})$ is torsion for each $i$), the twisted Blanchfield form is a non-singular, sesquilinear Hermitian pairing 
$$ \Bl_{\beta}(K) \colon H_1(M_K;\LF^d_{\beta}) \times H_1(M_K;\LF^d_{\beta})  \to \C(t) /\LC.$$
Details will be recalled in Section~\ref{sec:TwistedBlanchfield}, but we nevertheless note two facts.
First, when~$\beta$ is induced by abelianisation,~$\Bl_{\beta}(K)$ reduces to the classical Blanchfield form~$\Bl(K)$.
Secondly, while~$\Bl_{\beta}(K)$ has been studied by Powell~\cite{PowellThesis,Powell} and Miller-Powell~\cite{MillerPowell}, extracting concrete computable invariants from~$\Bl_{\beta}(K)$ has remained challenging.
For instance, while~$\Bl_{\beta}(K)~$ has been used to obstruct certain satellite knots from being slice~\cite{MillerPowell}, no explicit calculations have been possible for low crossing knots.

In the first two papers of this series~\cite{BCP_Alg, BCP_Top}, we proved a classification result for linking forms over~$\LF$, where $\F=\R,\C$, and described how to use this result to extract computable signature invariants from~$\Bl_{\beta}(K)$, including a new \emph{twisted signature function}
$$ \sigma_{K,\beta} \colon S^1 \to \Z$$
that reduces to~$\sigma_K$ for abelian~$\beta$ and that is related to the Casson-Gordon invariants~\cite{CassonGordon1,CassonGordon2} for metabelian~$\beta$.
We note that the Blanchfield form $\Bl_{\beta}(K)$ (and therefore the signature function~$\sigma_{K,\beta}$) only depends on the conjugacy class of~$\beta$.

Thanks to results from~\cite{BCP_Alg} and to formulas for the behavior of~$\Bl_{\beta}(K)$ under satellite operations from~\cite{BCP_Top}, this paper shows how twisted signature invariants of iterated torus knots can be understood explicitly.

\begin{remark}
\label{rem:MetabelianIntro}
The relationship between $\sigma_{K,\beta}$ and Casson-Gordon invariants arises by considering a representation $\alpha_K(n,\chi) \colon \pi_1(M_K) \to GL_n(\LC)$, where $\chi \colon H_1(\Sigma_n(K);\Z) \to \Z_m$ is a character; here $\Sigma_n(K)$ denotes the $n$-fold branched cover of $K$.
The definition of this representation will be recalled in Section~\ref{sec:MetabelianBlanchfield}, but, from now on, we refer to $\Bl_{\alpha(n,\chi)}(K)$ as a \emph{metabelian Blanchfield form} and to the corresponding signatures as \emph{metabelian signatures}.
\end{remark}

\subsection{Twisted signatures of iterated torus knots}

In order to describe our computation of twisted signatures for (iterated) torus knots, we briefly recall from~\cite{BCP_Top} how $\sigma_{K,\beta}$ is constructed; more details will be given in Section~\ref{sec:TwistedBlanchfield}.
As we review in Section~\ref{sec:SignaturesOfLinkingForms}, given a torsion $\LC$-module~$M$, our work from~\cite{BCP_Alg} associates to every linking form $\lambda \colon M \times M \to \C(t)/\LC$ a
locally constant signature function
$$ \sigma_{(M,\lambda)} \colon S^1 \to \Z. $$
The twisted signature $\sigma_{K,\beta}$ is obtained as the signature function of the linking form $\Bl_{\beta}(K)$.
In practice however, instead of describing the signature function, it is often more convenient to describe its value at the points $\xi \in S^1$ where it jumps.
Without going into details, given $\xi \in S^1$, these values are captured by the \emph{signature jump}
$$ \delta \sigma_{(M,\lambda)}(\xi) \in \Z,$$
the collection of which 
determines $\sigma_{M,\lambda}$ up to an additive constant.
In turn, the signature jumps can be explicitly calculated if one knows the isometry type 
of the linking form $(M,\lambda)$.
Indeed, every linking form decomposes uniquely as
\begin{equation}\label{eq:splittingIntro}
(M,\lambda)=
    \bigoplus_{\substack{n_i,\eps_i,\xi_i\\ i\in I}}\ee(n_i,\eps_i,\xi_i,\C)\oplus
    \bigoplus_{\substack{\xi_j\\ j\in J}}\ff(n_j,\xi_j,\C)
    \end{equation}
    where the summands are explicit linking forms (whose definitions are recalled in Subsection~\ref{sec:SignaturesOfLinkingForms}) and the value of $\delta \sigma_{(M,\lambda)}(\xi)$ can be read off this decomposition.
    
\begin{remark}
\label{rem:SignatureComputable}
In the decomposition of~\eqref{eq:splittingIntro}, the parameters~$n_i \in \N_{>0}$ and $\xi_i \in \C$ can be read off the primary decomposition of the module $H$; the knowledge of the linking form is not necessary.
The signs~$\epsilon_i \in \lbrace \pm 1 \rbrace$ can be obtained if one knows the value of the linking form on every pair of elements of the module~$\LF/\basicF_{\xi_i}^{n_i}$ that supports $\ee(n_i,\eps_i,\xi_i,\C)$.
\end{remark}
    
Summarising, given a linking form $(M,\lambda)$, Remark~\ref{rem:Algorithm} below describes an algorithm to obtain the aforementioned decomposition of $(M,\lambda)$ and to read off its signature jumps.
Thanks to this algorithm, the following theorem leads to a calculation of certain metabelian signatures of the~$(2,2k+1)$-torus knots.
 
 \begin{theorem}
  \label{thm:DecompoTwistedBlanchfieldIntro}
For any $k>0$ there are $2k+1$ characters $H_1(\Sigma_2(T_{2,2k+1})) \to \Z_{2k+1}$ which are denoted by $\chi_\theta$ for $\theta=0,\ldots,2k$. 
Each of the unitary representations $\alpha(2,\chi_{\theta})$ is acyclic and, for $\theta=1,\ldots,k$, the representation $\alpha(2,\chi_{\theta})$ is conjugate to $\alpha(2,\chi_{2k-\theta})$.

Set \(\xi = \exp\left(\frac{2 \pi i}{2k+1}\right)\).
 For  \(1 \leq \theta \leq k\), there exists an isometry
  \begin{align*}
    \Bl_{\alpha(2,\chi_\theta)}(T_{2,2k+1}) \cong \lambda_{\theta}^{even} \oplus \lambda_{\theta}^{odd},
  \end{align*}
  where the linking forms~$\lambda_{\theta}^{even}$ and~$\lambda_{\theta}^{odd}$ are as follows:
  \begin{align*}
    \lambda_{\theta}^{odd} &= \bigoplus_{\stackrel{1 \leq e \leq k}{2 \nmid \theta + e}} \left(\mathfrak{e}(1,1, \xi^{e},\C) \oplus \mathfrak{e}(1,-1,\xi^{-e},\C)\right),
  \end{align*}
  and
  \begin{align*}
    \lambda_{\theta}^{even} &= \bigoplus_{\stackrel{1 \leq e < \theta}{2 \mid \theta+e}}\left(\mathfrak{e}(1,1,\xi^{e}, \C) \oplus \mathfrak{e}(1,-1,\xi^{-e},\C)\right) \oplus \\
                 &\bigoplus_{\stackrel{\theta < e \leq k}{2 \mid \theta+e}}\left(\mathfrak{e}(1,-1,\xi^{e}, \C) \oplus \mathfrak{e}(1,1,\xi^{-e},\C)\right).
  \end{align*}
\end{theorem}

Thanks to Theorem~\ref{thm:DecompoTwistedBlanchfieldIntro} and to our satellite formula for metabelian Blanchfield forms (see Theorem~\ref{thm:metabelian-cabling-formula}), we are able to calculate twisted signatures of linear combinations of iterated torus knots of the form $T_{2,q_1;2,q_2;\ldots;2,q_n}$.
Here, \(T_{\ell,d;r,s}\) denotes the \((r,s)\)-cable of the \((\ell,d)\)-torus knot.
As should be apparent from Theorem~\ref{thm:DecompoTwistedBlanchfieldIntro}, a general formula for the twisted signatures of such knots is prone to be quite unruly.
Therefore instead of stating such a result, we illustrate our methods by obstructing the sliceness of $T_{2,3;2,13} \# T_{2,15} \# -T_{2,3;2,15} \# -T_{2,13}$, an example already considered by Hedden, Kirk and Livingston~\cite{HeddenKirkLivingston}.

\subsection{The algorithm}
\label{sub:Algorithm}

Finally, we describe how to compute the twisted signature jumps of a knot
associated to representations that take values in $\GL_d(\C[t^{\pm 1}])$
 (e.g.\ for the metabelian representation~$\alpha(n,\chi)$).
In particular, this is the procedure we follow to prove Theorem~\ref{thm:DecompoTwistedBlanchfieldIntro}.

Fix a unitary acyclic representation $\beta \colon \pi_1(M_K) \to GL_d(\LC)$. 
As was alluded to in Remark~\ref{rem:SignatureComputable}, the twisted signature jumps can be computed if one knows the (primary decomposition of the) twisted Alexander module $H_1(M_K;\LC^d_{\beta})$ and the value of the twisted Blanchfield pairing~$\Bl_{\beta}(K)$ on every pair of elements on this module; the latter is possible thanks to work of Powell~\cite{PowellThesis} that we recall in Subsection~\ref{sub:review_powell}.
The procedure to compute the twisted signature~$\sigma_{K,\beta}$ is now as follows.

\begin{description}
\item[Step 0] 
Fix a handle decomposition of the $0$-surgery~$M_K$, consider the induced presentation~$\langle x_1,\ldots,x_n \mid r_1,\ldots,r_{n-1} \rangle$ of~$\pi_1(M_K)$,  and calculate the Fox derivatives $\partial x_i/\partial x_j$.
Note that a handle decomposition of $M_K$ arises from a Wirtinger presentation of $\pi_1(X_K)$, as explained in~\cite[Section 3.1]{MillerPowell}.
\item[Step 1] Determine the primary decomposition of the twisted Alexander module over~$\LC$.

For some positive integers~$n_i$ and some~$\xi_i \in S^1$,
the primary decomposition of the torsion~$\LC$-module~$H_1(M_K,\LC^d_{\beta})$ takes the form
\begin{equation}
\label{eq:PrimaryDecompositionModule}
 H_1(M_K,\LC^d_{\beta})=\bigoplus_{\xi_i,n_i} \LC/(t-\xi_i)^{n_i} \oplus H',
 \end{equation}
 where $H'$ is a $\LC$-module on which multiplication by $(t-\xi)$
 is an isomorphism for all $\xi\in S^1$.
The parameters~$\xi_i \in \C$ and~$n_i \in \N_{>0}$ that appear in this decomposition determine the~$\xi_i$ and the~$n_i$ from the $\mathfrak{e}$-forms in the decomposition of~$\Bl_{\beta}(K)$ displayed in~\eqref{eq:splittingIntro}.
Thus, the twisted Alexander module determines the mod~$2$ value of the twisted signature jumps.
This step can be carried out using the Fox derivatives that were calculated in Step~0.
\item[Step 2] Determine the data needed to calculate the twisted Blanchfield pairing~$\Bl_{\beta}(K)$. 

Let~$(C^*(\widetilde{M}_K), \delta)$ be the cochain complex of left~$\Z[\pi_1(M_K)]$-modules of the universal cover~$\widetilde{M}_K$.
Use the presentation for~$\pi_1(M_K)$, the Fox derivatives and~\cite[Equation following Theorem 3.15]{MillerPowell}  to obtain a matrix for the~$\Z[\pi_1(M_K)]$-linear symmetric structure map~$\Phi \colon C^2(\widetilde{N}) \to C_1(\widetilde{N})$
(i.e.\ informally speaking ``Poincaré duality on the chain level"; see Remark~\ref{rem:CellularHandle} for details).
Passing to twisted chain complexes,~$\delta$ and~$\Phi$ induce maps
\begin{align*}
&\beta(\delta) \colon C^1(M_K;\LC^d_{\beta}) \to C^{2}(M_K;\LC^d_{\beta}), \\
&\beta(\Phi) \colon C^2(M_K;\LC^d_{\beta}) \to C_1(M_K;\LC^d_{\beta}).
\end{align*}
Matrices are assumed to act on row vectors from the right and the cohomological differentials are determined by the homological differentials via the formula~\(\beta(\delta^{i}) = (-1)^{i} \beta(\partial_{i})^{\#T}\), where $\#$ denotes the involution on $\LC$ given by $p(t) \mapsto \overline{p}(t^{-1})$, as described at the end of the introduction.
The twisted Blanchfield pairing is isometric to the pairing
  \begin{align*}
 H^2(N;\LC^d_{\beta}) \times H^2(N;\LC^d_{\beta}) &\to \C(t)/\LC \nonumber \\ 
    ([v],[w]) & \mapsto \frac{1}{s} \left(v \cdot \beta(\Phi) \cdot Z^{\# T}\right)^{\# T},
  \end{align*}
  where~$v,w \in C^2(M_K;\LC^d_{\beta})$ and~$Z \in C^1(M_K;\LC^d_{\beta})$ satisfies~$Z\beta(\delta)=sw$ for some~$s \in~\LC \setminus \lbrace 0 \rbrace$; see \eqref{eq:ChainBlanchfieldFormula}.
\item[Step 3]
Determine the twisted Blanchfield pairing on the generators of the cyclic summands~$\LC/(t-\xi_i)^{n_i}$. More precisely, determine the signs~$\epsilon_i$ in the decomposition~\eqref{eq:splittingIntro} of the twisted Blanchfield pairing.

Given~$\xi\in S^1$, a complex polynomial~$r(t)$ is called \emph{$\xi$-positive} if~$(t^{-1}-\overline{\xi})r(t)$ is a complex symmetric polynomial and the function
$\theta\mapsto (e^{-i\theta}-\overline{\xi})r(e^{i\theta})$ changes sign from positive to negative as~$\theta$ crosses the value~$\theta_0$  for which~$e^{i\theta_0}=\xi$.
Let~$1_{\xi_i,n_i}$ be a generator of the cyclic module~$\LC/(t-\xi_i)^{n_i}$ displayed in~\eqref{eq:PrimaryDecompositionModule}.
Write the expression for~$\Bl_{\beta}(K)(1_{\xi_i,n_i},1_{\xi_i,n_i})$ obtained in Step~$2~$ as 
$$\Bl_{\beta}(K)(1_{\xi_i,n_i},1_{\xi_i,n_i})=\frac{h}{(t-\xi_i)^{n_i}},$$
where we think of~$h$ as an element of~$\LC$ determined modulo~$f:=(t-\xi_i)^{n_i}$.
With these notations,~$\epsilon_i=1$ if~$\frac{h}{f}(t-\xi_i)^{(n_i+1)/2}(t^{-1}-\overline{\xi}_i)^{(n_i-1)/2}$ is~$\xi_i$-positive, and~$\epsilon_i=-1$ otherwise.
\item[Step 4] Collect the data from the previous steps and compute the signature jumps.

Using Steps~$1$ and~$3$, one can read off the decomposition of~$\Bl_{\beta}(K)$ displayed in~\eqref{eq:splittingIntro}: Step~$1$ provides the parameters~$n_i \in \N_{>0}$ and~$\xi_i \in \C$, while Step~$3$ describes how to obtain the signs~$\epsilon_i$.
The signature jump is now obtained as
$$    \ds_{(M,\lambda)}(\xi)=\sum_{\substack{n \textrm{\ odd}\\ \eps=\pm 1}} \eps \hodgep(n,\eps,\xi,\C),$$
where $\hodgep(n,\eps,\xi,\C)$ is the number of times~$\mathfrak{e}(n,\epsilon,\xi,\C)$ enters the decomposition of $(M,\lambda)$.
\end{description}

\subsection{Discussion of the algorithm}

In order to catalogue the difficulties that can arise when implementing this algorithm, we first list some familiar problems that occur in such settings but purposefully avoid the terminology of computational complexity theory (our aim is to describe the limitations of our algorithm, not analyse it formally):
\begin{itemize}
 \item[(i)]  problems that can be solved algorithmically using reasonable time and memory relative to the size of the problem;
  \item[(ii)]  problems that can be solved algorithmically using significant time and memory relative to the size of the problem;
  \item[(iii)]  problems for which there is no algorithm that produces an exact solution.
\end{itemize} 


We now our analyse our algorithm in light of these familar problems; the initial data we are given is a diagram for our knot $K$.
\begin{description}
\item[Step 0] Determining a presentation, Fox derivatives and a longitude are each problems of type (i).
\item[Step 1]  Determining the primary decomposition of the twisted Alexander module over~$\LC$ involves problems of types (i),(ii) and (iii) as we now explain in more detail.
\begin{itemize}
\item Calculating the twisted chain complex from the Fox derivatives is a problem of type~(i) whereas calculating the homology of this chain complex of $\C[t^{\pm 1}]$-modules typically involves finding the Smith normal form of the differentials; a problem of type (ii).
\item Finding the primary decomposition of a $\C[t^{\pm 1}]$-module involves finding the roots of a complex polynomial (the order of the module) and this is a problem of type (iii).
\end{itemize}
\item[Step 2] Determining the data needed to calculate the twisted Blanchfield pairing~$\Bl_{\beta}(K)$ consists of calculating further Fox derivatives, as well as finding an identity for the presentation of $\pi_1(M_K)$ (see
Section~\ref{sub:ChainComplex} for details) both of which are problems of type (i).
\item[Step 3] Determining the twisted Blanchfield pairing on the generators of the cyclic summands~$\LC/(t-\xi_i)^{n_i}$ is a problem of type (i) if exact values for the roots of the twisted Alexander polynomial are known, otherwise it is a problem of type (iii).
\item[Step 4]  Collecting the data from the previous steps and computing the signature jumps is a problem of type (i).
\end{description}

Summarising,  if one puts aside issues related to the size of the knot, the main issue with our algorithm lies in finding the roots of the twisted Alexander polynomial.

The reason we are able to make such explicit calculations with torus knots is now apparent:
the knot group~$\pi_1(X_{T_{p,q}})$ admits a presentation with two generators and one relation (so there are few Fox derivatives to calculate),  and the roots of our twisted Alexander polynomials can be determined explicitly; see Proposition~\ref{prop:TwistedModuleComputation} as well as~\cite[Proposition 3.3 and Corollary 3.4]{ConwayKimPolitarczyk}.
Note however that this require we overcome an additional difficulty: in step zero of the algorithm, we cannot use the handle decomposition of $M_{T_{p,q}}$ arising from the Wirtinger presentation of $\pi_1(X_{T_{p,q}})$ as in~\cite[Section 3.1]{MillerPowell}: in Section~\ref{sec:ExplicitComputations} we are forced to find a handle decomposition of $M_{T_{p,q}}$ that induces this presentation.
\color{black}

\subsection{Organisation}
In Section~\ref{sec:SignaturesOfLinkingForms}, we review the classification of linking forms over $\LC$ and the definition of signature jumps.
In Section~\ref{sec:TwistedBlanchfield}, we discuss the twisted Blanchfield form, we describe Powell's algorithm to calculate it on any pair of elements and we recall the definition of twisted signature jumps.
In Section~\ref{sec:MetabelianBlanchfield}, we review metabelian Blanchfield forms and their satellite formulas.
In Section~\ref{sec:Identities},  we collect some facts about identities of group presentations and the cellular chain complex of the universal cover of a $3$-dimensional CW complex.
In Section~\ref{sec:symmetric-structure},  we use identites to describe the symmetric structure on the chain complex of the universal cover of a $3$-manifold.
In Section~\ref{sec:ExplicitComputations}, we give an explicit description of the metabelian Blanchfield forms of $(2,2k+1)$-torus knots.
Finally, in Section~\ref{sec:Sliceness}, we combine these calculations with the satellite formulas to obstruct the sliceness of certain linear combinations of iterated torus knots.
\color{black}

\begin{convention}
  \label{conventions}
  If~$R$ is a commutative ring and~$f,g\in R$, we write~$f\doteq g$ if there exists a unit~$u\in R$ such that~$f=ug$.
  For a ring~$R$ with involution, we denote this involution by~$x\mapsto x^\#$; the symbol~$\ol{x}$ is reserved for the complex conjugation.
  In particular, for~$R=\LC$ the involution~$(-)^\#$ is the composition of the complex conjugation with the map~$t\mapsto t^{-1}$.
  For example, if~$p(t)=t-i$, then~$p^\#(i)=0$, but~$\ol{p}(i)=2i$.
  Given an~$R$-module~$M$, we denote by~$M^\#$ the~$R$-module that has the same underlying additive group as~$M$, but for which the action by~$R$ on~$M$ is precomposed with the involution on~$R$.
  For a matrix~$A$ over~$R$, we write~$\makeithashT{A}$ for the transpose followed by the involution.
\end{convention}

\subsection*{Acknowledgements}
We thank the referee for valuable comments, especially concerning chain level Poincar\'e duality and our choice of handle decompositions.

This paper, together with~\cite{BCP_Alg,BCP_Top}, was originally part of a single paper titled ``Twisted Blanchfield pairings and Casson--Gordon invariants". 
While that paper was being written, the second named author was at University of Geneva and later at Durham University, supported by the Swiss National Science Foundation; during a subset of the revision and splitting process, he was a visitor at the Max Planck Institute for Mathematics.
The first named author is supported by the National Science Center grant 2019/B/35/ST1/01120.
The third named author is supported by the National Science Center grant 2016/20/S/ST1/00369.
All three authors wish to thank the University of Geneva and the University of Warsaw at which part of this work was conducted.
\color{black}

\section{Signatures of linking forms}
\label{sec:SignaturesOfLinkingForms}

We recall from~\cite{BCP_Alg} the construction of certain signature invariants associated to a linking form over~$\LC$.
The main input is a classification of linking forms over~$\LC$ which we also recall.

\subsection{Classification of linking forms}
\label{sub:classify}
A \emph{linking form} will refer to a pair~$(M,\lambda)$, where~$M$ is a torsion~$\LC$-module and~$ \lambda\colon M \times M\to \C(t)/\LC$ is a non-degenerate sesquilinear Hermitian pairing. 
In order to state our classification of linking forms, we recall some terminology.
\begin{definition}\label{def:positive}
  Given~$\xi\in S^1$, a complex polynomial~$r(t)$ is called \emph{$\xi$-positive} if~$(t^{-1}-\ol{\xi})r(t)$ is a complex symmetric polynomial and the function
  $\theta\mapsto (e^{-i\theta}-\ol{\xi})r(e^{i\theta})$ changes sign from positive to negative as~$\theta$ crosses the value~$\theta_0$  for which~$e^{i\theta_0}=\xi$.
\end{definition}

While it is helpful to work with arbitrary~$\xi$-positive polynomials, having a concrete example sometimes also proves useful:
if~$\xi\neq\pm 1$ then~$r(t)=1-\xi t$ is~$\xi$-positive if~$\iim(\xi)>0$ and~$-(1-\xi t)$ is~$\xi$-positive if~$\iim(\xi)<0$.
Next, we refer to
\[\basicC_{\xi}(t) =
  \begin{cases}
    (t-\xi), & \text{if } |\xi|=1,\\
    (t-\xi)(t^{-1}-\xi), & \text{if } |\xi| \in (0,1),
  \end{cases}\]
as \emph{basic polynomials}, and define the building blocks for our classification of linking forms.

\begin{definition}
  Fix an integer \(n>0\) and \(\eps = \pm1\).
  For~$|\xi|=1$, the \emph{basic linking form}~$\ee(n,\eps,\xi,\C)$ is defined as
  \begin{align}
    \LC / \basicC_{\xi}(t)^n \times \LC / \basicC_{\xi}(t)^n &\to \OC/\LC, \nonumber\\
    (x,y) &\mapsto \frac{\eps x y^{\#}}{\basicC_{\xi}(t)^{\frac{n}{2}}\basicC_{\bar{\xi}}(t^{-1})^{\frac{n}{2}}}, \quad \ \ \ \ \  \text{if~$n$ is even}, \label{eq:e_n_k_form_complex_even}\\
    (x,y) &\mapsto \frac{\eps r(t) x y^{\#}}{\basicC_{\xi}(t)^{\frac{n+1}{2}}\basicC_{\bar{\xi}}(t^{-1})^{\frac{n-1}{2}}}, \quad \text{if~$n$ is odd} \label{eq:e_n_k_form_complex_odd}
  \end{align}
  where~$r(t)$ is a~$\xi$-positive polynomial\footnote{The isometry type of~$\ee(n,\eps,\xi,\C)$ does not depend on the choice of~$r$.}.
  For~$|\xi|<1$, the \emph{basic linking form}~$\ff(n,\xi,\C)$ is
  \begin{align}
    \LC / \basicC_{\xi}(t)^n \times \LC / \basicC_{\xi}(t)^n &\to \OC/\LC, \nonumber \\
    (x,y) &\mapsto \frac{x y^{\#}}{\basicC_{\xi}(t)^{n}}.\label{eq:f_n_k_form_complex}
  \end{align}  
\end{definition}

We can now state the classification of linking forms over~$\LC$ that was proved in~\cite{BCP_Alg}.

\begin{theorem}\label{thm:decompose}
  Every linking form~$(M,\lambda)$ over~$\LC$ can be presented as a direct sum
  \[(M,\lambda)=
    \bigoplus_{\substack{n_i,\eps_i,\xi_i\\ i\in I}}\ee(n_i,\eps_i,\xi_i,\C)\oplus
    \bigoplus_{\substack{\xi_j\\ j\in J}}\ff(n_j,\xi_j,\C)\]
  for some finite set of indices~$I$ and~$J$. Here the~$n_i\ge 0$ are integers, the~$\eps_i$ equal~$\pm 1$ and the~$\xi_i$ and $\xi_j$ are non-zero complex numbers. The presentation is unique up
  to permuting summands.
\end{theorem}

\subsection{Signatures of linking forms}\label{sub:effective}
\label{sub:SignaturesOfLinkingForms}
Given~$n \geq 0,\xi \in \C$ and~$\epsilon=\pm 1$, we define the \emph{Hodge number}~$\mathcal{P}(n,\epsilon,\xi,\C)$ of a linking form~$(M,\lambda)$ over~$\LC$ as the number of times~$\mathfrak{e}(n,\epsilon,\xi,\C)$ enters the decomposition in~Theorem~\ref{thm:decompose}.
\begin{definition}
  \emph{The signature jump} of a linking form~$(M,\lambda)$ over~$\LC$ at~$\xi \in S^1$ is defined as the following integer:
  \begin{equation}
    \label{eq:JumpIntro}
    \ds_{(M,\lambda)}(\xi)=\sum_{\substack{n \textrm{\ odd}\\ \eps=\pm 1}} \eps \hodgep(n,\eps,\xi,\C).
  \end{equation}
\end{definition}

We collect some properties of signature jumps for later use; details can be found in~\cite{BCP_Alg}.

\begin{theorem}\label{thm:witt_equivalence}
  Signature jumps satisfy the following properties:
  \begin{enumerate}
  \item The signature jump of~$(M,\lambda)$ vanishes at~$\xi \in S^1$ if~$\xi$ is not a root of the order~$\Delta_M$ of~$M$. 
  \item The signature jump is additive: if~$(M_1,\lambda_1)$ and~$(M_2,\lambda_2)$ are linking forms, then the following equality holds for every~$\xi \in S^1$:
    $$\delta\sigma_{(M_1 \oplus M_2,\lambda_1 \oplus \lambda_2)}(\xi)=\delta\sigma_{(M_1,\lambda_1)}(\xi)+\delta\sigma_{(M_2,\lambda_2)}(\xi).$$
  \item A linking form is metabolic if and only if all its signature jumps vanish.\footnote{Recall that~$(M,\lambda)$ is \emph{metabolic} if there exists a submodule~$L \subset M$ such that~$P=P^\perp$.}
  \end{enumerate}
\end{theorem}

Next, we describe how to calculate signature jumps algorithmically.
This underlies several of the steps of the algorithm from Subsection~\ref{sub:Algorithm} of the introduction.
\begin{algorithm}
  \label{rem:Algorithm}
  Given a linking form~$(H,\lambda)$, here is how to calculate~$\delta \sigma_{(H,\lambda)}(\xi)$ for~$\xi\in S^1$.
  \begin{enumerate}
  \item Determine the~$(t-\xi)$-primary summand~$H_\xi$ of~$H$; the linking form~$\lambda$ restricts to a linking form~$\lambda_\xi\colon H_\xi\times H_\xi\to\C(t)/\LC$.
  \item Decompose~$(H_\xi,\lambda_\xi)$ as an orthogonal sum of pairings over cyclic modules
      $$(H_\xi,\lambda_\xi)=(H_\xi^1,\lambda_\xi^1)\oplus\dots\oplus (H_\xi^{r_\xi},\lambda_\xi^{r_\xi}).$$
    The orthogonalisation procedure is algorithmic; see \cite[Lemma 4.3]{BorodzikFriedl2};
  \item 
    By Theorem~\ref{thm:decompose}, for each~$j$,
    $(H_\xi^j,\lambda_\xi^j)$
    is isometric to 
    a~$\ee(n_j,\xi,\eps_j,\C)$, where~$n_j$  
    is such that~$H_\xi^j=\LC/(\basicC_\xi(t))^{n_j}$, and where~$\epsilon_{j}=\pm 1$ is obtained as described in the next step.
  \item If~$n_j$ is even,~$(H_\xi^j,\lambda_\xi^j)$ does not contribute to the signature jump, and so we disregard it.
    If~$n_j$ is odd,~$\epsilon_j$ can be determined as follows.
    Pick a generator~$x$ of the cyclic module~$H_\xi^j$, and write
    $$\lambda_\xi^j(x,x)= \frac{r(t)}{\basicC_{\xi}(t)^{\frac{n+1}{2}}\basicC_{\bar{\xi}}(t^{-1})^{\frac{n-1}{2}}}$$ 
    for some polynomial~$r(t)$. This polynomial has the property that~$(t-\xi)r(t)$ takes real values on the unit circle. 
    We set~$\eps_j=+1$ if~$r(t)$ is~$\xi$-positive and~$\eps_j=-1$ if~$-r(t)$ is~$\xi$-positive. 
  \item To obtain~$\delta \sigma_{(H,\lambda)}(\xi)$, sum up the~$\epsilon_j$ for each~$(H_\xi^j,\lambda_\xi^j)$ obtained so far.
  \end{enumerate}
  The algorithm (except for the last step) is used for linking forms over~$\Z$ with odd determinant; see \cite[Sections 9 and 10]{BGKM}.
  The case of linking forms over~$\LC$ is analogous, because
  both~$\Z$ and~$\LC$ are principal ideal domains.
\end{algorithm}

In the introduction, given a linking form $(M,\lambda)$, we not only mentioned the signature jumps~$\delta\sigma_{(M,\lambda)}$, but also the signature function $\sigma_{(M,\lambda)} \colon S^1 \to \Z$.
While the use of the latter is conceptually enlightening, this paper  only makes use of the former.
\color{black}

\section{Twisted Blanchfield forms}
\label{sec:TwistedBlanchfield}

We briefly review the definition of twisted homology and some first facts about twisted Blanchfield forms; references include~\cite{MillerPowell,BCP_Top}.
While we do not recall the definition of these pairings, we describe an algorithm to calculate them (due to Powell~\cite{PowellThesis}) that may just as well be taken as a definition.  
Throughout this section, we assume that $\F$ is either $\R$ or $\C$.

\subsection{Twisted Blanchfield pairings}\label{sub:review_metabelian}
\label{sub:TwistedBlanchfield}

Let~$X$ be a space with universal cover~$p \colon \widetilde{X} \to X$.
We assume that $X$ has the homotopy type of a finite CW complex.
The left action of~$\pi_1(X)$ on~$\widetilde{X}$ endows the singular
chain groups of~$C_*(\widetilde{X})$ with the structure of left~$\Z[\pi_1(X)]$-modules.
Given a representation~$\beta\colon\pi_1(X)\to GL_d(\LF)$, use~$\LF_\beta^d$ to denote the~$(\LF,\Z[\pi_1(X)])$-bimodule whose right~$\Z[\pi_1(X)]$-module structure is given by right multiplication by~$\beta(\gamma)$ on row vectors.
The 
chain complexes 
\begin{align*}
  C_*(X;\LF_\beta^d)&:=\LF_\beta^d \otimes_{\Z[\pi_1(X)]} C_*(\widetilde{X}) \\
  C^*(X;\LF_\beta^d)&:=\Hom_{\operatorname{right}-\Z[\pi_1(X)]}( C_*(\widetilde{X})^\#,\LF_\beta^d)
\end{align*}
of left~$\LF$-modules will be called the (co)chain complexes of~$X$ twisted by~$\beta$.
The corresponding homology of left~$\LF$-modules~$H_*(X;\LF_\beta^d)$ and~$H^*(X;\LF_\beta^d)$ will be called the (co)homology of~$X$ twisted by~$\beta$.
The representation~$\beta$ is \emph{acyclic} if the chain complex~$\F(t) \otimes_{\LF} C_*(X;\LF^d_\beta)$ is acyclic and \emph{unitary} if~$\beta(\gamma)=\beta(\gamma^{-1})^{\#T}$.

We now assume that $N$ is a closed oriented $3$-manifold.
If~$\beta \colon \pi_1(N) \to GL_d(\LF)$ is a representation that is both acyclic and unitary, then the~$\LF$-module~$H_1(N;\LF^d_\beta)$ is endowed with a linking form
\begin{equation}
  \label{eq:Blanchfield}
  \Bl_{\beta}(N) \colon H_1(N;\LF^d_\beta) \times H_1(N;\LF^d_\beta) \to \F(t)/\LF,
\end{equation}
called a \emph{twisted Blanchfield pairing}. 
The definition of this pairing on $x,y \in H_1(N;\LF^d_\beta)$ is as~$\Bl_\beta(N)(x,y)=\Theta(y)(x)$ where $\Theta$ is the composition
\begin{align*}
  \Theta \colon H_1(N;\F[t^{\pm 1}]^d_\beta) \,
  & \xrightarrow{\operatorname{PD}^{-1}}\, H^2(N;\F[t^{\pm 1}]^d_\beta)  \\
  & \xrightarrow{\operatorname{BS}^{-1}}   H^1(N;(\F(t)/\F[t^{\pm 1}])^d_\beta) \\
  & \xrightarrow{\text{ev}}  \makeithash{\Hom_{\F[t^{\pm 1}]}(H_1(N;\F[t^{\pm 1}]^d_\beta),\F(t)/\F[t^{\pm 1}])}
\end{align*}
of the inverse of the Poincar\'e duality isomorphism,  the inverse of a Bocktein isomorphism and an evaluation map.
\color{black}
While we do not give further details on the definition of this pairing (referring instead to \cite{BCP_Top,MillerPowell,Powell}), the next subsection instead describes an algorithm to compute its value on any pair of elements of~$H_1(N;\LF^d_\beta)$.

\subsection{A review of Powell's algorithm}\label{sub:review_powell}
We briefly recall Powell's algorithm to compute the Blanchfield pairing~\cite{PowellThesis}.
In~\cite{Powell} Powell defines twisted Blanchfield pairings for arbitrary~$3$-dimensional symmetric chain complexes.
When~$N$ is a closed oriented~$3$-manifold and~$\beta \colon \pi_1(N) \to GL_d(\LF)$ is a unitary acyclic representation, his definition yields a non-singular linking form 
$$ \Bl^{\beta}(N) \colon H^2(N;\LF^d_\beta) \times H^2(N;\LF^d_\beta) \to \F(t)/\LF $$
on the twisted \emph{co}homology of~$N$ whose relation to~$\Bl_\beta(N)$ is described in Remark~\ref{rem:CohomologicalPairing} below.
We now focus on the algorithm described in~\cite{PowellThesis,MillerPowell} to compute~$\Bl^{\beta}(N)$, and take the result as our definition of~$ \Bl^{\beta}(N)$.

\begin{remark}
\label{rem:CellularHandle}
Fixing a handle decomposition of $N$ (which we will do below) merely gives rise to a CW structure on a space homotopy equivalent to $N$.
As a consequence,  the singular chain complex of $N$ is chain homotopy equivalent to the cellular chain complex of this auxiliary space
The same can be said for the chain complexes of the universal covers viewed as chain complexes over $\Z[\pi_1(N)]$; see e.g.~\cite[Lemma 4.2]{LuckDimensionTheory}.
Following~\cite{MillerPowell},  we nevertheless slightly abuse notation by writing $C_*(\widetilde{N})$ and $C^*(\widetilde{N},\partial \widetilde{N})$ instead of invoking the space to which $N$ is homotopy equivalent.
Here the point is that the algorithm in~\cite{PowellThesis,MillerPowell} only depends on the chain homotopy type of a given symmetric chain complex.
\end{remark}

Fix a handle decomposition of $N$ and choose a chain representative $[N] \in C_3(N)$ for the fundamental class of $N$.
Here, as indicated in Remark~\ref{rem:CellularHandle} we are technically working in the chain complex of a space homotopy equivalent to $N$.
Use~$\widetilde{N}$ to denote the universal cover of~$N$ and let~$(C^*(\widetilde{N}), \partial^*)$ be the resulting cochain complex of left~$\Z[\pi_1(N)]$-modules.
As explained in~\cite[Proposition 2.10]{MillerPowell} the choice of~$[N]$ together with the symmetric construction~\cite{RanickiAlgebraicTheory} leads to a $\Z[\pi_1(N)]$-chain homotopy equivalence 
\begin{equation}
\label{eq:SymmetricStructure}
 \Phi \colon C^{3-*}(\widetilde{N}) \to C_*(\widetilde{N})
 \end{equation}
which should be thought of as a chain level version of Poincaré duality.
We will not delve into the details of the symmetric construction,  but instead note that 
\color{black}
passing to twisted chain complexes,~$\partial^*$ and~$\Phi$ induce maps
\begin{align*}
  &\beta(\partial^*) \colon C^*(N;\LF^d_\beta) \to C^{*+1}(N;\LF^d_\beta), \\
  &\beta(\Phi) \colon C^{3-*}(N;\LF^d_\beta) \to C_{*}(N;\LF^d_\beta).
\end{align*}
For later use, we also recall that matrices are assumed to act on row vectors from the right and that the cohomological differentials are determined by the homological differentials via the formula~\(\beta(\partial^{i}) = (-1)^{i} \beta(\partial_{i})^{\#T}\).

The following definition is due to Powell~\cite{Powell} (see also~\cite{MillerPowell}).

\begin{definition}
  \label{def:CohomBlanchfield}
  Let~$N$ be a closed oriented~$3$-manifold and let~$\beta \colon \pi_1(N) \to GL_d(\LF)$ be a unitary acyclic representation.
Fix a handle decomposition of $N$ and a  chain representative $[N] \in C_3(N)$ of the fundamental class  and let  $\Phi \colon C^{3-*}(\widetilde{N}) \to C_*(\widetilde{N})$ be the chain homotopy equivalence resulting from the symmetric construction.
  The \emph{cohomological twisted Blanchfield} is defined as 
  \begin{align}
    \label{eq:ChainBlanchfieldFormula}
    \Bl^{\beta}(N) \colon H^2(N;\LF^d_\beta) \times H^2(N;\LF^d_\beta) &\to \F(t)/\LF \nonumber \\ 
    ([v],[w]) & \mapsto \frac{1}{s} \left(v \cdot \beta(\Phi) \cdot Z^{\# T}\right)^{\# T},
  \end{align}
  where~$v,w \in C^2(N;\LF^d_\beta)$ and~$Z \in C^1(N;\LF^d_\beta)$ satisfies~$Z\beta(\partial^2)=sw$ for some~$s \in~\LF \setminus \lbrace 0 \rbrace$.
The fact that this pairing does not depend on any of the choices involved was proved in~\cite{Powell}.
\end{definition}

The next remark summarises the relation between the cohomological pairing of Definition~\ref{def:CohomBlanchfield} and the homological pairing mentioned in Subsection~\ref{sub:review_metabelian}.
\begin{remark}
  \label{rem:CohomologicalPairing}
  The cohomological twisted Blanchfield pairing $\Bl^{\beta}(N)$ is isometric to the twisted Blanchfield pairing $\Bl_{\beta}(N)$ from Subsection~\ref{sub:review_metabelian}.
Indeed~\cite[Proposition 5.3]{MillerPowell} implies that for~$x,y \in H_1(N;\LF^d_\beta)$, the pairings are related by
  \[ \Bl^{\beta}(N)(x,y) = \Bl_{\beta}(N)(\Phi^{-*}(x),\Phi^{-*}(y)),\]
where $\Phi^{-*}=(\Phi^*)^{-1}$ denotes the inverse of the homomorphism on (co)homology induced by $\Phi$.
We emphasise that the reader should consider Definition~\ref{def:CohomBlanchfield} as a computational device: as far as the definitions go,  the approach outlined in Subsection~\ref{sub:TwistedBlanchfield} is more satisfactory.
\end{remark}

\begin{remark}
  \label{rem:HandleDiagram}
Suppose that $N=M_K$ is obtained by $0$-surgery on a knot~$K$.
In this case~\cite[Construction 3.2]{MillerPowell} shows how to associate to any reduced diagram $D$ of $K$ with $c \geq 3$ crossings a specific handle decomposition of $M_K$ with two $3$-handles $h_1^3$ and $h_2^3$.
The cohomological twisted Blanchfield form that Miller and Powell work with is then constructed by applying the symmetric construction to the fundamental cycle $[M_K]:=-h_1^3-h_2^3$~\cite[Remark 3.12]{MillerPowell}.
It is with respect to this handle decomposition and fundamental cycle that Miller-Powell describe an explicit algorithm to calculate the differentials and $\Phi$ in terms of a Wirtinger presentation for~$\pi_1(S^3 \setminus K)$ associated to the diagram $D$~\cite[Section 3 and Theorem 3.9]{MillerPowell}. 
It is in this sense that computing the twisted Blanchfield pairing is algorithmic.
\end{remark}
\color{black}

\subsection{Twisted signatures}
\label{sub:TwistedSignatures}
We briefly recall the definition of the twisted signature invariants from~\cite{BCP_Top}.
As we mentioned in the introduction, given a knot $K$ and a unitary acyclic representation $\beta \colon \pi_1(M_K) \to GL_d(\LC)$, the twisted signature jump of $(K,\beta)$ is the signature jump of the linking form $\Bl_\beta(K)$:
$$ \delta \sigma_{K,\beta}(\xi):=\delta_{(H_1(M_K;\LC^n_\beta), \Bl_\beta(K))}.$$
Properties of $\delta \sigma_{K,\beta}$ can be deduced from the properties of $\delta_{(M,\lambda)}$ listed in Theorem~\ref{thm:witt_equivalence}.
We record two additional remarks for later use:
\begin{remark}
\label{rem:LevineTristram}
 When $\beta \colon \pi_1(M_K) \to GL_1(\LC)$ is the map induced by abelianisation, the twisted Blanchfield form reduces to the classical Blanchfield form and the twisted signature jumps reduce to the jumps of the classical Levine-Tristram signature~\cite{BCP_Top}.
\end{remark}
\begin{remark}
\label{rem:AlgoTop}
Thanks to the algorithm described in Remark~\ref{rem:Algorithm}, the twisted signature jumps $\delta \sigma_{K,\beta}(\xi)$ can be calculated algorithmically from a knot diagram.
Indeed, the key point is that the fourth step of Algorithm~\ref{rem:Algorithm} requires that we be able to calculate elements of the form $\Bl_\beta(K)(x,x)$, where $x \in H_1(M_K;\LC^n_\beta)$.
This is possible thanks to the algorithm as described in Subsection~\ref{sub:review_powell}.
The whole process was summarised in Subection~\ref{sub:Algorithm} from the introduction and is illustrated in Section~\ref{sec:ExplicitComputations} in the case of~$(2,2k+1)$-torus knots. 
\end{remark}

\color{black}

\section{Metabelian Blanchfield forms}
\label{sec:MetabelianBlanchfield}

We now restrict to a specific class of twisted Blanchfield pairing associated to certain metabelian representations that arise in Casson-Gordon theory~\cite{FriedlEta,HeraldKirkLivingston,MillerPowell}.
For a knot~$K$, this representation is of the form
$$\alpha_K(n,\chi) \colon \pi_1(M_K) \to GL_n(\LC),$$
where~$\chi \colon H_1(\Sigma_n(K);\Z) \to \Z_m$ is a prime power order character on the first homology of~$\Sigma_n(K)$, the~$n$-fold cyclic branched cover of~$S^3$ branched along~$K$.
After reviewing the definition of~$\alpha_K(n,\chi)$, we list some properties of the resulting metabelian Blanchfield forms.

\subsection{Metabelian representations}
\label{sub:MetabelianRep}

The abelianisation homomorphism~$\phi\colon\pi_1(M_K)\xrightarrow{\cong}\Z=\langle t_K\rangle$ endows~$\Z[t_K^{\pm 1}]$ with a right~$\Z[\pi_1(M_K)]$-module structure.
This gives rise
to the twisted homology~$\Z[t_K^{\pm1}]$-module~$H_1(M_K;\Z[t_{K}^{\pm1}])$.
As in~\cite[Corollary 2.4]{FriedlPowell}, identify~$H_1(\Sigma_n(K);\Z)$ with the quotient module~$H_1(M_K;\Z[t_K^{\pm 1}])/(t_K^n-1)$. Consider the semidirect product~$\Z \ltimes H_1(\Sigma_n(K);\Z)$, where the group law is given by~$ (t_K^i,v)\cdot (t_K^j,w)=(t_K^{i+j},t_K^{-j}v+w)$. Next, let~$\gamma_K(n,\chi)$ be the homomorphism:
\begin{align}
  \label{eq:Matrix}
  \gamma_K(n,\chi) \colon \Z \ltimes H_1(\Sigma_n(K);\Z) &\to~\operatorname{GL}_n(\LC)  \nonumber \\
  (t_K^j,v) &\mapsto \begin{pmatrix}
    0& 1 & \cdots &0 \\
    \vdots & \vdots & \ddots & \vdots  \\
    0 & 0 & \cdots & 1 \\
    t & 0 & \cdots & 0
  \end{pmatrix}^j
                \begin{pmatrix}
                  \xi_{m}^{\chi(v)} & 0 & \cdots &0 \\
                  0 & \xi_{m}^{\chi(t_K \cdot v)} & \cdots &0 \\
                  \vdots & \vdots & \ddots & \vdots \\
                  0 & 0 & \cdots & \xi_{m}^{\chi(t_K^{n-1} \cdot v)}
                \end{pmatrix}.
\end{align}
Identify the module~$H_1(M_K;\Z[t_{K}^{\pm1}])$ with the quotient~$\pi_1(M_K)^{(1)}/\pi_1(M_K)^{(2)}$ and consider the following composition of canonical projections
\begin{equation}
  \label{eq:qKmap}
  q_K \colon \pi_{1}(M_{K})^{(1)}  \to  H_1(M_K;\Z[t_K^{\pm 1}]) \to H_1(\Sigma_n(K);\Z).
\end{equation}
Fix an element~$\mu_{K}$ in~$\pi_1(M_K)$ such that~$\phi_K(\mu_{K})=t_K$.
For every~$g \in \pi_1(M_K)$, we have~$\phi_K(\mu_K^{-\phi_K(g)}g)=1$ and so~$\mu_K^{-\phi_K(g)}g\in\pi_1(M_K)^{(1)}$.
As a consequence, we obtain the following map:
\begin{align*}
  \widetilde{\rho}_K \colon  \pi_1(M_K) &\to \Z \ltimes H_1(\Sigma_n(K);\Z) \\
  g &\mapsto (\phi_K(g),q_K(\mu_{K}^{-\phi_K(g)}g)). \nonumber
\end{align*}
The unitary representation~$\alpha_K(n,\chi)$ is obtained as the composition
\[
  \alpha_K(n,\chi) \colon \pi_1(M_K) \stackrel{\widetilde{\rho}_K}{\to}  \Z \ltimes H_1(\Sigma_n(K);\Z) \stackrel{\gamma_K(n,\chi)}{\longrightarrow} \operatorname{GL}_n(\LC).
\]

This representation is unitary and, if~$m$ is a prime power and~$\chi \colon H_1(\Sigma_n(K);\Z) \to \Z_m$ is nontrivial, then~$\alpha_K(n,\chi)$ is acyclic; see~\cite{FriedlPowellInjectivity} and~\cite[Lemma 6.6]{MillerPowell}. 
Furthermore, when the knot is clear from the context, we will often write $\alpha(n,\chi)$ instead of $\alpha_K(n,\chi)$.

\begin{definition}
  \label{def:MetabelianBlanchfield}
  For a prime power order representation~$\chi \colon H_1(\Sigma_n(K);\Z) \to \Z_m$, we refer to the twisted Blanchfield form~$\Bl_{\alpha(n,\chi)}(K)$ as a \emph{metabelian Blanchfield form}.
\end{definition}

The relevance of~$\Bl_{\alpha(n,\chi)}(K)$ to knot concordance stems from the following theorem of Miller-Powell~\cite[Theorem~6.9]{MillerPowell} which itself builds on work of Casson and Gordon~\cite{CassonGordon2}.
\begin{theorem}\label{thm:cg_obs}
  Suppose~$K$ is a slice knot. Then for any prime power~$n$ there exists a metaboliser~$P$ of the linking pairing~$H_1(\Sigma_n(K);\Z)\times H_1(\Sigma_n(K);\Z)\to\Q/\Z$ such that for any prime power~$q^a$ and any non-trivial character~$\chi\colon H_1(\Sigma_n(K);\Z)\to\Z_{q^a}$ vanishing on~$P$, the Blanchfield form~$\Bl_{\alpha(n,\chi_b)}(K)$ is metabolic for some~$b\ge a$, where~$\chi_b$ is the composition of~$\chi$ with the inclusion~$\Z_{q^a}\hookrightarrow \Z_{q^b}$.
\end{theorem}

The take-away from Theorem~\ref{thm:cg_obs} is that if we can find enough representations~$\chi$ for which~$\Bl_{\alpha(n,\chi_b)}(K)$ is not metabolic, then we can show that~$K$ is not slice.
To obstruct $\Bl_{\alpha(n,\chi_b)}(K)$ from being metabolic, we will use signature jumps.


\subsection{The satellite formula for the metabelian Blanchfield form}\label{sub:cabling}

Let~$P,K \subset S^{3}\) be knots and let \(\eta\) be a simple closed curve in the complement of \(P\),
The \emph{satellite knot} \(P(K,\eta)\) with \emph{pattern}~\(P\), \emph{companion} \(K\) and \emph{infection curve} \(\eta\)
is the image of \(P\) under the diffeomorphism \((S^{3}\setminus \mathcal{N}(\eta)) \cup_{\partial} (S^{3}  \setminus \mathcal{N}(K)) \cong S^{3}\), where the gluing of the exteriors of \(\eta\) and \(K\) identifies the meridian of \(\eta\) with the zero-framed longitude of \(K\) and vice-versa.
The zero-surgery on the satellite knot \(P(K,\eta)\) can be obtained by the infection of~\(M_{P}\) by \(K\) along \(\eta \subset S^{3} \setminus P \subset M_{P}\):
\begin{equation}
  \label{eq:DecompoSurgerySatellite}
  M_{P(K,\eta)}=M_P \setminus \mathcal{N}(\eta) \cup_\partial S^3 \setminus \mathcal{N}(K).
\end{equation}
Let \(\mu_{\eta}\) denote a meridian of \(\eta\).
A representation \(\beta \colon \pi_{1}(M_{P(K,\eta)}) \to GL_{d}(\F[t^{\pm1}])\) induces representations on~$ \pi_{1}(M_{P} \setminus \mathcal{N}(\eta))~$ and~$\pi_{1}(S^{3} \setminus \mathcal{N}(K))$.
In turn, as explained in \cite[Section 3.3]{BCP_Top}, these representations extend to representations
\[\beta_{P} \colon \pi_{1}(M_{P}) \to GL_{d}(\F[t^{\pm1}]), \quad \beta_{K} \colon \pi_{1}(M_K) \to GL_{d}(\F[t^{\pm1}]).\]
provided \(\beta(\mu_{\eta}) = \id\) and \(\det(\id - \beta(\eta)) \neq 0\), in which case we say that~$\beta$ is $\eta$-regular.

Thus, given a character $\chi \colon H_1(\Sigma_n(P(K,\eta));\Z) \to \Z_m$, if the representation $\alpha(n,\chi)$ is $\eta$-regular, then it gives rise to representations $\alpha(n,\chi)_P$ on $\pi_1(M_P)$ and $\alpha(n,\chi)_K$ on $\pi_1(M_K)$.
The representation $\alpha(n,\chi)_P$ can be shown to to agree with~$\alpha(n,\chi_P)$, where $\chi_P$ is the character induced by~$\chi$ on~$H_1(\Sigma_n(P);\Z)$, as in~\cite[Section 4]{LitherlandCobordism}.
The main step in the proof of the metabelian satellite formula is to decompose $\alpha(n,\chi)_K$ into $h:=\operatorname{gcd}(n,w)$ representations.
To describe the outcome, for $i=1,\ldots,h$, one considers the character
\begin{align*}
 \chi_i \colon H_1(\Sigma_{n/h}(K);\Z) &\to \Z_m \\
v &\mapsto \chi(t_Q(\iota_n(v))),
\end{align*}
where $t_Q$ denotes the generator of the deck transformation group of the infinite cyclic cover of~$M_{P(K,\eta)}$ and~$\iota_n \colon H_1(\Sigma_{n/h}(K);\Z) \to H_1(\Sigma_{n}(P(K,\eta));\Z)$ is inclusion induced; we refer to~\cite{LitherlandCobordism} and~\cite{BCP_Top} for further details on this later map.
\color{black}
Additionally, use~$\mu_Q$ to denote a meridian of~$P(K,\eta)$ and define the map
$$ q_Q \colon \pi_1^{(1)}(M_{P(K,\eta)}) \to H_1(\Sigma_n(P(K,\eta));\Z)$$
as in~\eqref{eq:qKmap}.
The satellite formula for the metabelian Blanchfield form now reads as follows in the winding number~$w:=\ell k(\eta,P) \neq 0$ case; we refer to~\cite{BCP_Top} for the general statement.

\begin{theorem}\label{thm:metabelian-cabling-formula}
  Let \(K,P\) be two knots in \(S^{3}\), let \(\eta\) be an unknotted curve in the complement of~\(P\) with meridian~$\mu_\eta$, let \(w=\lk(\eta,P) \neq 0\), let \(n>1\) and set \(h = \gcd(n,w)\).
  For any character \(\chi \colon H_1(\Sigma_n(P(K,\eta));\Z) \to \Z_{m}\) of prime power order, the metabelian representation~\(\alpha(n,\chi)\) is \(\eta\)-regular.
  Moreover,
  \begin{enumerate}
  \item if \(w\) is divisible by \(n\), then there exists an isometry of linking forms
    \[\Bl_{\alpha(n,\chi)}(P(K,\eta)) \cong \Bl_{\alpha(n,\chi_P)}(P) \oplus \bigoplus_{i=1}^{n} \Bl(K)(\xi_{m}^{\chi_{i}(q_Q(\mu_Q^{-w}\eta))}t^{w/n});\]
  \item if \(w\) is not divisible by \(n\), then \(\Bl_{\alpha(n,\chi)}(P(K,\eta))\)  is Witt equivalent to
    \[\Bl_{\alpha(n,\chi_{P})} (P)\oplus \bigoplus_{i=1}^{h} \Bl_{\alpha(n/h,\chi_{i})}(K)(\xi_{m}^{\chi_{i}(q_Q(\mu_Q^{-w}\eta))}t^{w/h}).\]
  \end{enumerate}
\end{theorem}

Theorem~\ref{thm:metabelian-cabling-formula} takes a particularly simple form for connected sums.
In this case, we have~$w=1$ (so~$h=1$) as well as~$\eta=\mu_P$.
\begin{corollary}
  \label{cor:MetabelianConnectedSum}
  Let~$K, P$ be two knots. If~$\chi \colon H_1(\Sigma_n(K \# P);\Z) \to \Z_{m}$ is a character of prime power order, then  \(\Bl_{\alpha(n,\chi)}(K \# P)\) is Witt equivalent to~$\Bl_{\alpha(n,\chi_P)}(P) \oplus \Bl_{\alpha(n,\chi_K)}(K)$.
\end{corollary}

\section{Identities of presentations and \(3\)-dimensional CW-complexes}
\label{sec:Identities}


Given a $3$-manifold $M$, the goal of the next two sections is to describe how Fox derivatives can be used to calculate the differentials in the handle chain complex $C_*^{\text{hnd}}(M;\Z[\pi_1(M)])$ as well as the symmetric structure it supports.
The main difficulty in these calculations lies in understanding the third differential and the symmetric structure.

This section focuses on the third differential.
In order to explain the procedure needed to calculate it, we will need some facts about identities of group presentations.
After recalling some terminology on crossed modules in Section~\ref{sub:CrossedModuless}, Section~\ref{sub:identities-of-presentation} focuses on identities and, given a 3-dimensional CW complex $Y$, Section~\ref{sub:Identities3dCW} builds on these notions to describe the (chain homotopy type of the) cellular chain complex of $\widetilde{Y}$.

With some effort, the main result of this section, namely Proposition~\ref{prop:alg-data-to-3-dim-cw-complexes} can be deduced from results in Part 1 of the monograph~\cite[Part~1]{Brown-Higgins-Sivera}, or by combining the work of Whitehead~\cite{whiteheadElementsHomotopyTheory1978} and Trotter~\cite{TrotterSystem}.
Since these results do not seem to be frequently used by the low dimensional topology community, we include both recollections and detailed proofs.

\begin{remark}
\label{rem:MillerPowellDetails}
In~\cite[Section 3.1]{MillerPowell}, Miller and Powell explain how a Wirtinger presentation of~$\pi_1(X_K)$ gives rise to a handle decomposition of $M_K$,  as well as how to calculate the $\Z[\pi_1(M_K)]$-handle chain complex of $M_K$ and its symmetric structure.
The work we carry out in this section and the next stems both from the fact that we plan to use a handle decomposition that does not arise from a Wirtinger presentation and because we felt the need to supplement more details to the paragraph in~\cite[proof Theorem~3.9]{MillerPowell} that begins with ``Note that there is a correspondence between 3-cells and identities of a presentation"; compare that paragraph with (the proofs of) Proposition~\ref{prop:alg-data-to-3-dim-cw-complexes} and Corollary~\ref{corollary:3-dim-cw-complex-to-alg-data}.
\end{remark}

\subsection{Crossed modules}
\label{sub:CrossedModuless}

In this section we give a brief summary of the theory of crossed modules.
For more details refer to the Part~1 of~\cite{Brown-Higgins-Sivera}.
The reason for considering crossed modules is that for $2$-complexes, they provide a convenient formalism to relate identities of presentations (which we need to calculate the third differential and the symmetric structure of our chain complex) and~$\pi_2$; see Proposition~\ref{prop:Peifferpi2}.

\begin{definition}
\label{def:CrossedModule}
~
\begin{itemize}
\item Given a group $\pi$,  a \emph{crossed $\pi$-module} is a pair $(G,\partial)$ consisting of a group $G$ upon which $\pi$ acts from the left 
and a group homomorphism $\partial  \colon G \to \pi$ such that 
$$\partial(\gamma \cdot g)=\gamma \partial(g)\gamma^{-1}$$
for every $\gamma \in \pi$ and every $g \in G$.
We often refer to $(G,\partial)$ as a \emph{crossed $\pi$-module}.

\item Suppose we are given a group~$\pi$, a set~\(S\), and a map \(m \colon S \to \pi\).
Let \(H\) be the free group on the set \(S \times \pi\) and extend the map~\(m\) to a map
\[\partial_{m} \colon H \to \pi, \quad \partial_{m}(x,\gamma) = \gamma m(x) \gamma^{-1}.\]
For two elements \(a,b \in H\), we define their \emph{Peiffer commutator}
\[[[a,b]]_{P} = a b a^{-1} (\partial_{m}(a) \cdot b)^{-1}.\]
We denote by~\([[H,H]]_{m}\) the subgroup of~\(H\) generated by Peiffer commutators.

\item Given a group~$\pi$, a set~\(S\), and a map \(m \colon S \to \pi\), we construct the \emph{free crossed \(\pi\)-module} \((FC_{\pi}(S,m),\partial)\) generated by the pairs~$(S,m)$.
As a group, \(FC_{\pi}(S,m)\) is the quotient
\[FC_{\pi}(S,m) = H / [[H,H]]_{m},\]
where, as in the previous item, \(H\) is the free group on the set~\(S \times \pi\).
In other words, for any~\(x,y \in S\) and~\(\gamma_{1},\gamma_{2} \in \pi\), the following relation holds in~\(FC_{\pi}(S,m)\):
\[(y,\gamma_1) (x,\gamma_2) (y,\gamma_1)^{-1}\sim (x,\gamma_{1} m(y) \gamma_{1}^{-1} \gamma_{2}).\]
The action of~\(\pi\) on \(FC_{\pi}(S,m)\) comes from the natural left action of~\(\pi\) on itself, i.e.
\(\gamma_{1} \cdot (x,\gamma_{2}) = (x,\gamma_{1} \gamma_{2})\).
Furthermore, since the map \(\partial_{m} \colon H \to \pi\) vanishes on~\([[H,H]]_{m}\), it descends to a map defined on~\(FC_{\pi}(S,m)\), and we define \(\partial := \partial_{m}\).
\end{itemize}
\end{definition}

Let us note the following facts whose proofs are left to the reader.

\begin{lemma}\label{lemma:crossed-modules-elementary-facts}
    Let \((G,\partial)\) be a crossed \(\pi\)-module.
    \begin{enumerate}
        \item The image \(\partial(G) \subset \pi\) is a normal subgroup.
        \item The subgroup~\(\ker(\partial) \subset G\) is abelian.
        \item The action of \(\pi\) on \(G\) descends to an action of \(\pi/\partial(G)\) on~\(\ker(\partial)\).
    \end{enumerate}
\end{lemma}

\begin{example}\label{example:free-crossed-module}
Let \(\mathcal{P} = \langle \mathbf{x} \mid  \mathbf{r} \rangle\) be a finite presentation of a group \(\pi\), let~\(X\) denote the \(2\)-dimensional presentation CW-complex of~\(\mathcal{P}\) and denote by \(X^{1}\) the \(1\)-skeleton of~\(X\).
  It is known (see e.g.~\cite{Whitehead}) that the pair~\( (\pi_{2}(X,X^{1}),\partial)\) is a crossed \(\pi_{1}(X^{1})\)-module, where 
  \[\partial \colon \pi_{2}(X,X^{1}) \to \pi_{1}(X^{1}),\]
is the connecting homomorphism from the long exact sequence of the pair \((X,X^{1})\).

Whitehead~\cite[Section~16]{Whitehead} proved that~$(\pi_{2}(X,X^{1}),\partial)$ is isomorphic to the free crossed~\(\pi_{1}(X^{1})\)-module generated by the set \(R= \{f^{2}_{r} \colon r \in
\mathbf{r}\}\) of characteristic maps \(f^{2}_{r} \colon (D^{2},\partial D^{2}) \to (X,X^{1})\) of the \(2\)-cells  \(\{e^{2}_{r} \colon r \in \mathbf{r}\}\) of \(X\).
Furthermore, the map \(  m  \colon R \to \pi_{1}(X^{1})\) is given by the formula~\(m(f^{2}_{r}) = r\).
\end{example}

\subsection{Identities of presentations}
\label{sub:identities-of-presentation}
This section is concerned with identities of group presentations.
As will become apparent in the next section, this is the algebra that underlies the calculation of the third differential in the chain complex of the universal cover of a $3$-dimensional CW complex.
\medbreak

  Let~\(\mathcal{P} = \langle \mathbf{x} \mid  \mathbf{r} \rangle\)  be a presentation of a group~$G$, let \(F\) be the free group generated by~\(\mathbf{x}\), and let \(P=\langle \rho_{r} \colon r \in \mathbf{r} \rangle\)
  be the free group generated by symbols~\(\rho_{r}\), for~\(r \in \mathbf{r}\).
  Moreover, following Trotter~\cite[Section 2.1]{TrotterSystem}, consider the homomorphism 
  $$\psi \colon F \ast P \to F$$
  defined on generators by~$\psi(x)=x$, for~\(x \in \mathbf{x}\), and $\psi(\rho_{r}) = r$ for~\(r \in \mathbf{r}\).

\begin{definition}\label{definition:identity-of-presentation}
    Let~\(\mathcal{P} = \langle \mathbf{x} \mid  \mathbf{r} \rangle\)  be a presentation of a group~$G$.
    Denote by~\(N(\mathcal{P})\) the normal subgroup of~\(F \ast P\) generated by~\(P\):
    $$ N(\mathcal{P}):=\langle \langle P \rangle \rangle \trianglelefteq F \ast P.$$
   An~\emph{identity of the presentation $\mathcal{P}$} is an element of \(\ker(\psi) \cap N(\mathcal{P}) \).
   We write~\(I(\mathcal{P})\) for the group of identities of~\(\mathcal{P}\) :
   $$ I(\mathcal{P}):=\ker(\psi) \cap N(\mathcal{P}).$$
   More explicitly,  an identity is an element of~\(\ker(\psi)\), which can be written as a product of words of the form \(w \rho^{\epsilon}_{r} w^{-1}\), where~$w$ lies in~$F$, \(\epsilon =
   \pm 1\), and \(r \in \mathbf{r}\).
\end{definition}

\begin{construction}
\label{cons:ChainComplex}
Given a group $G$, following Trotter~\cite[page 473]{TrotterSystem},  we outline the definition of  a 3-dimensional chain complex of free left $\Z[G]$-modules
$$C_\bullet(\mathbf{x},\mathbf{r},\mathbf{s})=\left( C_3 \xrightarrow{\partial_3} C_2 \xrightarrow{\partial_2}   C_1  \xrightarrow{\partial_1}  C_0  \right)$$
associated to a presentation $\mathcal{P}=\langle \mathbf{x} \mid \mathbf{r}\rangle$ of $G$ and a set $\mathbf{s}$ of identities of $\mathcal{P}$.
The chain complex~\(C_\bullet(\mathbf{x},\mathbf{r},\mathbf{s})\) satisfies \(H_{0}(C_\bullet(\mathbf{x},\mathbf{r},\mathbf{s}))
\cong \Z\) (where $\Z$ is endowed the $\Z[G]$-module structure induced by augmentation),  \(H_{1}(C_\bullet(\mathbf{x},\mathbf{r},\mathbf{s})) = 0\).

The chain module $C_0$ is free of rank one on an element $v$, $C_1$ is free on elements \(\{a_{x} \colon x \in \mathbf{x}\}\), $C_2$ is free on elements~\(\{b_{r} \colon r \in \mathbf{r}\}\), and $C_3$ is free
on the set~\(\{c_{s} \colon s \in \mathbf{s}\}\).
The differentials are defined in terms of Fox derivatives
\begin{align*}
  \partial_{1}(a_{x}) &= (x-1)v, \quad \text{\ for\ } x \in \mathbf{x}, \\
  \partial_{2}(b_{r}) &= \sum_{x \in \mathbf{x}} \frac{\partial r}{\partial x} a_{x}, \quad \text{\ for\ } r \in \mathbf{r}, \\
  \partial_{3}(c_{s}) &= \sum_{r \in \mathbf{r}} \frac{\partial s}{\partial \rho_{r}} b_{r}, \quad \text{\ for\ } s \in \mathbf{s},
\end{align*}
where the Fox~derivatives~\(\frac{\partial s}{\partial \rho_{r}}\) are computed in~\(F \ast P\).
The fact that  \(H_{0}(C_\bullet(\mathbf{x},\mathbf{r},\mathbf{s}))
\cong \Z\) and~\(H_{1}(C_\bullet(\mathbf{x},\mathbf{r},\mathbf{s})) = 0\) follows because Fox derivatives calculate the differentials in the cellular chain complex of the universal cover.
\end{construction}

\begin{definition}
    Let~\(\mathcal{P} = \langle \mathbf{x} \mid  \mathbf{r} \rangle\)  be a presentation of a group~$G$.
    We say that a set of identities~$\mathbf{s}$ is \emph{complete} if~\(H_{2}(C_\bullet(\mathbf{x},\mathbf{r},\mathbf{s})) = 0\).
\end{definition}

When~$\mathbf{s}$ is a complete set of identites for a presentation $\mathcal{P}$, we can think of~\(C_\bullet(\mathbf{x},\mathbf{r},\mathbf{s})\) as a truncation of a free resolution of the trivial \(\Z[G]\)-module \(\Z\).



\subsection{Identities and 3-dimensional CW complexes}
\label{sub:Identities3dCW}
We use identities to describe the chain complex of the universal cover of a $3$-dimensional CW complex.
\medbreak

Let \(\mathcal{P} = \langle \mathbf{x} \mid \mathbf{r} \rangle\) be a presentation of a group~\(G\).
As in the previous section, we write~\(F\) for the free group on the set~\(\mathbf{x}\), \(P\) for the free group generated by symbols~\(\rho_{r}\) with~\(r \in \mathbf{r}\), and
$$ \psi \colon F \ast P \to F $$
for the homomorphism  defined on generators by~$\psi(x)=x$, for~\(x \in \mathbf{x}\), and $\psi(\rho_{r}) = r$ for~\(r \in \mathbf{r}\).
We also recall that
$$ N(\mathcal{P}):=\langle \langle P \rangle \rangle \trianglelefteq F \ast P.$$
Observe that \(F\) acts on \(N(\mathcal{P})\) by conjugation.

Let~\(X\) denote a \(2\)-dimensional presentation CW-complex of~\(\mathcal{P}\).
In particular,~\(X\) is built of \(1\)-cells \(e^{1}_{x}\), for \(x \in \mathbf{x}\), and \(2\)-cells \(e^{2}_{r}\), for \(r \in \mathbf{r}\).

\begin{construction}
There is a group homomorphism~\(\kappa_{N} \colon N(\mathcal{P}) \to \pi_{2}(X,X^{1})\) given by the formula
\[\kappa_{N}(w \rho_{r}^{\epsilon} w^{-1}) = (w \cdot f^{2}_{r})^{\epsilon},\]
where \(f^{2}_{r} \colon (D^{2},\partial D^{2}) \to (X,X^{1})\) is the characteristic maps of the \(2\)-cell~\(e^{2}_{r}\).

This uniquely defines $\kappa_{N}$ because, as we now assert,~\(N(\mathcal{P})\) is freely generated by words of the form~\(w \rho_{r} w^{-1}\), for \(r \in \mathbf{r}\) and~\(w \in F\).
Consider the canonical projection \(p \colon F \ast P \to F\) with \(\ker(p) = N(\mathcal{P})\) and \((F \ast P) / N(\mathcal{P}) \cong F\).
      Therefore, \(N(\mathcal{P}) = \pi_{1}(Y)\), where \(Y \to K(F \ast P,1)\) is the covering associated to the projection~\(p\).
      The space~\(Y\) can be constructed as a pullback of the universal cover~\(\widetilde{K(F,1)} \to K(F,1)\) along the map \(K(F \ast P,1) \to K(F,1)\) induced by the projection~\(p\). 
Both~\(K(F,1)\) and~\(K(F \ast P,1)\) are bouquets of circles and we write $(W_f)_{f \in F}$ for the lifts of the single \(0\)-cell of~\(K(F,1)\). 
The aforementioned pullback of~\(\widetilde{K(F,1)} \to K(F,1)\) can be build by attaching a copy of~\(K(P,1)\) at every \(0\)-cell of $\widetilde{K(F,1)}$, i.e.\ at every $w_f$.
It follows that~\(Y\) is homotopy equivalent to a bouquet of~\(K(P_{w},1)\), where \(P_{w} = w P w^{-1} \subset F \ast P\) and~\(w \in F\), i.e., \(Y \simeq \bigvee_{w \in F} K(P_{w},1)\).
      Hence,  $N(\mathcal{P})$ is freely generated by words of the form~\(w \rho_{r} w^{-1}\), asserted and~\(\kappa_N\) is defined.
\end{construction}

The following proposition, though not explicitly stated, is a consequence of the results of~Section~16 of~\cite{Whitehead} and of~Section~2.3 in~\cite{TrotterSystem}.
  \begin{proposition}
      \label{prop:Peifferpi2}
      Let \(X\) be a \(2\)-dimensional CW-complex realising a presentation~\(\mathcal{P}\) of a group~\(G\).
The map~$\kappa_{N} \colon  N(\mathcal{P}) \to \pi_{2}(X,X^{1})$ is surjective and descends to an isomorphism of crossed \(F\)-modules
      \[\kappa_{N} \colon \left(N(\mathcal{P}) / [[N(\mathcal{P}),N(\mathcal{P})]]_{\psi},\psi\right) \xrightarrow{\cong} \left( \pi_{2}(X,X^{1}), \partial \right).\]
      The restriction of \(\kappa_{N}\) to~\(I(\mathcal{P})\) induces an isomorphism of left \(\Z[G]\)-modules
      $$\kappa_{I} \colon  I(\mathcal{P})/[[N(\mathcal{P}),N(\mathcal{P})]]_{\psi} \to \pi_2(X),$$
      where, by abuse of notation, we denote by~\(\psi\) the restriction of the map~\(\psi \colon F \ast P \to F\) to~\(N(\mathcal{P})\).
  \end{proposition}
  \begin{proof}
      First, we argue that~\(\kappa_{N}\) is surjective and that~\(\ker( \kappa_{N}) = [[N(\mathcal{P}),N(\mathcal{P})]]_{\psi}\).
      As in Definition~\ref{def:CrossedModule}, we denote by~\(H\) the free group generated by the set~\(S \times F\), where \(S = \{e^{2}_{r} \colon r \in \mathbf{r}\}\) is the collection of \(2\)-cells of~\(X\).
      It follows that the map 
      \[\theta \colon N(\mathcal{P}) \owns w \rho_{r} w^{-1} \mapsto (f^{2}_{r},w) \in H\]
      defines a group isomorphism.
      Indeed, this follows from the fact that \(N(\mathcal{P})\) is freely generated by words of the form~\(w \rho_{r} w^{-1}\), where \(w \in F\).
  
This group homomorphism fits into the commutative diagram
      \begin{equation}
          \label{eq:commutative-diagram-1}
          \begin{tikzcd}
              N(\mathcal{P}) \arrow[rr,"{\theta,\cong}"] \arrow[dr,"\kappa_{N}"] & & H \arrow[dl,"q"] \\
              & \pi_{2}(X,X^{1}), & 
          \end{tikzcd}
      \end{equation}  
where \(q \colon H \to \pi_2(X,X^{1})\) is given on generators by the formula \(q(f_{r},w) = w \cdot f_{r}\).      
Set $F:=\pi_1(X^1)$.
As noted in Example~\ref{example:free-crossed-module}, \(\pi_2(X,X^{1})\) is a free \(F\)-crossed module, meaning that the homomorphism~$q$ is surjective with kernel \(\ker(q) = [[H,H]]_{\partial_{H}}\), where \(\partial_{H}\) denotes the map
      \[\partial_{H} \colon S \times F \to F=\pi_{1}(X^{1}), \quad \partial_{H}(f^{2}_{r},w) = w r w^{-1}.\]
      Since $\theta$ and $q$ are both surjective,  the commutativity of~\eqref{eq:commutative-diagram-1} implies that so is~\(\kappa_{N}\).
      
%

We now show that~\(\ker( \kappa_{N}) = [[N(\mathcal{P}),N(\mathcal{P})]]_{\psi}\).
Since $\ker(q)=[[H,H]]_{\partial_{H}}$,  the commutativity of~\eqref{eq:commutative-diagram-1} implies that
      \[\ker(\kappa_{N}) = \theta^{-1}(\ker(q)) = \theta^{-1}([[H,H]]_{\partial_{H}}).\]
      In order to understand the subgroup~\(\theta^{-1}([[H,H]]_{\partial_{H}})\), consider the diagram
      \begin{center}
          \begin{equation}
              \label{eq:commutative-diagram-2}
              \begin{tikzcd}
                  N(\mathcal{P}) \arrow[r,"\theta"] \arrow[d,"\psi"] & H \arrow[d,"\partial_{H}"] \\
                  F \arrow[r,"="] & \pi_{1}(X^{1}).
              \end{tikzcd}
          \end{equation}
      \end{center}
whose commutativity implies the required equality:     
      \[\ker(\kappa_{N}) = \theta^{-1}(\ker(q)) = \theta^{-1}([[H,H]]_{\partial_{H}}) = [[N(\mathcal{P}),N(\mathcal{P})]]_{\psi}.\]
      In particular, \(\kappa_{N}\) induces an isomorphism
      \[N(\mathcal{P}) / [[N(\mathcal{P}),N(\mathcal{P})]]_{\psi} \cong \pi_{2}(X,X^{1}).\]
      It remains to show that the restriction of~\(\kappa_{N}\) to~\(I(\mathcal{P})\) gives rise to the claimed isomorphism.     
      Consider the following diagram with exact rows:
      \begin{center}
      \begin{equation}
      \label{eq:CommutativeForExplicitKappa}
          \begin{tikzcd}
              1 \arrow[r] & I(\mathcal{P}) \arrow[r] \arrow[d,"\kappa_{I}"] & N(\mathcal{P}) \arrow[r,"\psi"] \arrow[d,"\kappa_{N}"] & F \arrow[d,"\id"] \\
              1 \arrow[r] & \pi_2(X) \arrow[r] & \pi_{2}(X,X^{1}) \arrow[r, "\partial"] & \pi_{1}(X^{1}).
          \end{tikzcd}
          \end{equation}
      \end{center}
      The right-hand side square of~\eqref{eq:CommutativeForExplicitKappa} is commutative because for any~\(w \rho_{r}^{\epsilon} w^{-1} \in N(\mathcal{P})\), we have
      \[(\partial \circ \kappa_{N})(w \rho_{r}^{\epsilon} w^{-1}) = \partial((f^{2}_{r},w)^{\epsilon}) = w r^{\epsilon} w^{-1} = \psi(w \rho_{r}^{\epsilon} w^{-1}).\]
The commutativity of this diagram and the exactness of its rows gives the required isomorphism:
      \[\pi_{2}(X) = \ker(\partial) \cong \ker(\psi) / \ker (\kappa_{N}) = I(\mathcal{P}) /
          [[N(\mathcal{P}),N(\mathcal{P})]]_{\psi}.\]
      This concludes the proof of the proposition.
  \end{proof}

Roughly speaking, the next proposition describes how a set of identities $\mathbf{s}$ for a group presentation $\mathcal{P}$ determines a $3$-dimensional CW-complex that realises the data of $\mathcal{P}$ and $\mathbf{s}$.

  \begin{proposition}\label{prop:alg-data-to-3-dim-cw-complexes}
 Let \(\mathcal{P} = \langle \mathbf{x} \mid \mathbf{r} \rangle\) be a presentation of a group~\(G\) and let \(\mathbf{s}\) be a collection of identities of~\(\mathcal{P}\).
  This data determines a \(3\)-dimensional CW-complex \(Z := Z(\mathbf{x},\mathbf{r},\mathbf{s})\) such that:
      \begin{enumerate}
      \item \(Z\) is connected and~\(\pi_{1}(Z) \cong G\),
      \item \(Z^{2}\) realizes the presentation~\(\mathcal{P}\),
      \item the \(3\)-cells of~\(Z\) are in bijection with~\(\mathbf{s}\),
      \item  \(\pi_{2}(Z) \cong \pi_{2}(Z^{2}) / S\), where~\(S\) denotes the left \(\Z[G]\)-submodule of \(\pi_{2}(Z^{2}) \cong I(\mathcal{P}) / [[N(\mathcal{P}),N(\mathcal{P})]]_{\psi}\) generated by elements of~\(\mathbf{s}\).
      \item there is an identification \(C^{cell}_{\bullet}(\widetilde{Z}) \cong C_\bullet(\mathbf{x},\mathbf{r},\mathbf{s})\) of $\Z[G]$-chain isomorphism where~\(C_{\bullet}(\mathbf{x},\mathbf{r},\mathbf{s})\) is the chain complex from Construction~\ref{cons:ChainComplex}.
      \end{enumerate}
      In particular, if \(\mathbf{s}\) is a complete set of identities, then \(Z\) is the \(3\)-skeleton of a model for the classifying space~\(BG\).
  \end{proposition}
  \begin{proof}
  Let \(X\) be a presentation \(2\)-complex of~\(\mathcal{P}\).
      By Proposition~\ref{prop:Peifferpi2}, the set~\(\mathbf{s}\) determines a collection of elements \(\{[g_{s}]\}_{s \in \mathbf{s}}\) of \(\pi_{2}(X)\).
      Construct~\(Z\) by adjoining \(3\)-cells \(\{e^{3}_{s}\}_{s \in \mathbf{s}}\), where~\(e^{3}_{s}\) is attached using~\(g_{s} \in \pi_{2}(X)\).
      For later use, we denote the characteristic map of the \(3\)-cell~\(e^{3}_{s}\) by
$$ f^{3}_{s} \colon  (D^{3},\partial D^{3}) \to (Z,Z^{2}).$$ 
      The first three points of the proposition now follow immediately from the construction of~\(Z\).
      We prove the fourth point.
      Since~\(Z\) is a \(3\)-complex, the relative homotopy group \(\pi_{3}(Z,Z^{2})\) is a free \(\Z[G]\)-module generated by homotopy classes of characteristic maps \(\{f^{3}_{s}\}_{s \in \mathbf{s}}\)~\cite[Chapter
      V.1, Theorem~1.1]{whiteheadElementsHomotopyTheory1978}.
      Using the exact sequence
      \[\pi_{3}(Z,Z^{2}) \xrightarrow{\partial} \pi_{2}(Z^{2}) \to \pi_{2}(Z) \to 0\]
      we conclude that~\(\pi_{2}(Z) \cong \pi_{2}(Z^{2}) / \operatorname{im} \partial\).
      Observe that the boundary map in the exact sequence maps the homotopy class of the characteristic map~\([f_{s}]\) to the homotopy class of the corresponding attaching map~\([g_{s}]\), for \(s \in \mathbf{s}\).
      Observe that for any~\(s \in \mathbf{s}\), \(\kappa_{I}([g_{s}]) = s\), where \(\kappa_{I}\) is the isomorphism from Proposition~\ref{prop:Peifferpi2}.
      Therefore, \(\kappa_{I}^{-1}\) maps~\(\operatorname{im} \partial\) isomorphically onto the~\(\Z[G]\)-submodule of~\(I(\mathcal{P}) / [[N(\mathcal{P}),N(\mathcal{P})]]_{\psi}\) generated by identities \(s \in \mathbf{s}\), as desired.
This concludes the proof of the fourth point.


We prove the fifth and the last point.
For each $k$, the chain groups underlying the chain complex~$C^{\text{cell}}_\bullet(\widetilde{Z})$ are~\(C^{\text{cell}}_{k}(\widetilde{Z}) = H_{k}(\widetilde{Z}^{k},\widetilde{Z}^{k-1})\) and the boundary maps are given by the composition
      \[\partial^{\text{cell}}_{k+1} \colon  C_{k+1}^{\text{cell}}(\widetilde{Z}) = H_{k+1}(\widetilde{Z}^{k+1},\widetilde{Z}^{k}) \xrightarrow{\partial} H_{k}(\widetilde{Z}^{k}) \xrightarrow{j}
          H_{k}(\widetilde{Z}^{k},\widetilde{Z}^{k-1}) = C_{k}^{\text{cell}}(\widetilde{Z}),\]
      where the maps~\(\partial\) and~\(j\) come from the exact sequence of the pairs~\((\widetilde{Z}^{k+1},\widetilde{Z}^{k})\) and~\((\widetilde{Z}^{k},\widetilde{Z}^{k-1})\), respectively.
  The expressions for~\(\partial_{1}^{\text{cell}}\) and~\(\partial_{2}^{\text{cell}}\) are well-known (see e.g.~\cite[page~473]{TrotterSystem}) and agree with the formulas given in Construction~\ref{cons:ChainComplex}.
  Therefore, we focus only on identifying the differential~\(\partial_{3}^{\text{cell}}\). 
      
  Consider the following commutative diagram
      \begin{center}
          \begin{tikzcd}
              & I(\mathcal{P}) \arrow[r,"j"] \arrow[d,"\kappa_{I}"] & N(\mathcal{P}) \arrow[d,"\kappa_{N}"] \\
              \pi_{3}(Z,Z^{2}) \arrow[r,"\partial"] & \pi_{2}(Z^{2}) \arrow[r,"j"] & \pi_{2}(Z^{2},Z^{1}) \\
              \pi_{3}(\widetilde{Z},\widetilde{Z}^{2}) \arrow[r,"\partial"] \arrow[d,"h_{3}","\cong"'] \arrow[u,"q_{\ast}"',"\cong"] & \pi_{2}(\widetilde{Z}^{2}) \arrow[r,"j"] \arrow[d,"h_{2}","\cong"'] \arrow[u,"q_{\ast}"',"\cong"] & \pi_{2}(\widetilde{Z}^{2},\widetilde{Z}^{1})
              \arrow[d,"h_{2,\text{rel}}"] \arrow[u,"q_{\ast}"] \\
              H_{3}(\widetilde{Z},\widetilde{Z}^{2}) \arrow[r,"\partial"] & H_{2}(\widetilde{Z}^{2}) \arrow[r,"j"] & H_{2}(\widetilde{Z}^{2},\widetilde{Z}^{1}),
          \end{tikzcd}
      \end{center}
      where  the vertical map~\(\kappa\) comes from Proposition~\ref{prop:Peifferpi2},
       the vertical maps between the third and the second row are induced by the covering map \(q\colon \widetilde{Z} \to Z\), and the vertical maps between the third and fourth row come from the (relative)~Hurewicz~Theorem.
      In particular, (relative)~Hurewicz~Theorem implies that~\(h_{3}\) and~\(h_{2}\) are isomorphisms and that~\(h_{2,\text{rel}}\) induces an isomorphism
      \[\pi_{2}(\widetilde{Z}^{2},\widetilde{Z}^{1})^{ab} \cong
      H_{2}(\widetilde{Z}^{2},\widetilde{Z}^{1}).\]
Note that the map $\partial_3^{\text{cell}} \colon H_3(\widetilde{Z},\widetilde{Z}^{2}) \to H_2(\widetilde{Z}^{2},\widetilde{Z}^{1})$ appears as the composition of the maps on the last row of this diagram.
Since~\(H_{3}(\widetilde{Z},\widetilde{Z}^{2}) \cong \pi_{3}(\widetilde{Z},\widetilde{Z}^{2}) \cong \pi_{3}(Z,Z^{2})\) is freely generated as a~\(\Z[G]\)-module by the characteristic maps~\(f^{3}_{s} \colon (D^{3},\partial D^{3}) \to (Z,Z^{2})\),
it remains to understand
$$\partial_3^{\text{cell}}([f^{3}_{s}])=j \circ \partial ([f^{3}_{s}]).$$
      For~\([f^{3}_{s}] \in H_{3}(\widetilde{Z},\widetilde{Z}^{2})\), we have~\(\partial([f^{3}_{s}]) = [f^{3}_{s}|_{\partial D^{2}}] = [g_{s}]\),  where $g_s \in \pi_2(Z^2)=\pi_2(X)$ are the maps along which we glued the $3$-cells.
It follows that \(s \in \kappa^{-1}( q_{\ast} (h_{2}^{-1}(\{[g_{s}]\})))\) and therefore the commutativity of the diagram gives
      \[\partial_{3}^{\text{cell}}([f^{3}_{s}]) = (j \circ \partial) ([f^{3}_{s}]) = j([g_{s}]) = h_{2,\text{rel}}(q_{\ast}^{-1}(\kappa(j(s)))).\]
      We calculate this last expression.
      Observe that for any~\(x = w \rho_{r_{0}} w^{-1} \in N(\mathcal{P})\) with~\(r_{0} \in \mathbf{r}\) we have
      \[h_{2,\text{rel}}(q_{\ast}^{-1}(\kappa(x))) = w \cdot f^{2}_{r_{0}} = \frac{\partial x}{\partial \rho_{r_{0}}} \cdot f^{2}_{r_{0}} = \sum_{r \in \mathbf{r}} \frac{\partial x}{\partial \rho_{r}} \cdot f^{2}_{r}.\]
Here the Fox~derivatives \(\frac{\partial x}{\partial \rho_{r}}\) are calculated in~\(F \ast P\) and we used that~\( \frac{\partial x}{\partial \rho_{r_{0}}}=w\) and~\( \frac{\partial x}{\partial \rho_{r}}=0\) for~$r \neq r_0$.
      
One then verifies that for any~\(y \in N(\mathcal{P})\), we have
      \[h_{2,\text{rel}}(q_{\ast}^{-1}(\kappa(y))) = \sum_{r \in \mathbf{r}} \frac{\partial y}{\partial \rho_{r}} \cdot f_{r}^{2}.\]
      Plugging this into the previous calculation we obtain 
      \[\partial_{3}^{\text{cell}}([f^{3}_{s}]) = h_{2,\text{rel}}(q_{\ast}^{-1}(\kappa(j(s)))) = \sum_{r \in \mathbf{r}} \frac{\partial s}{\partial \rho_{r}} \cdot [f^{2}_{r}],\]
This concludes the proof of the proposition.
  \end{proof}

  \begin{corollary}\label{corollary:3-dim-cw-complex-to-alg-data}
Any $3$-complex~\(Y\) determines a presentation \(\mathcal{P} = \langle \mathbf{x} \mid \mathbf{r} \rangle \) of~$\pi_1(Y)$
and a collection of identities~\(\mathbf{s}\) of~\(\mathcal{P}\), that is well-defined modulo Peiffer commutators.

      Furthermore,  $Y$ is homotopy equivalent to the $3$-complex~\(Z(\mathbf{x},\mathbf{r},\mathbf{s})\) from Proposition~\ref{prop:alg-data-to-3-dim-cw-complexes}.
  \end{corollary}
  \begin{proof}
      The \(2\)-skeleton of~\(Y\)  determines a presentation \(\mathcal{P} = \langle \mathbf{x} \mid \mathbf{r} \rangle\) of~\(\pi_{1}(Y)\).
      Using the isomorphism~\(\pi_{2}(Y^{2}) \cong I(\mathcal{P}) / [[N(\mathcal{P}),N(\mathcal{P})]]_{\psi}\) from Proposition~\ref{prop:Peifferpi2}, we see that the attaching maps of the \(3\)-cells of~\(Y\) determine identities \(s \in I(\mathcal{P})\), which are well-defined modulo Peiffer commutators.

      The homotopy equivalence~\(Y \simeq Z(\mathbf{x},\mathbf{r},\mathbf{s})\) follows from the fact that there is a bijection between cells of~\(Y\) and~\(Z(\mathbf{x},\mathbf{r},\mathbf{s})\) in each
      dimension, and the corresponding cells are attached via homotopic maps.
  \end{proof}

  \section{Symmetric structures of aspherical 3-manifolds and identities of presentation}
  \label{sec:symmetric-structure}

Given a $3$-manifold $M$, this section builds on Section~\ref{sec:Identities} to describe the symmetric structure on the handle chain complex $C_*^{\text{hnd}}(M;\Z[\pi_1(M)])$.
Section~\ref{sub:handle-chain-complex} recalls some facts about the handle chain complex.
Section~\ref{sub:ashp-3-manif} records a technical result concerning identities.
Section~\ref{sub:expl-form-symm} recalls work of Trotter according to which the symmetric structure can be calculated using identities and Fox calculus.
When $M$ admits a handle decomposition with a single $3$-handle, the results of this section and Section~\ref{sec:Identities} is summarised in Proposition~\ref{prop:Summary}.

  \subsection{The handle chain complex}
  \label{sub:handle-chain-complex}

  Let~\(M\) be a connected \(n\)-manifold that admits a handle decomposition.
  Assume that the handles are attached in increasing order of index.
  In what follows we denote the $n$-handles of $M$ by $h_i^n$ and, for $d \geq 0$, write $M^{(d)}$ for the submanifold of~\(M\) obtained by taking the union of all handles of indices~\(\leq d\).
  
  For $d \geq 0$, the associated handle chain complex has chain groups
   \[C_{d}^{\text{hnd}}(M) = H_{d}(M^{(d)},M^{(d-1)})\]
   and differentials 
     \[\partial^{\text{hnd}}_{d+1} \colon C_{d+1}^{\text{hnd}}(M) = H_{d+1}(M^{(d+1)},M^{(d)}) \xrightarrow{\partial} H_{d}(M^{(d)}) \xrightarrow{j} H_{d}(M^{(d)},M^{(d-1)}) = C_{d}^{\text{hnd}}(M),\]
  where $\partial$ is the connecting homomorphism in the long exact homology sequence of the pair~\((M^{(d+1)},M^{(d)})\) and $j$ is induced by inclusion of pairs \((M^{(d)},\emptyset) \subset
  (M^{(d)},M^{(d-1)})\).
  
  Let~\(p \colon \widetilde{M} \to M\) be the universal covering projection.
  The handle decomposition of $M$ lifts to a handle decomposition of $\widetilde{M}$ and we write~\(\widetilde{M}^{(d)} = p^{-1}(M^{(d)})\).
We can now consider the handle chain complex of $\widetilde{M}$:
  \[C_{\ast}^{\text{hnd}}(M;\Z[\pi_{1}(M)]) := C_{\ast}^{\text{hnd}}(\widetilde{M}).\]
  Using the action of~\(\pi_{1}(M)\) on~\(\widetilde{M}\),  each chain group~\(C^{\text{hnd}}_{d}(M;\Z[\pi_{1}(M)])\) becomes a free left~\(\Z[\pi_{1}(M)]\)-module.
These chain~\(\Z\pi_{1}(M)\)-modules are generated by a collection of lifts of the \(d\)-dimensional handles of~\(M\) where we
  choose one lift for each handle.
In what follows, we implicitly choose (arbitrarily) such lifts and use them as a basis for the relevant chain groups.

\begin{construction}
\label{cons:3Complex}
  By contracting each handle to its core,  one obtains a \(3\)-complex~\(X\) that embeds in~\(M\), \(\iota \colon X \hookrightarrow M\) and onto which $M$ deformation retracts via a map~\(r \colon M \to X\).
 We write~\(e_{i}^{j}\) for the \(j\)-dimensional cell of~\(X\) which forms the core of the handle~\(h_{i}^{j}\).
 \end{construction}

  \begin{lemma}\label{lemma:comparison-handle-cellular-chain-cplx}
      The inclusion map \(\iota \colon X \hookrightarrow M\) induces chain isomorphisms
      \[\iota_{\ast} \colon C^{\text{cell}}_{\bullet}(X) \cong C^{\text{hnd}}_{\bullet}(M), \quad \iota_{\ast} \colon C^{\text{cell}}_{\bullet}(X;\Z[\pi_{1}(X)]) \cong
          C^{\text{hnd}}_{\bullet}(M;\Z[\pi_{1}(M)]).\]
Every cell~\(e^{d}_{j}\) of \(X\) gives rise to a generator \([e^{d}_{j}]\) in~\(C^{\text{cell}}_{d}(M;\Z[\pi_{1}(M)])\) and every handle~\(h^{d}_{j}\) of~\(M\) gives rise to a generator~\(h^{d}_{j} \in C^{\text{hnd}}_{d}(M;\Z[\pi_{1}(M)])\).
      The map induced by~\(\iota\), maps each generator~\(e^{d}_{j}\) to the corresponding generator~\(h^{d}_{j}\).
  \end{lemma}
  \begin{proof}
      Since~\(\iota\) preserves the respective filtrations $\lbrace X^{(d)} \rbrace_d$ of~\(X\) and $\lbrace M^{(d)} \rbrace_d$ of~\(M\), it follows that~\(\iota\) induces isomorphisms of respective chain groups.
The naturality of the exact sequence of a pair in homology implies that~\(\iota\) is a chain map,  and hence the lemma follows.
  \end{proof}

\subsection{Aspherical \(3\)-manifolds and identities of presentation}
\label{sub:ashp-3-manif}

This short section record a technical lemma concerning identities.
A $1$- and $2$- handles in the handle decomposition $\mathcal{H}$ of a connected manifold $M$ determine a presentation $\mathcal{P}_{\mathcal{H}}$ of $\pi_1(M)$.
The next lemma notes that the $3$-handles of $M$ determine identities for this presentation.

  \begin{lemma}\label{lemma:generating-identities}
If~\(M\) is a closed, connected, oriented, aspherical \(3\)-manifold admitting a handle decomposition $\mathcal{H}$ with~\(k+1\) \(3\)-handles, for some~\(k \geq 0\), then the presentation~\(\mathcal{P}_{\mathcal{H}}\) of~\(\pi_{1}(M)\) admits a complete set of identities
$$s_{1},\ldots,s_{k+1}  \in I(\mathcal{P}_{\mathcal{H}})$$
 such that, modulo Peiffer commutators, any other identity can be written (uniquely up to order of factors) as a product of conjugates of the~\(s_i\).
  \end{lemma}
  \begin{proof}
      Since~\(M\) is aspherical, it follows that~\(H_i(\widetilde{M})=0\), for~\(i \geq 1\).
      The Hurewicz theorem and  the long exact sequence of the pair~\((\widetilde{M},\widetilde{M}^{(2)})\) gives rise to isomorphisms
$$\pi_2(M^{(2)}) \cong H_{2}(\widetilde{M}^{(2)}) \cong H_{3}(\widetilde{M},\widetilde{M}^{(2)}) \cong \Z[\pi_{1}(M)]^{k+1} $$
Let~\(f_{1},\ldots,f_{k+1} \in \pi_{2}(M^{(2)})\) be a~\(\Z[\pi_{1}(M)]\)-basis and consider the isomorphism
      \[\kappa_{I} \colon I(\mathcal{P_{\mathcal{H}}}) / [[N(\mathcal{P}_{\mathcal{H}}),N(\mathcal{P}_{\mathcal{H}})]]_{\psi} \cong \pi_{2}(M^{(2)})\]
      from Proposition~\ref{prop:Peifferpi2}.
A complete set of identities with the required properties is now given by considering~\(s_{i} = \kappa^{-1}_{I}(f_{i})\), for~\(i=1,2,\ldots,k+1\).
  \end{proof}
  
  Taking $k=0$ in the previous lemma gives the following result.

  \begin{proposition}\label{prop:generating-identities}
If~\(M\) is a closed, connected, oriented, aspherical \(3\)-manifold admitting a handle decomposition $\mathcal{H}$ with a single \(3\)-handle, 
then the presentation~\(\mathcal{P}_{\mathcal{H}}\) of~\(\pi_{1}(M)\) admits a unique (up to conjugation and modulo Pfeiffer commutators) identity 
$$ s  \in I(\mathcal{P}_{\mathcal{H}}).$$
%
  \end{proposition}

\subsection{Explicit formulas for the symmetric structure of a closed aspherical \(3\)-manifold}
\label{sub:expl-form-symm}

Given a $3$-manifold $M$, this section collects the work from Section~\ref{sec:Identities} to describe the differentials and symmetric structure on chain complex $C_*^{\text{hnd}}(M,\Z[\pi_1(M)]).$

\medbreak
Given a closed,  connected, oriented, aspherical \(3\)-manifold $M$,  with a handle decomposition $\mathcal{H}$,
we would like to obtain an explicit formula for the map
\[\Phi \colon C^{2}_{\text{hnd}}(M;\Z[\pi_{1}(M)]) \to C_{1}^{\text{hnd}}(M;\Z[\pi_{1}(M)])\]
induced by the symmetric structure on $C^{3-*}(M;\Z[\pi_{1}(M)])$.

For the $3$-complex $X$ obtained from $M$ by deformation retracting each handle to its core (as in Construction~\ref{cons:3Complex})
the most important ingredient that goes into the calculation of the map~\(\Phi\) (namely the diagonal chain map) was worked out by Trotter~\cite[Section~2.4]{TrotterSystem}, see also~\cite[Section~3.3]{MillerPowell}.
Building on this result we obtain the following proposition.

\begin{proposition}\label{prop:symmetric-structure}
    Let~\(M\) be a closed,  connected,  oriented,  aspherical \(3\)-manifold that admits a handle decomposition $\mathcal{H}$.
    Let
    \[\mathcal{P}_{\mathcal{H}} = \langle \mathbf{x} \mid \mathbf{r} \rangle\]
    denote the presentation of~\(\pi_{1}(M)\) determined by~\(\mathcal{H}\) and let \(\mathbf{s}\) denote the set of identities from Lemma~\ref{lemma:generating-identities}.
    Each \(s \in \mathbf{s}\) can be written in the form
    \begin{equation}
        \label{eq:expression-identity}
        s = \prod_{j=1}^{l_{s}} w_{s,j} \rho_{r_{s,j}}^{\epsilon_{s,j}} w_{s,j}^{-1},
    \end{equation}
    where~\(r_{s,j} \in \mathbf{r}\) and~\(\epsilon_{s,j} \in \{\pm 1\}\), for all~\(s\) and~\(j\).
 
    Endow  $C_{2}^{\text{hnd}}(M;\Z[\pi_{1}(M)])$ with a basis~\(h_{r}^{2}\), for \(r \in \mathbf{r}\), induced by the handle structure~$\mathcal{H}$.
    Denote by~\((h^{2}_{r})^{\#}\), for \(r \in \mathbf{r}\), the dual basis of~\(C^{2}_{\text{hnd}}(M;\Z[\pi_{1}(M)])\).
    With respect to these bases, the coefficients of the matrix of the map~\(\Phi\) is given by the formula
\[\Phi((h^{2}_{r})^{\#}) = \sum_{(s,j) \in \mathcal{I}(r)} \sum_{x \in \mathbf{x}} \epsilon_{s,j} w_{s,j}^{-1} \frac{\partial w_{s,j}}{\partial x} h^{1}_{x},\]
where~\(\mathcal{I}(r) = \{(s,j) \in \mathbf{s} \times \Z_{>0} \colon r_{s,j} = r\}\), for \(r \in \mathbf{r}\).
In other words, for a specified \(r \in \mathbf{r}\), in the interior sum includes only the terms~\(w_{s,j} \rho_{r_{s,j}} w_{s,j}^{-1}\) coming from~\eqref{eq:expression-identity} which are conjugates of~\(\rho_{r}^{\pm1}\).
\end{proposition}
\begin{proof}
    This formula can be deduced from~\cite[Section 3.3]{MillerPowell}.
\end{proof}

For later use, we summarise the considerations of this section and the previous one in the case of an aspherical $3$-manifold endowed with a handle decomposition with a single $3$-handle.

\begin{proposition}\label{prop:Summary}
    Let~\(M\) be a closed, connected, oriented, aspherical \(3\)-manifold admitting a handle decomposition $\mathcal{H}$ with a single \(0\)-handle and a single \(3\)-handle, and let
    \[s = \prod_{j=1}^{l} w_{j} \rho_{r_{j}}^{\epsilon_{j}} w_{j}^{-1}.\]
    be the identity of the presentation~\(\mathcal{P}_{\mathcal{H}}=\langle \textbf{x} \mid \textbf{r} \rangle \) of~\(\pi_{1}(M)\) from Proposition~\ref{prop:generating-identities}.
 \begin{itemize}
 \item The handle chain module $C_0^{\text{hnd}}(M;\Z[\pi_{1}(M)])$ is freely generated by the 
 class of a lift the \(0\)-handle~$h^{0}$.
 \item The handle chain module $C_1^{\text{hnd}}(M;\Z[\pi_{1}(M)]$ is freely generated by 
classes of lifts of the \(1\)-handles $\{h^{1}_{x} \colon x \in \mathbf{x}\}.$
 \item  The handle chain module $C_2^{\text{hnd}}(M;\Z[\pi_{1}(M)]$ is freely generated by the classes of lifts of the~\(2\)-handles~\(\{h^{2}_{r} \colon r \in \mathbf{r}\}\),
 \item The handle chain module $C_3^{\text{hnd}}(M;\Z[\pi_{1}(M)]$ is freely generated by the class of a lift of the \(3\)-handle~$h^{3}_{s}$.
 \end{itemize}
The differentials can be expressed in terms of Fox derivatives as
\begin{align*}
  \partial_{1}(h^{1}_{x}) &= (x-1)h^{0}, \quad \text{\ for\ } x \in \mathbf{x}, \\
  \partial_{2}(h^{2}_{r}) &= \sum_{x \in \mathbf{x}} \frac{\partial r}{\partial x} h^{1}_{x}, \quad \text{\ for\ } r \in \mathbf{r}, \\
  \partial_{3}(h^{3}_{s}) &= \sum_{r \in \mathbf{r}} \frac{\partial s}{\partial \rho_{r}} h^{2}_{r},
\end{align*}
where the Fox~derivatives~\(\frac{\partial s}{\partial \rho_{r}}\) are computed in~\(F \ast P\).

With respect to these bases, the map~\(\Phi\) is given by the formula
\[\Phi((h^{2}_{r})^{\#}) = \sum_{x \in \mathbf{x}} \sum_{j \in \mathcal{I}(r)} \epsilon_{j} w_{j}^{-1} \frac{\partial w_{j}}{\partial x} h^{1}_{x},\]
where~\(\mathcal{I}(r) = \{j \in \Z_{>0} \colon r_{j} = r\}\), for \(r \in \mathbf{r}\).
\end{proposition}
\begin{proof}
The calculation of the handle chain complex follows from Corollary~\ref{corollary:3-dim-cw-complex-to-alg-data} and Lemma~\ref{lemma:comparison-handle-cellular-chain-cplx}.
The formula for the symmetric structure comes from Proposition~\ref{prop:symmetric-structure}.
\end{proof}

\color{black}

\section{Explicit computations for the torus knot~$T_{2,2k+1}$}
\label{sec:ExplicitComputations}
We now describe how to apply the algorithm mentioned in Subsection~\ref{sub:Algorithm} (see also Algorithm~\ref{rem:Algorithm} and Remark~\ref{rem:AlgoTop}) to metabelian Blanchfield forms of $(2,2k+1)$-torus knots.
In the notation of Section~\ref{sec:MetabelianBlanchfield}, these pairings are of the form $\Bl_{\alpha(2,\chi)}(T_{2,2k+1})$, where $\chi \colon H_1(\Sigma_2(T_{2,2k+1})) \to \Z_{2k+1}$ is a character.\footnote{In Section~\ref{sec:MetabelianBlanchfield}, the characters were assumed to be of prime power order to ensure that $\alpha_K(n,\chi)$ is acyclic. In this section, $\alpha_{T_{2,2k+1}}(2,\chi)$ will turn out to be acyclic for \emph{any} $\chi$ and so no prime power assumption is needed.}
In more detail, Subsection~\ref{sub:ChainComplex} describes the chain complex $C_*(\widetilde{M}_{T_{2,2k+1}})$; Subsection~\ref{sub:blanchf-forms-twist}, uses Powell's algorithm (as described in Subsection~\ref{sub:review_powell}) to understand $\Bl_{\alpha(2,\chi)}(T_{2,2k+1})(x,x)$ for any $x \in H_1(M_{T_{2,2k+1}};\LC^2_{\alpha(2,\chi)})$ (in fact we will be working with cohomological pairing) and Subsection~\ref{sub:MetabelianSignaturesTorusKnot} applies the algorithm from Remark~\ref{rem:Algorithm} to deduce the isometry type of $\Bl_{\alpha(2,\chi)}(T_{2,2k+1})$, thus proving Theorem~\ref{thm:DecompoTwistedBlanchfieldIntro} from the introduction.
\color{black}

Throughout this section, we will be working with the diagram of \(T_{2,2k+1}\) depicted in Figure~\ref{fig:diagram-T_2_2k+1}.
\begin{figure}[!htb]
  \centering
  \def\svgscale{0.7}
  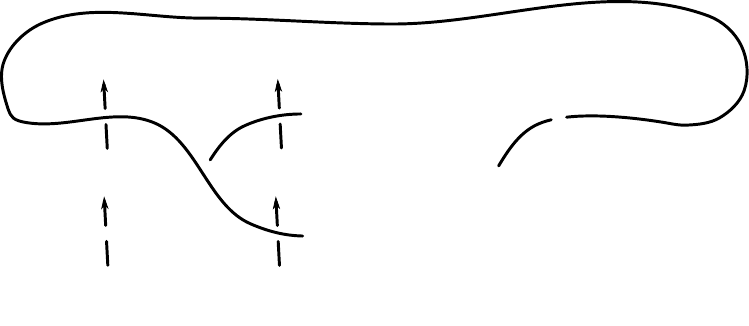
  \caption{A diagram of \(T_{2,2k+1}\) together with generators of the knot group. Arrows indicate orientation of the respective meridian when going under the knot. The blue loop is \(a = x_{2k}x_{2k+1}=x_{1}x_{2}\).}
  \label{fig:diagram-T_2_2k+1}
\end{figure}

\subsection{The chain complex of the universal cover}
\label{sub:ChainComplex}

In this subsection, we explicitly describe the chain complex of the universal cover of the~$0$-framed surgery~$M_{T_{2,2k+1}}$.
\medbreak
We start by providing a presentation for the fundamental group of~$M_{T_{2,2k+1}}$.
Such a presentation can be computed from the knot group once a word for the longitude of~$T_{2,2k+1}$ is known.
Using Figure~\ref{fig:diagram-T_2_2k+1}, a Wirtinger presentation of the fundamental group of the complement of \(T_{2,2k+1}\) is
\[\pi := \pi_{1}(S^{3} \setminus T_{2,2k+1}) = \langle x_{1}, x_{2}, \ldots, x_{2k+1} \mid r_{1}, r_{2}, \ldots, r_{2k} \rangle,\]
where \(r_{i} = x_{i} x_{i+1} x_{i+2}^{-1} x_{i+1}^{-1}\) and the indices are taken mod \(2k+1\).
We will now simplify the presentation. A standard argument is recalled, because
it is used
to describe the presentation of the fundamental group of~$M_{T_{2,2k+1}}$.

\begin{lemma}
  \label{lem:PresentationTorusKnot}
  There exists an isomorphism \(\phi \colon \pi \to G:=\langle a, b \mid a^{2k+1}b^{2} \rangle\) such that
  \begin{align*}
    \phi(x_{2(k-s)+1}) &= a^{s} (a^{k}b)^{-1} a^{-s}, \quad 0 \leq s \leq k, \\ 
    \phi(x_{2(k-s)}) &= a^{s} (a^{k+1}b) a^{-s}, \quad 0 \leq s \leq k-1.
  \end{align*}
\end{lemma}
\begin{proof}
  We first prove that \(\pi\) is generated by \(x_{2k}\) and \(x_{2k+1}\).
  Notice that the relation \(r_{i}\) implies that the equality \(x_{i}x_{i+1} = x_{i+1}x_{i+2}\) holds.
  As a consequence, we obtain \(x_{i} x_{i+1} = x_{i+s}x_{i+s+1}\) for any \(i\) and \(s\) and, applying this formula recursively, we obtain
  \begin{align*}
    x_{2k-1} &\stackrel{r_{2k-1}}{=} x_{2k}x_{2k+1}x_{2k}^{-1}, \\
    x_{2k-2} &\stackrel{r_{2k-2}}{=} x_{2k-1}x_{2k}x_{2k-1}^{-1} = (x_{2k} x_{2k+1}) x_{2k} (x_{2k}x_{2k+1})^{-1}, \\
    x_{2k-3} &\stackrel{r_{2k-3}}{=} x_{2k-2}x_{2k-1}x_{2k-2}^{-1}  \\
             &= (x_{2k}x_{2k+1}x_{2k}x_{2k+1}^{-1}x_{2k}^{-1})(x_{2k}x_{2k+1}x_{2k}^{-1})(x_{2k}x_{2k+1}x_{2k}^{-1}x_{2k+1}^{-1}x_{2k}^{-1})  \\
             &= (x_{2k}x_{2k+1})x_{2k}x_{2k+1}x_{2k}^{-1}(x_{2k}x_{2k+1})^{-1}.
  \end{align*}
  As a consequence, we eliminate the generators \(x_{1}, x_{2}, \ldots, x_{2k-1}\) and the relations \(r_{1}, r_{2}, \ldots, r_{2k-1}\) since we have the following equalities:
  \begin{align*}
    x_{2(k-s)+1} &= (x_{2k}x_{2k+1})^{s-1} x_{2k}x_{2k+1}x_{2k}^{-1} (x_{2k}x_{2k+1})^{-s+1}, \quad 1 \leq s \leq k, \\
    x_{2(k-s)} &= (x_{2k}x_{2k+1})^{s} x_{2k+1} (x_{2k}x_{2k+1})^{s}, \quad 1 \leq s \leq k-1.
  \end{align*}
  Reformulating, we have proved that~$\pi$ is indeed generated by~$x_{2k}$ and~$x_{2k+1}$ and we have obtained the presentation~$\pi = \langle x_{2k}, x_{2k+1} \mid r_{2k}' \rangle$, where
  \begin{align*}
    r_{2k}' &= x_{2k} x_{2k+1} x_{1}^{-1} x_{2k+1}^{-1} = x_{2k} x_{2k+1} (x_{2k} x_{2k+1})^{k-1} x_{2k} x_{2k+1}^{-1} x_{2k}^{-1} (x_{2k}x_{2k+1})^{-k+1} x_{2k+1}^{-1} \\
            &= (x_{2k} x_{2k+1})^{k} x_{2k} (x_{2k} x_{2k+1})^{-k} x_{2k+1}^{-1}.
  \end{align*}
  Using this presentation of~$\pi$, it is now straightforward to verify that \(\phi \colon \pi \to G \) is a group isomorphism, where the inverse is defined by setting~$\phi^{-1}(a) = x_{2k}x_{2k+1}$ and~$\phi^{-1}(b) = (x_{2k+1}^{-1}x_{2k}^{-1})^{k} x_{2k+1}^{-1}$.
  This concludes the proof of the lemma.
\end{proof}

As a consequence, we can describe the fundamental group of the 0-framed surgery on~$T_{2,2k+1}$.

\begin{lemma}
  \label{lem:Presentation0Surgery}
  The fundamental group of the \(0\)-surgery on \(T_{2,2k+1}\) admits the following presentation:
  \begin{equation}\label{eq:presentation-fund-gp-zero-surgery}
    G_{0} = \langle a, b \mid a^{2k+1}b^{2}, (a^{k}b)^{2k+1} a^{2k+1} (a^{k}b)^{2k+1}\rangle.
  \end{equation}
\end{lemma}
\begin{proof}
  Let \( (\mu,\lambda) \) be the meridian-longitude pair of \(T_{2,2k+1}\) expressed in words of~$\pi$.
  Using van Kampen's theorem, to prove the proposition, it suffices to show that~$\phi(\lambda) = (a^{k}b)^{2k+1}a^{2k+1}(a^{k}b)^{2k+1}$, where~$\phi$ is the isomorphism described in Lemma~\ref{lem:PresentationTorusKnot}.
  Working with the Wirtinger presentation arising from Figure~\ref{fig:diagram-T_2_2k+1}, we choose~$(\mu,\lambda)$ as follows:
  \[\mu = x_{2k+1}, \quad \lambda = x_{2k+1}^{-2k-1} x_{1} x_{3} x_{5} \cdots x_{2k+1} x_{2} x_{4} \cdots x_{2k}.\]
  Next, the definition of~$\phi$ gives~$\phi(x_{1} x_{3} \cdots x_{2k+1}) = a^{k} (b^{-1} a^{-k-1})^{k+1} a = a^{2k+1} (a^{-k-1}b^{-1})^{k+1} a^{-k}$ and~$\phi(x_{2} x_{4} \ldots x_{2k}) = a^{k} (a^{k}b)^{k}$.
  Since the relation \(a^{2k+1}b^{2}=1\) holds, we have \(a^{-k-1}b^{-1} = a^{k}b\), we deduce that~$\phi(x_{1}x_{3} \cdots x_{2k+1} x_{2} \cdots x_{2k}) = a^{2k+1} (a^{k}b)^{2k+1}$, and the lemma is now concluded by recalling the definition of~$\lambda$.
\end{proof}

We now describe a handle decomposition for~$M_{T_{2,2k+1}}$.
Recall from Remark~\ref{rem:HandleDiagram} that it is possible to obtain a handle decomposition of $M_{T_{2,2k+1}}$ from a reduced diagram for $T_{2,2k+1}$ and that this decomposition can be used to calculate twisted Blanchfield forms.
  While this handle decomposition is easy to describe, it has one serious disadvantage: the number of handles grows with \(k\).
  To be able to perform computations for the whole family of torus knots \(T_{2,2k+1}\), for~\(k \geq 1\), we use a handle decomposition with far fewer handles.

\begin{construction}
We describe an explicit handle decomposition for the exterior $X_{T_{2,2k+1}}$.
Our strategy is to first describe~$X_{T_{2,2k+1}}$ as a union of three solid tori $U_1,U_2$ and $V_1$.
We will then read off a handle decomposition for~$X_{T_{2,2k+1}}$ from a handle decomposition of the solid tori.

  Consider the standard genus-one Heegaard decomposition~\(S^{3} = H_{1} \cup_{\partial} H_{2}\) with Heegaard surface the torus~\(T = H_{1} \cap H_{2}\).
  The solid tori $H_1$ and $H_2$ are sketched on the left frame of Figure~\ref{fig:HeegaardSplitting}.
Pick two parallel copies~$K_1,K_2 \subset T$ of the torus knot $T_{2,2k+1}$ that lie on the torus $T$ as illustrated in the central frame of Figure~\ref{fig:HeegaardSplitting}.
The right hand side of this figure shows closed neighbourhoods~$\overline{\nu}(K_1) ,\overline{\nu}(K_2) \subset T$ that satisfy 
$$T = \overline{\nu}(K_1) \cup \overline{\nu}(K_2) \text{ and } \overline{\nu}(K_1) \cap \overline{\nu}(K_2) =  \partial \overline{\nu}(K_1) = \partial \overline{\nu}(K_2).$$
Shrink the solid tori~$H_1 \cong S^1 \times D^2$ and $H_2 \cong S^1 \times D^2$ to obtain half-sized solid tori $S^1 \times D^2_{\frac{1}{2}} \cong U_i \subset H_i$.
   It follows that~\(V := \overline{S^{3} \setminus (U_{1} \cup U_{2})}\) is a tubular neighbourhood of~\(T \subset S^3\) and can be identified  with~$V \cong T \times I \cong (\overline{\nu}(K_1) \cup \overline{\nu}(K_2)) \times I$.
Since we shrank the solid tori $H_1$ and $H_2$ but expanded the torus $T=\overline{\nu}(K_1) \cup \overline{\nu}(K_2)$, we obtain
 $$S^3=V \cup U_1 \cup U_2=(\overline{\nu}(K_1) \times I) \cup (\overline{\nu}(K_2) \times I)    \cup U_1 \cup U_2.$$

\begin{figure}[!htbp]
    \centering
    \subcaptionbox{The Heegard splitting of~\(S^{3}\).}[0.3\textwidth]{\def\svgscale{0.3}
      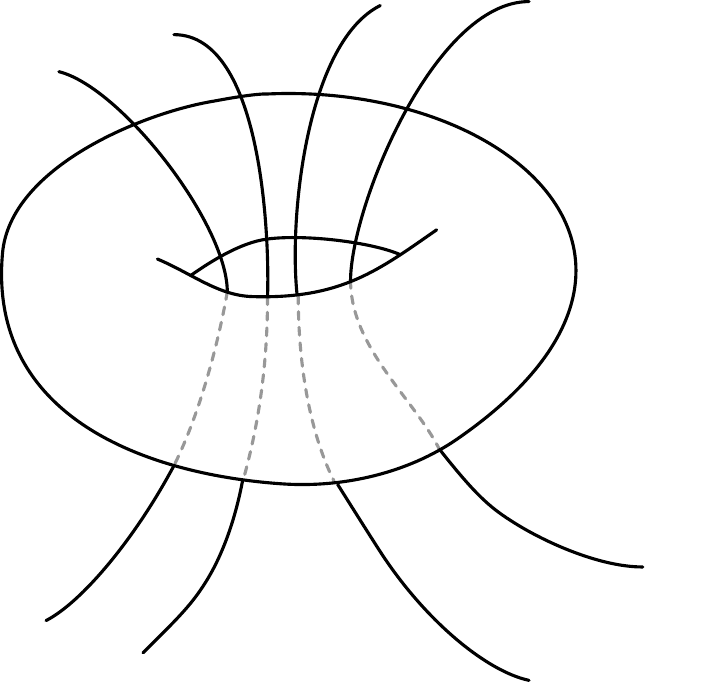}
    \subcaptionbox{The knots \(K_{1}\) (blue) and \(K_{2}\) (red) on the torus~\(T\).}[0.3\textwidth]{\def\svgscale{0.2}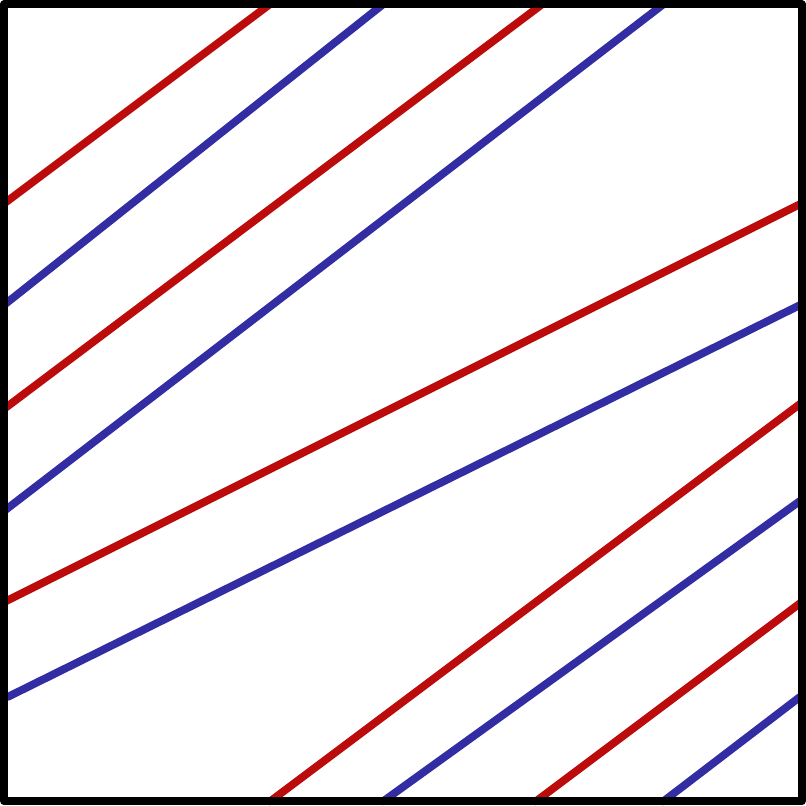}
    \subcaptionbox{The neighbourhoods \(\overline{\nu}(K_{1})\) (light blue) and \(\overline{\nu}(K_{2})\) (light red).}[0.3\textwidth]{\def\svgscale{0.2}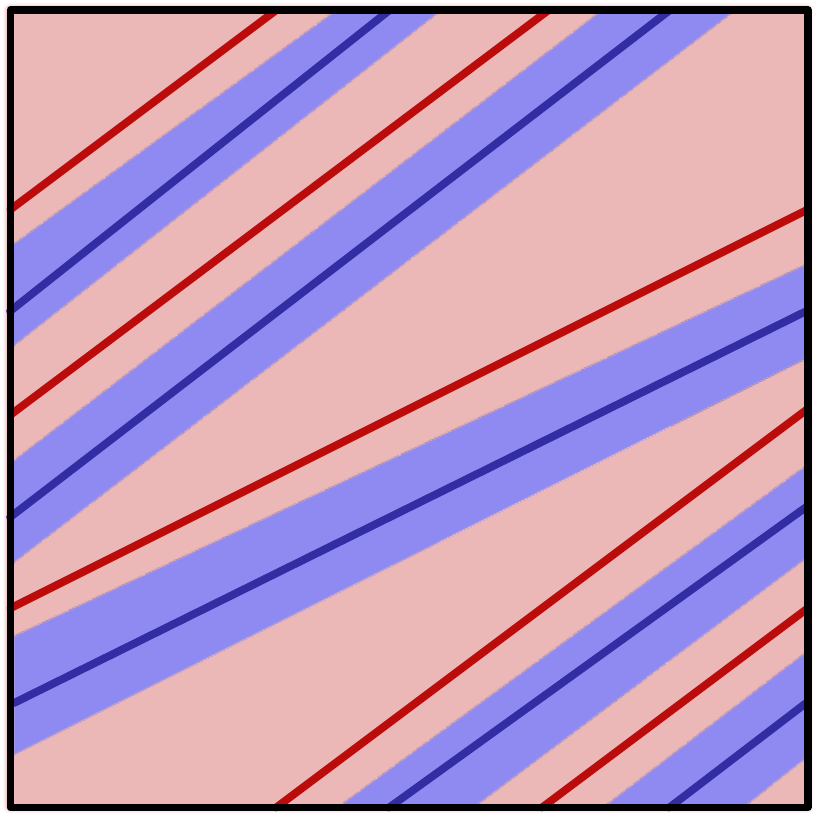}
    \caption{Left frame: the standard genus $1$ Heegaard decomposition of $S^3$. 
      Central frame: the knots $K_1$ and $K_2$ lying in the Heegaard torus $T$.
      Right frame: the neighbourhoods $\overline{\nu}(K_1)$ and $\overline{\nu}(K_2)$ of $K_1$ and $K_2$ that satisfy~$T = \overline{\nu}(K_1) \cup \overline{\nu}(K_2)$ and~$\overline{\nu}(K_1) \cap \overline{\nu}(K_2) =  \partial \overline{\nu}(K_1) = \partial \overline{\nu}(K_2)$.
}
\label{fig:HeegaardSplitting}
\end{figure}

Recall that $K_2 \subset T$ is a copy of the torus knot $T_{2,2k+1}$.
Thus when we remove~$\nu(K_2) \cong K_2 \times \mathring{D}^2$ from $S^3$, and write $V_1:=\overline{\nu}(K_1) \times I$, we obtain the following decomposition of $X_{T_{2,2k+1}}$:
  \begin{equation}
      \label{eq:presentation-torus-knot-complement}
      X_{T_{2,2k+1}} = U_{1} \cup V_{1} \cup U_{2}.
  \end{equation}
The intersection $V_1 \cap U_i$ consists a normal push off of the neighbourhood $\overline{\nu}(K_{1}) \subset T$ into $U_i$.
Put differently,  $V_1 \cap U_i$ is diffeomorphic to $\overline{\nu}(K_i')\cong K_i' \times I$, where $K_i' \subset U_i$ is a normal push-off of the knot~$K_i\subset T$.

The solid tori $U_i$ have handle decompositions with a single \(0\)-handle and a single \(1\)-handle, say \(U_{i} = h_{i}^{0} \cup h_{i}^{1}\) and we will obtain a handle decomposition for $X_{T_{2,2k+1}} $ by decomposing $V_1$ as a union of two $3$-balls $B_1$ and $B_2$ and showing that adding $V_1$ to $U_1 \sqcup U_2$ is the same as adding a~$1$-handle followed by a $2$-handle.

\begin{figure}[!htbp]
    \centering
    \def\svgscale{0.3}
    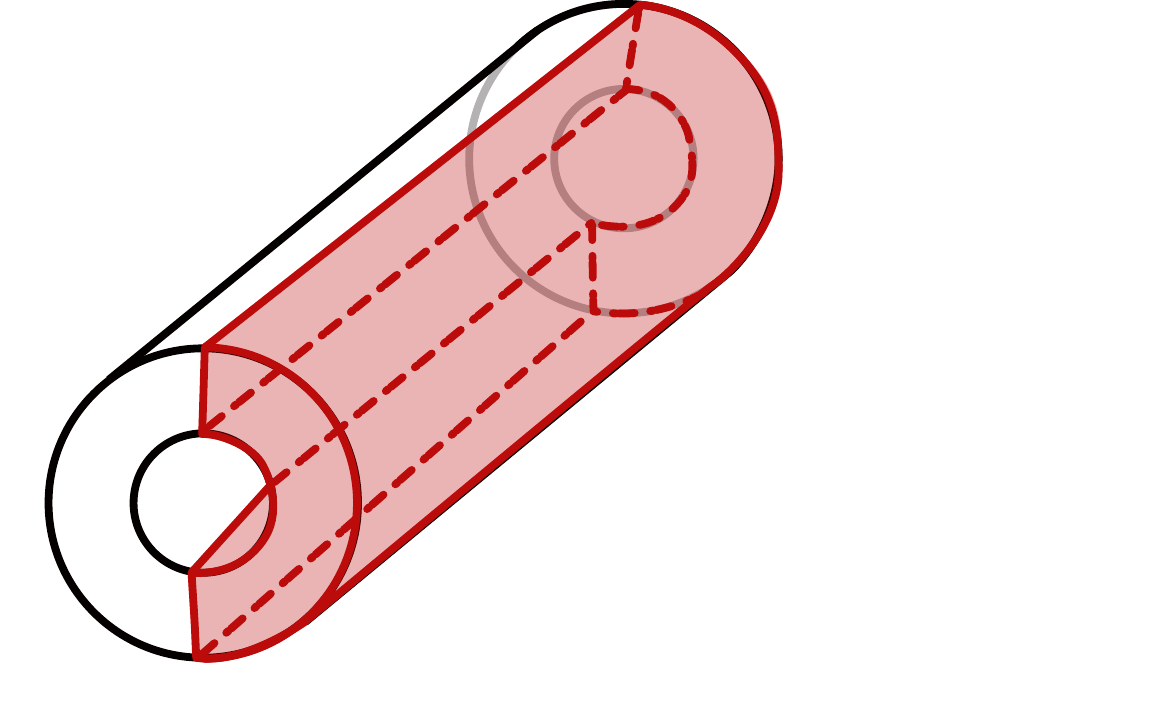
    \caption{The solid torus $V_1 \cong S^1 \times I^2$ and its decomposition as a union of two~$3$-balls, $B_1$ (red) and $B_2$ (the remaining part of \(V_{1}\)).
      The attaching region~$\partial_+ B_1=\partial_{+,1} B_{1} \sqcup \partial_{+,2} B_{1} $ of $B_1$, thought of as a $1$-handle, is also shown.}
\label{fig:V1}
\end{figure}

Fix a diffeomorphism \(\psi \colon S^{1} \times I^{2} \xrightarrow{\cong} V_{1}\)
 such that \(N_{1} = V_{1} \cap T = \psi(S^{1} \times I \times \{1/2\})\), and \(V_{1}   \cap U_{i} = \psi(S^{1} \times I \times \{i-1\})\).
Define \(B_{1} = \psi([0,\pi] \times I^{2}) \subset V_{1}\) and \(B_{2} = \psi([\pi,2 \pi] \times I^{2}) \subset V_{1}\).
Note that both~$B_1$ and~$B_2$ are diffeomorphic to $3$-balls and that~$V_1=B_1 \cup B_2$, as illustrated in Figure~\ref{fig:V1}.
    In what follows,  we will be thinking of $B_1$ as a $3$-dimensional $1$-handle and of $B_2$ as a $3$-dimensional $2$-handle. 
The attaching region~$\partial_+ B_1=\partial_{+,1} B_{1} \sqcup \partial_{+,2} B_{1} $ of~$B_1$ can be also be seen in Figure~\ref{fig:V1}.

\begin{itemize}
\item When $V_1$ is glued to $U_1 \sqcup U_2$,  the $1$-handle \(h_{3}^{1}=B_{1}\) is attached to the union \(U_{1} \sqcup U_{2}\) as illustrated in the top frame of Figure~\ref{fig:AttachB2}.
In more detail, we consider the aforementioned normal push offs $K_1' \subset U_1$ and $K_2' \subset U_2$ of~$K_1,K_2 \subset T$, and  the attaching regions of~$B_1$ and~$B_2$ are identified with a small $2$-disc neighbourhood of an unknotted portion of~$K_1'$ and~$K_2'$.

The effect of the handle attachment, $Z := U_{1} \cup B_{1} \cup U_{2},$ is diffeomorphic to a genus two handlebody and admits $h_{0}^{1} \cup h_{2}^{0} \cup h_{1}^{1} \cup h_{2}^{1} \cup h_{3}^{1}$ as a handle decomposition.
By performing isotopies of handle decompositions of \(U_{1}\) and \(U_{2}\), we can assume that~\(h_{3}^{1}\) cancels geometrically with either \(0\)-handle. 
  Thus, there is a handle decomposition of~\(Z\) of the  form
  \[Z = h_{1}^{0} \cup h_{2}^{0} \cup h_{1}^{1} \cup h_{2}^{1} \cup h_{3}^{1} \cong h_{1}^{0} \cup h_{1}^{1} \cup h_{1}^{2},\]
  where the right-hand side diffeomorphism results from cancelling \(h_{2}^{0}\) and \(h_{3}^{1} = B_{1}\).

\item When $V_1$ is glued to $U_1 \sqcup U_2$,  the $2$-handle \(h_r^{2}=B_{2}\) is attached to $Z$ as illustrated in the bottom frame of Figure~\ref{fig:AttachB2}.
  In more detail, the $2$-handle $B_2$ is attached to $Z$ along the curve $J \subset \partial Z$ given by the connected sum of $K_1'$ and $K_2'$ that is also illustrated in the bottom frame of Figure~\ref{fig:AttachB2}.

\begin{figure}[!htbp]
\centering
\includegraphics[scale=0.45]{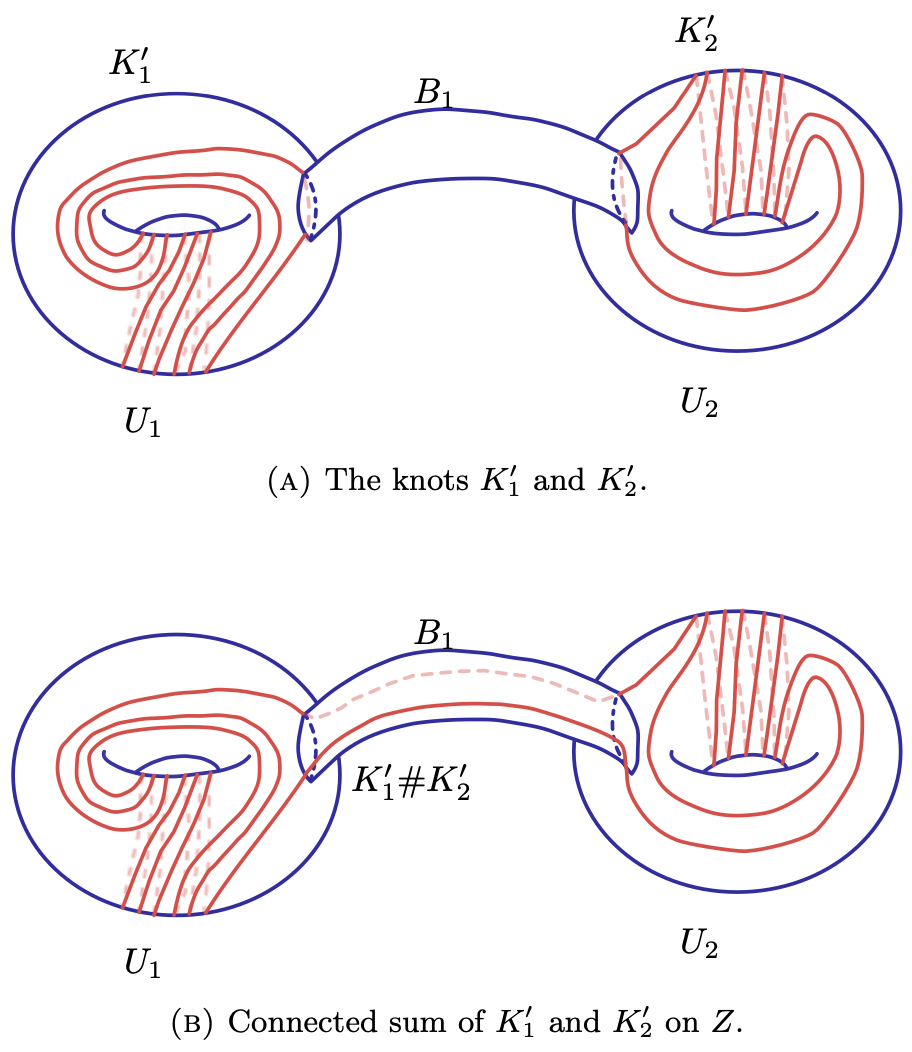}
%
%
 \label{fig:AttachingHandles}
  \caption{The space $Z$ obtained from $U_1 \sqcup U_2$ by attaching $B_1$ is diffeomorphic to a genus two handlebody.
This top figure shows the torus knots $K_1' \subset U_1$ and $K_2' \subset U_2$, expressed using a symplectic basis for $\partial U_1$ and $\partial U_2$.
Taking the connected sum of these knots (as depicted in the bottom figure) leads to the knot~$J$, which serves as the attaching circle for the attachment of the $2$-handle~$B_2$.}
\label{fig:AttachB2}
\end{figure}

\end{itemize}

As a result, we have obtained a handle decomposition of \(X_{T_{2,2k+1}}\) of the form
  \[X_{T_{2,2k+1}} = h^{0}_{1} \cup h^{0}_{2} \cup h^{1}_{1} \cup h^{1}_{2} \cup h^{1}_{3} \cup h^{2}_{r}.\]
When we cancel \(h_{2}^{0}\) with \(h_{3}^{1}\), we obtain a handle decomposition
  \begin{equation}
      \label{eq:handle-decomposition-complement}
      X_{T_{2,2k+1}} = h^{0} \cup h_{1}^{1} \cup h_{2}^{1} \cup h_{r}^{2}.
  \end{equation}
  Thus~\(X_{T_{2,2k+1}}\) can be obtained from the genus-two handlebody \(Z\) by attaching a \(2\)-handle \(h_{r}^{2}\) along a simple closed curve
  representing the element \(\alpha = a^{2k+1} c^{2} b^{2} d^{2k+1}\)
     where \(a,c,b,d\) is a standard basis of
      \(\pi_{1}(\partial Z)\),  with \(c\) and \(d\)   representing meridians of \(Z\).
  In particular, \(\pi_{1}(Z)\) is a rank-two free group generated by \(a\) and \(b\),  and \(J\), viewed as a curve in \(Z\), represents the element \(a^{2k+1}b^{2}\), as desired. 
Observe that there is no framing ambiguity in this case, see~\cite[Example 4.1.4.(c)]{Gompf-Stipsicz}. 
\end{construction}

\begin{construction}\label{construction:handle-decomp-zero-surgery}
Now that we have the handle decomposition~\eqref{eq:handle-decomposition-complement} for $X_{T_{2,2k+1}}$,  we can obtain a handle decomposition of the zero-surgery \(M_{T_{2,2k+1}}\).
Indeed, to the handle decomposition~\eqref{eq:handle-decomposition-complement} we attach a \(2\)-handle \(h_{\lambda}^{2}\), which annihilates the zero-framed longitude of \(T_{2,2k+1}\) and a~\(3\)-handle~\(h^{3}\).
  Consequently, we obtain a handle decomposition of the zero-surgery of the following form
\begin{equation}
\label{eq:handle-decomposition-zero-surgery}
M_{T_{2,2k+1}} = h^{0} \cup h_{1}^{1} \cup h_{2}^{1} \cup h_{r}^{2} \cup h_{\lambda}^{2} \cup h^{3}.
\end{equation}
By construction, this handle decomposition leads to the presentation $G_0$ of $\pi_1(M_{T_{2,2k+1}} )$.
\end{construction}

\color{black}

Now let us study identities of the presentation~\(G_{0}\) from~\eqref{eq:presentation-fund-gp-zero-surgery}.
Following the notation in Section~\ref{sub:identities-of-presentation}, let \(\mathbf{x} = \{a,b\}\) be the set of generators of~\(G_{0}\) and let \(\mathbf{r} = \{R = a^{2k+1}b^{2},\lambda = (a^{k}b)^{2k+1}a^{2k+1}(a^{k}b)^{2k+1}\}\) be the set of relations of~\(G_{0}\).
Denote by \(F\) the free group on the set~\(\mathbf{x}\) and by \(P\) the free group generated by symbols \(\rho_{r}\), for~\(r \in \mathbf{r}\).
By~\(N(G_{0})\) we denote the normal subgroup of~\(F \ast P\) generated by~\(P\).
Consider the map \(\psi \colon F \ast P \to F\) given by the formula
\[\psi(x) = x, \quad \psi(\rho_{r}) = r,\]
where~\(x \in \mathbf{x}\) and~\(r \in \mathbf{r}\).
We will denote by the same symbol the restriction of~\(\psi\) to~\(N(G_{0})\).
Recall that \(I(G_{0}) := \ker(\psi) \cap N(G_{0})\) and \([[N(G_{0}),N(G_{0})]]_{\psi}\) denote the normal subgroup of~\(N(G_{0})\) generated by Peiffer commutators, see Definition~\ref{def:CrossedModule}.

\begin{lemma}\label{lemma:identity-of-presentation}
  Set \(R = a^{2k+1}b^{2}\), \(\lambda = (a^{k}b)^{2k+1}a^{2k+1}(a^{k}b)^{2k+1}\) and \(\mu = (a^{k}b)^{-1}\).
  The following equality is satisfied in the free group~$F$ generated by~$a$ and~$b$:
  \[\mu \lambda \mu^{-1} \lambda^{-1} = \left[\mu^{-2k}a^{k}\right] R \left[\mu^{-2k}a^{k}\right]^{-1} \mu^{-2k-1} R^{-1} \mu^{2k+1}.\]
  Consequently, the presentation~\eqref{eq:presentation-fund-gp-zero-surgery} of~$\pi_{1}(M_{T_{2,2k+1}})$ admits the following identity:
  \begin{equation}\label{eq:identity-of-the-presentation}
      s = \mu \rho_{2} \mu^{-1} \rho_{2}^{-1} \mu^{-2k-1} \rho_{1} \mu^{2k+1} \left(\mu^{-2k} a^{k}\right) \rho_{1}^{-1}\left(\mu^{-2k}a^{k}\right)^{-1}.
  \end{equation}
Furthermore, \(s\) is a generator of the group of identities~\(I(G_{0}) / [[N(G_{0}),N(G_{0}))]]_{\psi}\).
\end{lemma}
\begin{proof}
  The proof of the first assertion is a direct computation:
  \begin{align*}
    \mu  \lambda  \mu^{-1} \lambda^{-1} &= (a^{k}b)^{2k}  a^{3k+1}  b  a^{-2k-1}  (a^{k}b)^{-2k-1}  \\
                        &= \left[(a^{k}b)^{2k+1} a^{k} \right]  (a^{2k+1}b^{2})  b  (b^{-2}a^{-2k-1})  (a^{k}b)^{-2k-1}  \\
                        &= \left[\mu^{-2k-1} a^{k} \right]  R  \left[\mu^{-2k-1} a^{k} \right]^{-1}  \left[\mu^{-2k-1} a^{k} \right] b  \mu^{2k+1}  \mu^{-2k-1}  R^{-1}  \mu^{2k+1}  \\
                        &= \left[\mu^{-2k}a^{k}\right] R \left[\mu^{-2k}a^{k}\right]^{-1} \mu^{-2k-1} R^{-1} \mu^{2k+1}.
  \end{align*}
  Consequently, \(s\) is indeed an identity of the presentation~\(G_{0}\).

To prove the last assertion, observe that the handle decomposition from Construction~\ref{construction:handle-decomp-zero-surgery} admits a single \(3\)-handle.
    Therefore, by Proposition~\ref{prop:generating-identities}, \(G_{0}\) admits a single identity, let us denote it by~\(s_{1}\), which is unique up to conjugation by elements of~\(F\) and up to Peiffer commutators.
    Furthermore, every element of~\(I(G_{0}) / [[N(G_{0}),N(G_{0}))]]_{\psi}\) can be written, modulo Peiffer commutators, as a product of conjugates of~\(s_{1}\) by elements of~\(F\).
\end{proof}

\begin{remark}
We came up with the identity $s$ in~\eqref{eq:identity-of-the-presentation} as follows.
  Since~\(\mu\) and~\(\lambda\) commute in~$\pi_1(X_{T_{2,2k+1}})$, the commutator \(\mu \lambda \mu^{-1} \lambda^{-1}\) is expressible as a product of conjugates of the relation~\(r\) and its inverse (cf.\ the first equation in Lemma~\ref{lemma:identity-of-presentation}).
Identities are typically obtained by combining two relations in $\pi_1$, so we then combined the aforementioned relation (specifically, using $r$) with the relation~\(\lambda=1\).
\end{remark}

Having fixed the handle decomposition of \(M_{T_{2,2k+1}}\), we can describe the \purple{handle} chain
complex~\(C_{\ast}^{\text{hnd}}(\widetilde{M}_{T_{2,2k+1}})\) as well as the component of the chain
homotopy equivalence~\(\Phi \colon C^{2}_{\text{hnd}}(\widetilde{M}_{T_{2,2k+1}}) \to C_{1}^{\text{hnd}}(\widetilde{M}_{T_{2,2k+1}})\).
For a quick review of the handle chain complex, refer to Section~\ref{sub:handle-chain-complex}.

\begin{proposition}\label{prop:cellular-chain-complex}
  The handle chain complex for the universal cover of the zero-surgery \(M_{T_{2,2k+1}}\) is given by the~$\Z[G_0]$-module chain complex~$\Z[G_{0}] \xrightarrow{\partial_{3}} \Z[G_{0}]^{2} \xrightarrow{\partial_{2}} \Z[G_{0}]^{2} \xrightarrow{\partial_{1}} \Z[G_{0}],$ where the differentials are described by the following formulas:
  \begin{align*}
    \partial_{3} &=
         \begin{pmatrix}
           \mu^{-2k-1}-\mu^{-2k} \cdot a^{k} & \mu-1 \\
         \end{pmatrix}, \\
    \partial_{2} &= 
         \begin{pmatrix}
           \frac{\partial (a^{2k+1}b^{2})}{\partial a} & \frac{\partial (a^{2k+1}b^{2})}{\partial b} \\
           \frac{\partial \lambda_{0}}{\partial a}    & \frac{\partial \lambda_{0}}{\partial b} \\
         \end{pmatrix}, \\
    \partial_{1} &=
         \begin{pmatrix}
           a-1 & b-1
         \end{pmatrix}^{T}.
  \end{align*}
  Moreover, if \(\Phi \colon C^{2}_{\text{hnd}}(\widetilde{M}_{T_{2,2k+1}}) \to C_{1}^{\text{hnd}}(\widetilde{M}_{T_{2,2k+1}})\) denotes the chain map 
mentioned in~\eqref{eq:SymmetricStructure},   then
  \begin{equation}
    \label{eq:Phi}
    \Phi =
    \begin{pmatrix}
      \mu^{2k+1} \cdot \frac{\partial (\mu^{-2k-1})}{\partial a} - a^{-k} \cdot \mu^{2k} \cdot \frac{\partial (\mu^{-2k} \cdot a^{k})}{\partial a} & \mu^{2k+1} \cdot \frac{\partial (\mu^{-2k-1})}{\partial b} - a^{-k} \cdot \mu^{2k} \cdot \frac{\partial (\mu^{-2k} \cdot a^{k})}{\partial b} \\
      \mu^{-1} \cdot \frac{\partial \mu}{\partial a} & \mu^{-1} \cdot \frac{\partial \mu}{\partial b} \\
    \end{pmatrix}.
  \end{equation}
\end{proposition}
\begin{proof}
Endow $M_{T_{2,2k+1}}$ with  the handle decomposition from~\eqref{eq:handle-decomposition-zero-surgery}.
This handle decomposition admits a single $3$-handle.
      By construction, this handle decomposition induces the presentation
$$ \pi_1(M_{T_{2,2k+1}}) \cong \langle a, b \mid a^{2k+1}b^{2}, (a^{k}b)^{2k+1} a^{2k+1} (a^{k}b)^{2k+1}\rangle $$
and the identity 
$$ R = \mu \rho_{2} \mu^{-1} \rho_{2}^{-1} \mu^{-2k-1} \rho_{1} \mu^{2k+1} \left(\mu^{-2k} a^{k}\right) \rho_{1}^{-1}\left(\mu^{-2k}a^{k}\right)^{-1}.$$
Since~\(M_{T_{2,2k+1}}\) is aspherical~\cite[Corollary 5]{Gabai}, the result now follows from Proposition~\ref{prop:Summary}.
\end{proof}

In later computations, we will need the explicit formulas for the~$\partial_i$ and, for this reason, we record the following computations which only require Fox calculus: 
\begin{align*}
  \frac{\partial (a^{2k+1}b^{2})}{\partial a} &= 1+a+a^{2}+\cdots+a^{2k}, \\
  \frac{\partial (a^{2k+1}b^{2})}{\partial b} &= a^{2k+1}(1+b), \\
  \frac{\partial \lambda_{0}}{\partial a} &= (1 + (a^{k}b)^{2k+1}   a^{2k+1})   \frac{\partial (a^{k}b)^{2k+1}}{\partial a} + (a^{k}b)^{2k+1} (1+a+a^{2}+\cdots+a^{2k}), \\
  \frac{\partial \lambda_{0}}{\partial b} &= (1 + (a^{k}b)^{2k+1}  a^{2k+1})  \frac{\partial (a^{k}b)^{2k+1}}{\partial b}, \\
  \frac{\partial (a^{k}b)^{2k+1}}{\partial a} &= (1+a^{k}b+(a^{k}b)^{2}+\cdots+(a^{k}b)^{2k})  (1+a+a^{2}+\cdots+a^{k-1}), \\
  \frac{\partial (a^{k}b)^{2k+1}}{\partial b} &= (1+a^{k}b+(a^{k}b)^{2}+\cdots+(a^{k}b)^{2k})  a^{k}.
\end{align*}

\subsection{Blanchfield forms twisted by dihedral representations}
\label{sub:blanchf-forms-twist}
In this subsection, we use Powell's algorithm reviewed in Subsection~\ref{sub:review_powell} to compute metabelian Blanchfield pairings of~$T_{2,2k+1}$.
\medbreak

Use \(\Sigma_{2}(T_{2,2k+1})\) to denote the double cover of \(S^{3}\) branched along \(T_{2,2k+1}\) and let \(\xi = \xi_{2k+1}\) be a primitive root of unity of order \(2k+1\).
Recall from Subsection~\ref{sub:review_metabelian} that for every character~$\chi \colon H_1(\Sigma_{2}(T_{2,2k+1});\Z) \to \Z_{2k+1}$, there is a metabelian representation
\[ \alpha(2,\chi) \colon \pi_1(M_{T_{2,2k+1}}) \to GL_2(\C[t^{\pm 1}]).\]
The representation~$\alpha(2,\chi)$ can be described quite explicitly.
                                                               %
We start by producing generators of the Alexander module of~$T_{2,2k+1}$.
Since the abelianisation map 
$$\operatorname{Ab} \colon G_{0} \cong \pi_1(M_{T_{2,2k+1}})\to \Z$$
sends~$x_{2k}$ and~$x_{2k+1}$ to~$1$, the commutator subgroup \(G_{0}^{(1)} = [G_{0},G_{0}]\) consists of words in \(x_{2k}, x_{2k+1}\) such that the sum of the exponents is zero.
In particular, \(x_{2k}x_{2k+1}^{-1} \in G_{0}^{(1)}\) and it is easy to check that~\(G_{0}^{(1)}\) is normally generated by \(x_{2k} x_{2k+1}^{-1}\).
Therefore, the image of \(x_{2k} x_{2k+1}^{-1}\) generates the Alexander module \(H_{1}(M_{K};\Z[t^{\pm1}]) = G_{0}^{(1)} / G_{0}^{(2)}\) as a \(\Z[t^{\pm1}]\)-module.
Consider the projection
\[q \colon G_{0}^{(1)} \to H_{1}(M_{K};\Z[t^{\pm1}]) \stackrel{}{\to} H_{1}(\Sigma_{2}(T_{2,2k+1});\Z). \]
Since the image of~$x_{2k}x_{2k+1}^{-1}$ generates the Alexander module, it follows that~$q(x_{2k} x_{2k+1}^{-1})$ generates~$H_{1}(\Sigma_{2}(T_{2,2k+1});\Z)$.
This, in particular, recovers the well-known fact that~$H_1(\Sigma_2(T_{2,2k+1});\Z)$ is cyclic, and -- in fact --~$H_1(\Sigma_2(T_{2,2k+1});\Z)\cong\Z_{2k+1}$. 
As a consequence, the~$2k+1$ characters~$H_{1}(\Sigma_{2}(T_{2,2k+1});\Z) \to \Z_{2k+1}$ can be described by imposing that 
\[\chi_{\theta} \colon H_{1}(\Sigma_{2}(T_{2,2k+1});\Z) \to \Z_{2k+1}\]
satisfies \(\chi_{\theta}(q(x_{2k} x_{2k+1}^{-1})) = \theta\) for~$\theta=0,\ldots,2k$.
We shall now use these observations to compute the value of the metabelian representation~$\rho_\theta:=\alpha(2,\chi_{\theta})$ on the generators~$a$ and~$b$ of~$G_0$.
Using~\eqref{eq:Matrix}, we have~$\rho_{\theta}(x_{2k} x_{2k+1}^{-1}) =
\bsm   \xi^{\theta} & 0 \\
0           & \xi^{-\theta} \\
\esm$ and~$ \rho_{\theta}(x_{2k+1}) =\bsm 0 & 1 \\ t & 0 \\ \esm~$, and we therefore obtain\[\rho_{\theta}(x_{2k}) =
  \begin{pmatrix}
    \xi^{\theta} & 0 \\
    0           & \xi^{-\theta} \\
  \end{pmatrix}
  \begin{pmatrix}
    0 & 1 \\
    t & 0 \\
  \end{pmatrix}
  =
  \begin{pmatrix}
    0 & \xi^{\theta} \\
    t \xi^{-\theta} & 0 \\
  \end{pmatrix}.
\]
Next, recall from Lemma~\ref{lem:PresentationTorusKnot} that the generators~$a$ and~$b$ of~$G_0$ are related to the generators~$x_{2k}$ and~$x_{2k+1}$ by the formulas~$a=\phi(x_{2k}x_{2k+1})$ and~$b=\phi(x_{2k+1}^{-1}x_{2k}^{-1})^kx_{2k+1}^{-1}$.
As a consequence, the metabelian representation~$\rho_\theta$ is entirely described by
\begin{align*}
  \rho_\theta(a) &=
           \begin{pmatrix}
             t\xi^{-\theta} & 0 \\
             0 & t \xi^{\theta} \\
           \end{pmatrix}, \\
  \rho_\theta(b) &=
           \begin{pmatrix}
             0 & t^{-k-1} \xi^{k \cdot \theta} \\
             t^{-k} \xi^{-k \cdot \theta} & 0 \\
           \end{pmatrix}.
\end{align*}
We can now work towards the description of the module~$H^{2}(M_{T_{2,2k+1}};\C[t^{\pm 1}]_{\rho_{\theta}}^{2})$ which supports the cohomological Blanchfield pairing.
Consider the polynomials~$P_{k}(t) = 1+t+t^{2}+\cdots+t^{k}$ and~$R_\eta(t)=t+t^{-1}-2 \operatorname{Re}(\eta)$, for~$\eta \in S^1$.
The latter  polynomial is the basic polynomial from Subsection~\ref{sub:review_metabelian}.
Furthermore, we shall also need the following symmetric polynomial:
\[\Delta_{\theta}(t) = t^{-k} \frac{P_{2k}(t)}{\basicR_{\xi^{\theta}}(t)} = \prod_{\stackrel{i=1}{i \neq \theta}}^{k} \basicR_{\xi^{i}}(t).\]
The next lemma describes the cohomology~$\C[t^{\pm 1}]$-module~$H^{2}(M_{T_{2,2k+1}};\C[t^{\pm 1}]_{\rho_{\theta}}^{2})$.

\begin{proposition}
  \label{prop:TwistedModuleComputation}
  The module~$H^{2}(M_{T_{2,2k+1}};\C[t^{\pm 1}]_{\rho_{\theta}}^{2})$ is isomorphic to~$ \C[t^{\pm 1}] / (\Delta_{\theta}(t))$ and admits a generator \([v_2] \in H^{2}(M_{T_{2,2k+1}};\C[t^{\pm 1}]_{\rho_{\theta}}^{2})\) such that
  \[\Bl^{\rho_{\theta}}(T_{2,2k+1})([v_2],[v_2]) = \frac{\frac{1}{2}(t^{2k+1} \xi^{k\theta} - t^{k+1} -t^{k}+\xi^{(k+1)\theta})(t^{-2k-1}+1)}{t^{-k}P_{2k}(t)}.\]
\end{proposition}
\begin{proof}
  Since~$H^{2}(M_{T_{2,2k+1}};\C[t^{\pm 1}]_{\rho_{\theta}}^{2})$ is equal to the quotient~$\ker({{\partial_3}^{\#T}})/\im({{\partial_2}^{\# T}})$, we first compute~$\ker({{\partial_3}^{\#T}})$, before studying~$\im({{\partial_2}^{\#T}})$.
  First of all, note that~$\rho_{\theta}(\mu) = \rho_{\theta}(x_{2k+1}) = \bsm  0 & 1 \\ t & 0 \esm~$.
  Using this computation and looking back to the definition of~$\partial_3$ (Proposition~\ref{prop:cellular-chain-complex}), we deduce that
  \[\rho_{\theta}(\partial_{3}) =
    \begin{pmatrix}
      -\xi^{-k\theta} & t^{-k-1}     & -1 & 1 \\
      t^{-k}         & -\xi^{k\theta} & t & -1 \\
    \end{pmatrix}.
  \]
  The differential \(\rho_{\theta}(\partial_{3})^{\#T}\) has rank two, hence \(\ker (\rho_{\theta}(\partial_{3})^{\#T})\) is~$2$-dimensional.
  In fact, we claim that the kernel of \(\rho_{\theta}(\partial_{3})^{\# T}\) is freely generated by
  \begin{align*}
    v_{1} &= \left(t^{-k}\xi^{-k\theta}, t^{-2k-1}, \xi^{-k\theta} P_{2k}(t^{-1}), \xi^{-k\theta}P_{2k}(t^{-1})\right), \\
    v_{2} &= \left(0, t^{-1}-1, \xi^{(k+1)\theta}-t^{k+1}, \xi^{-k\theta}-t^{k}\right).
  \end{align*}
  Indeed, it is easy to check that \(v_{1} \cdot \rho_{\theta}(\partial_{3})^{\# T} = v_{2} \cdot \rho_{\theta}(\partial_{3})^{\# T} = 0\) and that \(v_{1}\) and \(v_{2}\) are linearly independent.
  We now turn to \( \im(\rho_{\theta}(\partial_{2})^{\# T}) \).
  We first compute \(\rho_{\theta}(\partial_{2})^{\# T}\), then find a basis for~\(C^{1}(M_{T_{2,2k+1}};\C[t^{\pm 1}]_{\rho_{\theta}}^{2})\) and finally compute the image.
  To obtain the first line of~$\rho_\theta(\partial_2)$, we compute the Fox derivatives of \(r=a^{2k+1}b^2\):
  \begin{align*}
    \rho_{\theta}\left(\frac{\partial r}{\partial a}\right) &=
                                        \begin{pmatrix}
                                          P_{2k}(t\xi^{-\theta}) & 0 \\
                                          0 & P_{2k}(t\xi^{\theta}) \\
                                        \end{pmatrix}, \\
    \rho_{\theta}\left(\frac{\partial r}{\partial b}\right) &=
                                        \begin{pmatrix}
                                          t^{2k+1} & t^{k} \xi^{k\theta} \\
                                          t^{k+1} \xi^{-k\theta} & t^{2k+1} \\ 
                                        \end{pmatrix}.
  \end{align*}
  In order to obtain the second line of~$\rho_\theta(\partial_{2})$, we first compute
  \begin{align*}
    \rho_{\theta}\left(\frac{\partial (a^{k}b)}{\partial a}\right) &=
                                            \begin{pmatrix}
                                              P_{k-1}(t\xi^{-\theta}) & 0 \\
                                              0& P_{k-1}(t\xi^{\theta}) \\
                                            \end{pmatrix}, \\
    \rho_{\theta}\left(\frac{\partial (a^{k}b)}{\partial b}\right) &=
                                            \begin{pmatrix}
                                              t^{k}\xi^{-k\theta} & 0 \\
                                              0 & t^{k}\xi^{k\theta} \\
                                            \end{pmatrix}, \\
    \rho_{\theta}\left(\frac{\partial (a^{k}b)^{2k+1}}{\partial a}\right) &=
                                                   \begin{pmatrix}
                                                     P_{k}(t^{-1}) & t^{-1}P_{k-1}(t^{-1}) \\
                                                     P_{k-1}(t^{-1}) & P_{k}(t^{-1})
                                                   \end{pmatrix} \cdot
                                                                       \begin{pmatrix}
                                                                         P_{k-1}(t\xi^{-\theta}) & 0 \\
                                                                         0   & P_{k-1}(t\xi^{\theta}) \\
                                                                       \end{pmatrix}  \\
                                          &=
                                            \begin{pmatrix}
                                              P_{k}(t^{-1})P_{k-1}(t\xi^{-\theta}) & t^{-1}P_{k-1}(t^{-1})P_{k-1}(t\xi^{\theta}) \\
                                              P_{k-1}(t^{-1})P_{k-1}(t\xi^{-\theta}) & P_{k}(t^{-1})P_{k-1}(t\xi^{\theta}) \\
                                            \end{pmatrix}, \\
    \rho\left(\frac{\partial (a^{k}b)^{2k+1}}{\partial b}\right) &=
                                               \begin{pmatrix}
                                                 P_{k}(t^{-1}) & t^{-1}P_{k-1}(t^{-1}) \\
                                                 P_{k-1}(t^{-1}) & P_{k}(t^{-1})
                                               \end{pmatrix} \cdot
                                                                   \begin{pmatrix}
                                                                     t^{k}\xi^{-k\theta} & 0 \\
                                                                     0 & t^{k} \xi^{k\theta} \\
                                                                   \end{pmatrix}  \\
                                          &=
                                            \begin{pmatrix}
                                              P_{k}(t)\xi^{-k\theta }& P_{k-1}(t)\xi^{k\theta} \\
                                              t P_{k-1}(t) \xi^{-k\theta} & P_{k}(t) \xi^{k\theta} \\
                                            \end{pmatrix}.
  \end{align*}
  As a consequence, the second line of~$\partial_2$ is given by
  \begin{align*}
    \rho_{\theta}\left(\frac{\partial \lambda_{0}}{\partial a}\right) &=
                                            \begin{pmatrix}
                                              t^{-k}P_{k-1}(t\xi^{-\theta})P_{2k}(t)  & t^{-k}P_{k-1}(t\xi^{\theta})P_{2k}(t) + t^{-k-1}P_{2k}(t\xi^{\theta}) \\
                                              t^{-k+1}P_{k-1}(t\xi^{-\theta})P_{2k}(t) + t^{-k}P_{2k}(t\xi^{-\theta}) & t^{-k}P_{k-1}(t\xi^{\theta})P_{2k}(t) \\
                                            \end{pmatrix}   \\
                                          &=
                                            \begin{pmatrix}
                                              t^{-k}P_{k-1}(t\xi^{-\theta})P_{2k}(t) & \frac{P_{2k}(t)}{t\xi^{\theta}-1} (\xi^{k\theta}-t^{-k-1})\\
                                              \frac{P_{2k}(t)}{t\xi^{-\theta}-1}(t\xi^{-k\theta}-t^{-k}) & t^{-k}P_{k-1}(t\xi^{\theta})P_{2k}(t) \\
                                            \end{pmatrix}, \\
    \rho_{\theta}\left(\frac{\partial \lambda_{0}}{\partial b}\right) &=
                                            \begin{pmatrix}
                                              \xi^{-k\theta} P_{2k}(t) & \xi^{k\theta} P_{2k}(t) \\
                                              t \xi^{-k\theta} P_{2k}(t) & \xi^{k\theta} P_{2k}(t) \\
                                            \end{pmatrix}.
  \end{align*}
  This completes the first step of our computation of \( \im(\rho_{\theta}(\partial_{2})^{\# T}) \).
  To carry out the second step, consider the following set of vectors of \(C^{1}(M_{T_{2,2k+1}};\C[t^{\pm 1}]_{\rho_{\theta}}^{2})\):
  \begin{align}
    \label{eq:Z1}
    Z_{1} &= \left(\frac{\xi^{-(k-1)\theta}t^{k+1}}{\xi^{\theta}-\xi^{-\theta}}, \frac{\xi^{-\theta}}{\xi^{\theta}-\xi^{-\theta}}, 0, -\frac{t^{2k+1}\xi^{\theta}P_{2k}(t^{-1}\xi^{\theta})}{\xi^{\theta}-\xi^{-\theta}}\right), \\
    Z_{2} &= \left((t^{-1}\xi^{\theta}-1)\xi^{-k\theta}t^{k+1}, t^{-1}\xi^{-\theta}-1, 0, t^{2k+1}-1\right), \nonumber \\
    Z_{3} &= \left(0, t^{-k-2}\xi^{(k-1)\theta}-\xi^{k\theta}t^{-k-1}, 1, -\xi^{k\theta}t^{-k-1}\right),  \nonumber\\
    Z_{4} &= \left(0,0,0,1\right) \nonumber.
  \end{align}
  We claim that this collection of vectors yields a basis of \(C^{1}(M_{T_{2,2k+1}};\C[t^{\pm 1}]_{\rho_{\theta}}^{2})\).
  Indeed, if~\(Z\) denotes the matrix whose rows are \(Z_{1}, Z_{2}, Z_{3}\) and \(Z_{4}\), then \(\det (Z) = -t^{k+1}\xi^{-k\theta}\): this determinant can be computed by first expanding along the fourth row and then expanding along the third column.
  Finally, we carry out the third and last step: using this basis of vectors, we compute the image of~$\rho_\theta(\partial_2)^{\# T}~$ by observing that
  \begin{align}
    \label{eq:Z1v2}
    Z_{1} \cdot \rho_{\theta}(\partial_{2})^{\# T} &= t^{-k+1} \cdot\Delta_{\theta}(t^{-1}) \cdot v_{2}, \\
    Z_{2} \cdot \rho_{\theta}(\partial_{2})^{\# T} &= 0,  \nonumber \\
    Z_{3} \cdot \rho_{\theta}(\partial_{2})^{\# T} &= 0, \nonumber \\
    Z_{4} \cdot \rho_{\theta}(\partial_{2})^{\# T} &= v_{1}. \nonumber
  \end{align}
  We therefore deduce that the twisted homology~$\C[t^{\pm 1}]$-module~$H^{2}(M_{T_{2,2k+1}};\C[t^{\pm 1}]_{\rho_{\theta}}^{2})$ is cyclic with order \(\Delta_{\theta}(t)\).
  This concludes the proof of the first assertion of the proposition.

  In order to compute the twisted Blanchfield pairing, we first compute \(\rho_\theta(\Phi)\) using \eqref{eq:Phi} and then compute the Blanchfield pairing.
  We start with the calculation of the coefficients of the first line of~$\rho_\theta(\Phi)$.
  Looking back to~\eqref{eq:Phi}, we see that each of the two blocs consists of a difference of two expressions.
  We compute each of these terms separately. First, note that we have
  \begin{align*}
    \rho_{\theta}\left(\mu^{2k+1} \cdot \frac{\partial \mu^{-2k-1}}{\partial a}\right) &=
                                                           \begin{pmatrix}
                                                             0       & t^{k} \\
                                                             t^{k+1} & 0 \\
                                                           \end{pmatrix} 
    \begin{pmatrix}
      P_{k}(t^{-1})P_{k-1}(t\xi^{-\theta}) & t^{-1}P_{k-1}(t^{-1})P_{k-1}(t\xi^{\theta}) \\
      P_{k-1}(t^{-1})P_{k-1}(t\xi^{-\theta}) & P_{k}(t^{-1})P_{k-1}(t\xi^{\theta}) \\
    \end{pmatrix} \\
                                                         &=
                                                           \begin{pmatrix}
                                                             t^{k}P_{k-1}(t^{-1})P_{k-1}(t\xi^{-\theta}) & t^{k}P_{k}(t^{-1})P_{k-1}(t\xi^{\theta}) \\
                                                             t^{k+1}P_{k}(t^{-1})P_{k-1}(t\xi^{-\theta}) & t^{k}P_{k-1}(t^{-1})P_{k-1}(t\xi^{\theta}) \\
                                                           \end{pmatrix}, \\
    \rho_{\theta}\left(\mu^{2k+1} \cdot \frac{\partial \mu^{-2k-1}}{\partial b}\right) &=
                                                           \begin{pmatrix}
                                                             0       & t^{k} \\
                                                             t^{k+1} & 0 \\
                                                           \end{pmatrix} 
    \begin{pmatrix}
      P_{k}(t)\xi^{-k\theta }& P_{k-1}(t)\xi^{k\theta} \\
      t P_{k-1}(t) \xi^{-k\theta} & P_{k}(t) \xi^{k\theta} \\
    \end{pmatrix} \\
                                                         &=
                                                           \begin{pmatrix}
                                                             t^{k+1}\xi^{-k\theta}P_{k-1}(t) & t^{k} \xi^{k\theta}P_{k}(t) \\
                                                             t^{k+1} \xi^{-k\theta}P_{k}(t) & t^{k+1} \xi^{k\theta}P_{k-1}(t) \\
                                                           \end{pmatrix}.
  \end{align*}
  The computation of the upper left block of~$\rho_\theta(\Phi)$ also requires us to compute
  \begin{align*}
    \frac{\partial (\mu^{-2k} \cdot a^{k})}{\partial a} &= \frac{\partial{\mu^{-2k}}}{\partial a} + \mu^{-2k} \cdot \frac{\partial{a^{k}}}{\partial a}  \\
                                 &= (1 + \mu^{-1} + \mu^{-2} + \cdots + \mu^{-2k+1}) \cdot \frac{\partial (a^{k} b)}{\partial a} + \mu^{-2k} \cdot (1+a+a^{2}+\cdots+a^{k-1})  \\
                                 &= (1 + \mu^{-1}) \cdot (1 + \mu^{-2} + \cdots (\mu^{-2})^{k-1}) \cdot \frac{\partial a^{k}b}{\partial a} + \mu^{-2k} \cdot (1+a+a^{2}+\cdots+a^{k-1}) \\
    \rho_{\theta}\left(\frac{\partial(\mu^{-2k} \cdot a^{k})}{\partial a}\right) &=
                                                    \begin{pmatrix}
                                                      1 & t^{-1} \\
                                                      1 & 1 \\
                                                    \end{pmatrix} \cdot
    \begin{pmatrix}
      P_{k-1}(t^{-1}) & 0 \\
      0              & P_{k-1}(t^{-1}) \\
    \end{pmatrix} \cdot
    \begin{pmatrix}
      P_{k-1}(t\xi^{-\theta}) & 0 \\
      0                     & P_{k-1}(t\xi^{\theta}) \\
    \end{pmatrix}  \\
                                 &+
                                   \begin{pmatrix}
                                     t^{-k} & 0 \\
                                     0 & t^{-k} \\
                                   \end{pmatrix} \cdot
    \begin{pmatrix}
      P_{k-1}(t\xi^{-\theta}) & 0 \\
      0 & p_{k-1}(t\xi^{\theta}) \\
    \end{pmatrix}  \\
                                 &=
                                   \begin{pmatrix}
                                     P_{k}(t^{-1}) P_{k-1}(t\xi^{-\theta}) & t^{-1}P_{k-1}(t^{-1})P_{k-1}(t\xi^{\theta}) \\
                                     P_{k-1}(t^{-1}) P_{k-1}(t\xi^{-\theta}) & P_{k}(t^{-1})P_{k-1}(t\xi^{\theta})\\
                                   \end{pmatrix},
  \end{align*}
  and, similarly, the upper right block of~$\rho_\theta(\Phi)$ requires that we compute
  \begin{align*}
    \frac{\partial (\mu^{-2k} \cdot a^{k})}{\partial b} &= \frac{\partial \mu^{-2k}}{\partial b} = (1+\mu^{-1}) \cdot (1+\mu^{-2}+\cdots+(\mu^{-2})^{k-1}) \cdot \frac{\partial a^{k}b}{\partial b}, \\
    \rho_{\theta}\left(\frac{\partial (\mu^{-2k} \cdot a^{k})}{\partial b}\right) &=
                                                     \begin{pmatrix}
                                                       1 & t^{-1} \\
                                                       1 & 1 \\
                                                     \end{pmatrix} \cdot
    \begin{pmatrix}
      P_{k-1}(t^{-1}) & 0 \\
      0 & P_{k-1}(t^{-1}) \\
    \end{pmatrix} \cdot
    \begin{pmatrix}
      t^{k} \xi^{-k\theta} & 0 \\
      0 & t^{k} \xi^{k\theta} \\
    \end{pmatrix}  \\
                                 &=
                                   \begin{pmatrix}
                                     t^{k}\xi^{-k\theta}P_{k-1}(t^{-1}) & t^{k-1} \xi^{k\theta} P_{k-1}(t^{-1}) \\
                                     t^{k}\xi^{-k\theta}P_{k-1}(t^{-1}) & t^{k}\xi^{k\theta}P_{k-1}(t^{-1}) 
                                   \end{pmatrix},
  \end{align*}
  Consequently, using the two above sequences of computations, we have
  \begin{align*}
    \rho_{\theta}\left(a^{-k} \cdot \mu^{2k} \cdot \frac{\partial(\mu^{-2k} \cdot a^{k})}{\partial a}\right) &=
                                                                      \begin{pmatrix}
                                                                        \xi^{k\theta} & 0 \\
                                                                        0 & \xi^{-k\theta} \\
                                                                      \end{pmatrix} 
    \begin{pmatrix}
      P_{k}(t^{-1}) P_{k-1}(t\xi^{-\theta}) & t^{-1}P_{k-1}(t^{-1})P_{k-1}(t\xi^{\theta}) \\
      P_{k-1}(t^{-1}) P_{k-1}(t\xi^{-\theta}) & P_{k}(t^{-1})P_{k-1}(t\xi^{\theta})\\
    \end{pmatrix} \\
                                                                    &=
                                                                      \begin{pmatrix}
                                                                        \xi^{k\theta} P_{k}(t^{-1})P_{k-1}(t\xi^{-\theta}) & \xi^{k\theta}t^{-1}P_{k-1}(t^{-1}) P_{k-1}(t\xi^{\theta})\\
                                                                        \xi^{-k\theta} P_{k-1}(t^{-1})P_{k-1}(t\xi^{-\theta}) & \xi^{-k\theta}P_{k}(t^{-1})P_{k-1}(t\xi^{\theta}) \\
                                                                      \end{pmatrix}, \\
    \rho_{\theta}\left(a^{-k} \cdot \mu^{2k} \cdot \frac{\partial(\mu^{-2k} \cdot a^{k})}{\partial b}\right) &=
                                                                      \begin{pmatrix}
                                                                        \xi^{k\theta} & 0 \\
                                                                        0 & \xi^{-k\theta} \\
                                                                      \end{pmatrix}
    \begin{pmatrix}
      t^{k}\xi^{-k\theta}P_{k-1}(t^{-1}) & t^{k-1} \xi^{k\theta} P_{k-1}(t^{-1}) \\
      t^{k}\xi^{-k\theta}P_{k-1}(t^{-1}) & t^{k}\xi^{k\theta}P_{k-1}(t^{-1}) \\
    \end{pmatrix} \\
                                                                    &=
                                                                      \begin{pmatrix}
                                                                        t^{k}P_{k-1}(t^{-1}) & t^{k-1}\xi^{2k\theta}P_{k-1}(t^{-1}) \\
                                                                        t^{k}\xi^{-2k\theta}P_{k-1}(t^{-1}) & t^{k}P_{k-1}(t^{-1}) \\
                                                                      \end{pmatrix}.
  \end{align*}
  Looking at~\eqref{eq:Phi}, assembling these computations and taking the appropriate differences provides an explicit understanding of the first row of~$\rho_\theta(\Phi)$.
  Next, we compute entries in the second row of~$\rho_\theta(\Phi)$:
  \begin{align*}
    \rho_{\theta}\left(-\mu^{-1}\frac{\partial \mu}{\partial a}\right) &=
                                               \begin{pmatrix}
                                                 -P_{k-1}(t\xi^{-\theta}) & 0 \\
                                                 0                      & -P_{k-1}(t\xi^{\theta}) \\ \end{pmatrix}, \\
    \rho_{\theta}\left(-\mu^{-1}\frac{\partial \mu}{\partial b}\right) &=
                                               \begin{pmatrix}
                                                 -t^{k} \xi^{-k\theta} & 0 \\
                                                 0                   & -t^{k}\xi^{k\theta}
                                               \end{pmatrix}.
  \end{align*}
  Finally, using~$\rho_\theta(\Phi)$ and~\eqref{eq:ChainBlanchfieldFormula}, we can compute the cohomological twisted Blanchfield pairing of~$M_{T_{2,2k+1}}$.
  In more details, we know from~\eqref{eq:ChainBlanchfieldFormula} that if~$v \in Z^2(N;\LF^d_{\rho_\theta})$, then
  \[\Bl^{\rho_\theta}(T_{2,2k+1})([v],[v])=\frac{1}{s} \left(v \cdot \rho_\theta(\Phi) \cdot Z^{\# T}\right)^{\# T},\]
    where~$Z \in C^1(N;\LF^d_{\rho_\theta})$ satisfies~$Z\rho_\theta(\partial^2)=sv$ for some~$s \in~\C[t^{\pm 1}] \setminus \lbrace 0 \rbrace$.
  In our case, we take~$v=v_2$ and observe that~$ Z=Z_1$ is described in~\eqref{eq:Z1} and~$s=t^{-k+1}\Delta_\theta(t^{-1})=t^{-k+1}\Delta_\theta(t)$ (recall~\eqref{eq:Z1v2}). Therefore, we start by computing
  \begin{align*}
    Z_{1} \cdot &\rho_{\theta}(\Phi)^{\# T} \cdot v_{2}^{\# T}=\\
         &= \frac{1}{\xi^{\theta}-\xi^{-\theta}} \cdot \left[\xi^{-(k-1)\theta} t^{k+1} P_{k-1}(t^{-1}\xi^{\theta})(1-\xi^{k\theta}t^{k}) + \xi^{-\theta} P_{k-1}(t^{-1}\xi^{-\theta}) (1-\xi^{k\theta}t^{k+1})\right] -\\
         &- \frac{1}{\xi^{\theta}-\xi^{-\theta}}t^{k+1}\xi^{-(k-1)\theta}(1-\xi^{k\theta}t^{k})P_{2k}(t^{-1}\xi^{\theta}) = \\
         &= \frac{1}{\xi^{\theta}-\xi^{-\theta}} \cdot \left[\xi^{-(k-1)\theta} t^{k+1} P_{k-1}(t^{-1}\xi^{\theta})(1-\xi^{k\theta}t^{k}) + \xi^{-\theta} P_{k-1}(t^{-1}\xi^{-\theta}) (1-\xi^{k\theta}t^{k+1})\right] -\\
         &- \frac{1}{\xi^{\theta}-\xi^{-\theta}} \cdot t^{k+1}\xi^{-(k-1)\theta}(1-\xi^{k\theta}t^{k})\left[P_{k-1}(t^{-1}\xi^{\theta}) + t^{-k}\xi^{k\theta}P_{k}(t^{-1}\xi^{\theta})\right] = \\
         &= \frac{1}{\xi^{\theta}-\xi^{-\theta}}\left[\xi^{-\theta}P_{k-1}(t^{-1}\xi^{-\theta})(1-\xi^{-(k+1)\theta}t^{k+1}) - t \cdot \xi^{\theta} (1-t\xi^{\theta}) P_{k-1}(t\xi^{\theta})P_{k}(t^{-1}\xi^{\theta})\right] = \\
         &= \frac{P_{k-1}(t\xi^{\theta}) P_{k}(t\xi^{-\theta})}{\xi^{\theta} - \xi^{-\theta}} \left[t^{-k+1} \xi^{-k\theta} (1-t\xi^{-\theta}) - t^{-k+1}\xi^{(k+1)\theta}(1-t\xi^{\theta})\right] = \\
         &= \xi^{(k+1)\theta}t^{-k+2}P_{k-1}(t\xi^{\theta})P_{k}(t\xi^{-\theta}).
  \end{align*}
  To facilitate computations, we will now arrange that both the numerator and denominator of~$\Bl^{\rho_\theta}(T_{2,2k+1})$ are symmetric. To that effect, we observe that if~$b$ is a symmetric polynomial and~$(a/b)^\#=a/b$ in~$\F(t)/\LF$, then~$\frac{a}{b}=\frac{\frac{1}{2}(a+a^\#)}{b}$ in~$\F(t)/\LF$. Applying this remark to the symmetric polynomial~$b=\Delta_\theta(t)$, using~$\Bl^{\rho_\theta}$ as a shorthand for~$\Bl^{\rho_\theta}(T_{2,2k+1})$ and recalling from above that~$\Bl^{\rho_\theta}([v_2],[v_2])=\frac{1}{s}Z_1 \cdot \rho_\theta(\Phi)^{\#T} \cdot v_2^{\#T}$ with~$s=t^{-k+1}\Delta_\theta(t)$, we obtain
  \begin{align*}
    \Bl^{\rho_\theta}([v_{2}],[v_{2}]) &= \frac{\frac{1}{2}\xi^{(k+1)\theta}t P_{k-1}(t\xi^{\theta})P_{k}(t\xi^{-\theta})}{\Delta_{\theta}(t)} + \frac{\frac{1}{2}\xi^{-(k+1)\theta}t^{-1}P_{k-1}(t^{-1}\xi^{-\theta})P_{k}(t^{-1}\xi^{\theta})}{\Delta_{\theta}(t)} \\
                         &= \frac{\frac{1}{2}\xi^{(k+1)\theta}t P_{k-1}(t\xi^{\theta})P_{k}(t\xi^{-\theta})(t^{-2k-1}+1)}{\Delta_{\theta}(t)} \\
                         &= \frac{\frac{1}{2}(t^{2k+1} \xi^{k\theta} - t^{k+1} -t^{k}+\xi^{(k+1)\theta})(t^{-2k-1}+1)}{t^{-k}P_{2k}(t)}.
  \end{align*}
  This concludes the computation of the twisted Blanchfield pairing on the generator of the twisted cohomology~$\C[t^{\pm 1}]$-module~$H^{2}(M_{T_{2,2k+1}};\C[t^{\pm 1}]_{\rho_{\theta}}^{2})$ and thus the proof of the proposition.
\end{proof}

\subsection{The isometry type of~$\Bl^{\rho_{\theta}}(T_{2,2k+1})$.}
\label{sub:MetabelianSignaturesTorusKnot}

The aim of this subsection is to determine the isometry type of the twisted Blanchfield forms~$\Bl^{\rho_\theta}(T_{2,2k+1})$, where~$\rho_\theta=\alpha(2,\chi_\theta)$. 
The following proposition (which is Theorem~\ref{thm:DecompoTwistedBlanchfieldIntro} from the introduction) implicitly contains the values of the twisted signature jumps~$\delta\sigma_{T_{2,2k+1},\rho_\theta} \colon S^1 \to \Z$.

\begin{theorem}
  \label{thm:DecompoTwistedBlanchfield}
  Set \(\xi = \exp\left(\frac{2 \pi i}{2k+1}\right)\).
  For any \(k>0\) and \(1 \leq \theta \leq k\), there exists an isometry
  \begin{align*}
    \Bl^{\rho_{\theta}}(T_{2,2k+1}) \cong \lambda_{\theta}^{even} \oplus \lambda_{\theta}^{odd},
  \end{align*}
  where the linking forms~$\lambda_{\theta}^{even}$ and~$\lambda_{\theta}^{odd}$ are as follows:
  \begin{align*}
    \lambda_{\theta}^{odd} &= \bigoplus_{\stackrel{1 \leq e \leq k}{2 \nmid \theta + e}} \left(\mathfrak{e}(1,1, \xi^{e},\C) \oplus \mathfrak{e}(1,-1,\xi^{-e},\C)\right),
  \end{align*}
  and
  \begin{align*}
    \lambda_{\theta}^{even} &= \bigoplus_{\stackrel{1 \leq e < \theta}{2 \mid \theta+e}}\left(\mathfrak{e}(1,1,\xi^{e}, \C) \oplus \mathfrak{e}(1,-1,\xi^{-e},\C)\right) \oplus \\
                 &\bigoplus_{\stackrel{\theta < e \leq k}{2 \mid \theta+e}}\left(\mathfrak{e}(1,-1,\xi^{e}, \C) \oplus \mathfrak{e}(1,1,\xi^{-e},\C)\right).
  \end{align*}
\end{theorem}
\begin{proof}
  Using Proposition~\ref{prop:TwistedModuleComputation}, we can choose a generator~$x$ of the cyclic module~$H^{2}(M_{T_{2,2k+1}};\C[t^{\pm 1}]_{\rho_{\theta}}^{2})$ so that~$\Bl^{\rho_{\theta}}(T_{2,2k+1})(x,x) =F(t)$, where
  \[F(t) = \frac{\frac{1}{2}(t^{2k+1} \xi^{k\theta} - t^{k+1} -t^{k}+\xi^{(k+1)\theta})(t^{-2k-1} + 1)}{t^{-k}P_{2k}(t)}.\]
  Since~$\Delta_\theta(t)=\prod_{\stackrel{e=1}{e \neq \theta}}^{k} \basicR_{\xi^{e}}(t)$, we know that the isometry type of~$\Bl^{\rho_{\theta}}(T_{2,2k+1})$ will involve a direct sum of the basic pairings~$ \ee(1,0,\pm 1,\xi^{e},\C)$. In order to determine the correct signs, and prove the theorem, we will apply the procedure described in Subsection~\ref{sub:effective}. The aforementioned signs depend on whether~$F(t)(t-\xi^e)$ is \(\xi^{e}\)-positive or \(\xi^{e}\)-negative, for \(1 \leq e \leq 2k\) and \(e \neq \theta, 2k+1-\theta\). Notice that~$F(t) = \frac{G(t)}{H_{e}(t)} \cdot \frac{1}{\basicR_{\xi^{e}}(t)}$,  where, for \(1 \leq e \leq 2k\), we set
  \begin{align*}
    G(t) &= \frac{1}{2}(t^{2k+1} \xi^{k\theta} - t^{k+1} -t^{k}+\xi^{(k+1)\theta})(t^{-2k-1} + 1), \\
    H_{e}(t) &= \frac{t^{-k}P_{2k}(t)}{\basicR_{\xi^{e}}(t)} = \prod_{\stackrel{i=1}{i \neq e}}^{k} \basicR_{\xi^{i}}(t).
  \end{align*}
  Reformulating, the theorem will immediately follow once we determine the~$e$ for which the following function is~$\xi^e$-positive or~$\xi^e$-negative:
  \begin{equation}\label{eq:twisted-Bl-torus-knot}
    F(t) \cdot (t-\xi^{e}) = (-1)^{\theta+e}\left[(-1)^{\theta+1}G(t)\right] \cdot \left[\frac{(-1)^{e-1}}{H_{e}(t)(1-t^{-1}\xi^{-e})}\right].
  \end{equation}
  In Lemma~\ref{lem:HxiPositive} below, we show that~$\frac{(-1)^{e-1}}{H_{e}(t)(1-t^{-1}\xi^{-e})}$ is~$\xi^e$-positive, while in Lemma~\ref{lem:GxiPositive}, we study the sign of~$(-1)^{\theta+1}G(\xi^e)$.

  \begin{lemma}
    \label{lem:HxiPositive}
    For any \(1 \leq e \leq 2k\), the following function is \(\xi^{e}\)-positive:
    \[\frac{(-1)^{e-1}}{H_{e}(t)(1-t^{-1}\xi^{-e})}.\]
  \end{lemma}
  \begin{proof}
    Notice that \(\frac{1}{1-t^{-1}\xi^{-e}}\) is \(\xi^{e}\)-positive if \(1 \leq e \leq k\) and \(\xi^{e}\)-negative if \(k+1 \leq e \leq 2k\).
    As a consequence, the lemma boils down to determining the sign of~$(-1)^{e-1}H_e(\xi^e)$. We first check this when \(1 \leq e \leq k\). For any \(1 \leq i \leq k\) such that \(e \neq i\), we have
    \[\basicR_{\xi^{i}}(\xi^{e}) = 2\re(\xi^{e}) - 2 \re(\xi^{i})
      \begin{cases}
        >0 & \text{if } e<i, \\
        <0 & \text{if } e>i. \\
      \end{cases}
    \]
    Combining these inequalities with the definition of~$H_e(t)$, we obtain
    \[(-1)^{e-1}H_{e}(\xi^{e}) = \underbrace{(-1)^{e-1}\prod_{i=1}^{e-1}\basicR_{\xi^{i}}(\xi^{e})}_{>0} \cdot \underbrace{\prod_{i=e+1}^{k}\basicR_{\xi^{i}}(\xi^{e})}_{>0}>0.\]
    The case \(k+1 \leq e \leq 2k\) can be reduced the previous one.
    Indeed, notice that \(H_{e}(t) = H_{2k+1-e}(t)\) and \(\basicR_{\xi^{i}}(\xi^{e}) = \basicR_{\xi^{i}}(\xi^{2k+1-e})\), because \(\basicR_{\xi^{i}}(t)\) is a real polynomial. This implies that
    \[(-1)^{e-1}H_{e}(\xi^{e}) = (-1)^{2k+1-e}H_{2k+1-e}(\xi^{2k+1-e})<0.\]
    This concludes the proof of the lemma.
  \end{proof}

  Let us now study the sign of the numerator of~\eqref{eq:twisted-Bl-torus-knot}.
  Notice first that for \(1 \leq e \leq k\), we have
  \[G(\xi^{e}) = G(\xi^{-e}) = G(\xi^{2k+1-e}),\]
  and it is therefore sufficient to determine the sign of \(G(\xi^{e})\) for \(1 \leq e \leq k\).
  \begin{lemma}\label{lem:GxiPositive}
    For \(1 \leq e \leq k\) such that \(e \neq \theta\) we have
    \[(-1)^{\theta+1}G(\xi^{e})
      \begin{cases}
        >0, & \text{\ if\ } 1 \leq e < \theta \leq k \text{\ and\ } 2 \mid \theta + e, \\
        <0, & \text{\ if\ } 1 \leq \theta < e \leq k \text{\ and\ } 2 \mid \theta+e, \\
        <0, & \text{\ if\ } 2 \nmid \theta+e.
      \end{cases}
    \]
  \end{lemma}
  \begin{proof}
    Using successively the definition of~$G(t)$, the fact that~$\xi=\operatorname{exp}(\frac{2\pi i}{2k+1})$ and the identities
    \[\cos (x) - \cos (y) = -2 \sin\left(\frac{x+y}{2}\right) \sin\left(\frac{x-y}{2}\right),\]
    we obtain the following equality:
    \begin{align*}
      G(\xi^{e}) &= 2 (\re(\xi^{k\theta})-\re(\xi^{ke})) \\
            &= 2 \left(\cos\left(\frac{2\pi k\theta}{2k+1}\right) - \cos\left(\frac{2\pi ke}{2k+1}\right)\right)  \\
            &= -4 \sin\left(\frac{\pi k (\theta+e)}{2k+1}\right) \sin\left(\frac{\pi k(\theta-e)}{2k+1}\right)  \\
            &= -4 \sin\left(\frac{(\theta+e)\pi}{2} - \frac{(\theta+e)\pi}{4k+2}\right) \sin\left(\frac{(\theta-e)\pi}{2}-\frac{(\theta-e)\pi}{4k+2}\right) = (\ast).
    \end{align*}
    As consequence, we consider two cases depending on the parity of \(\theta + e\).
    First, suppose that \(\theta+e\) is even.
    Using the identity~$\sin(y+n\pi)=(-1)^{n}\sin(y)$, we deduce that previous expression becomes 
    \begin{align*}
      (\ast) &= 4(-1)^{\theta+1}\sin\left(\frac{(\theta+e)\pi}{4k+2}\right) \sin\left(\frac{(\theta-e)\pi}{4k+2}\right).
    \end{align*}
    Since~$\sin\left(\frac{(\theta+e)\pi}{4k+2}\right)>0$ and \(\sin\left(\frac{(\theta-e)\pi}{4k+2}\right)>0\) if \(\theta>e\), we deduce that for \(\theta+e\) even
    \[(-1)^{\theta+1}(\ast)
      \begin{cases}
        >0, & 1 \leq e < \theta \leq k, \\
        <0, & 1 \leq \theta < e \leq k. \\
      \end{cases}
    \]

    Next, we suppose that \(\theta+e\) is odd.
    Using the identities \(\sin(y+\pi n) = (-1)^{n} \sin(y)\) and \(\sin(y + \pi/2) = \cos(y)\), we obtain
    \begin{align*}
      (\ast) &= 4(-1)^{\theta}\cos\left(\frac{(\theta+e)\pi}{4k+2}\right) \cos\left(\frac{(\theta-e)\pi}{4k+2}\right).
    \end{align*}
    Since \(1 \leq e \leq k\), we have~$\cos\left(\frac{(\theta+e)\pi}{4k+2}\right)>0$ and~$\cos\left(\frac{(\theta-e)\pi}{4k+2}\right)>0$.
    Consequently, if \(\theta+e\) is odd, \((-1)^{\theta+1}(\ast)<0\).
    This concludes the proof of the lemma.
  \end{proof}
  The proof of Theorem~\ref{thm:DecompoTwistedBlanchfield} is now concluded by using Lemmas~\ref{lem:HxiPositive} and~\ref{lem:GxiPositive} as well as~\eqref{eq:twisted-Bl-torus-knot} and the remarks which were made at the beginning of the proof.
\end{proof}

\section{Obstructing the sliceness of algebraic knots }
\label{sec:Sliceness}
The goal of this section is to illustrate how the combination of Theorem~\ref{thm:DecompoTwistedBlanchfield} and the satellite formula from Theorem~\ref{thm:metabelian-cabling-formula} can be used to obstruct certain algebraic knots from being slice.
For concreteness, we focus on the knot $T_{2,3;2,13} \# T_{2,15} \# -T_{2,3;2,15} \# -T_{2,13}$, an example that was previously considered by Hedden, Kirk and Livingston~\cite{HeddenKirkLivingston}.
\color{black}

Throughout this section, for an integer \(\ell>0\) we set~\(\xi_{\ell} = \exp\left(2\pi i/\ell\right)\).
\subsection{Characters on covers of cable knots}
\label{sub:HeddenKirkLivingston}

Given a knot~$K$ and an odd integer~$d$, use~$K_{2,d}$ to denote its~$(2,d)$-cable.
In other words,~$K_{2,d}$ is the satellite knot with pattern the~$(2,d)$ torus knot~$T_{2,d}$, companion~$K$ and infection curve~$\eta = a$ (using the notation from Section~\ref{sub:ChainComplex})~depicted in Figure~\ref{fig:diagram-T_2_2k+1}.

The preimage of~$\eta$ in the \(2\)-fold branched cover~$\Sigma_2(T_{2,d})$ consists of two curves~$\widetilde{\eta}_1,\widetilde{\eta}_2$.
Denote by~$\mu_\eta$ and~$\lambda_\eta$ the meridian and longitude of~$\eta$ and write~$\widetilde{\mu}_i$ and~$\widetilde{\lambda}_i$ for some meridian-longitude pair of the boundary of the tubular neighbourhood of~$\widetilde{U}_i \subset \Sigma_2(T_{2,d})$ for~$i=1,2$.
Note that~$\widetilde{\mu}_1$ and~$\widetilde{\mu}_2$ vanish in~$H_1(\Sigma_2(T_{2,d});\Z)=\Z_d$, while the lift~$\widetilde{\lambda}_1$ generates~$H_1(\Sigma_2(T_{2,d});\Z)$ and~$\widetilde{\lambda}_2=-\widetilde{\lambda}_1$ in~$H_1(\Sigma_2(T_{2,d});\Z)$.
The following topological result is proved in \cite[Section 2]{HeddenKirkLivingston}.
\begin{proposition}
  \label{prop:HKLCable}
  Let~$\ell$ be an odd prime.
  To any character~$\chi \colon H_1(\Sigma_2(K_{2,d});\Z) \to \Z_{\ell}$, one can associate an integer~$\theta$ modulo~$\ell$ by the condition~$\chi(\widetilde{\lambda}_1)=\xi_{\ell}^{\theta}$.
  This character is denoted~$\chi_{\theta}$.
  In particular, this sets up a bijective correspondence between~$\Z_{\ell}$-valued characters on~$H_1(\Sigma_2(K_{2,d});\Z)$ and on~$ H_1(\Sigma_2(T_{2,d});\Z)$.
\end{proposition}

Given an oriented knot~$K$, as is customary in knot concordance, we use~$-K$ to denote the mirror image of~$K$ with the reversed orientation, i.e.~$-K=\overline{K}^r$. The next remark, which follows~\cite[Lemma 3.2]{HeddenKirkLivingston}, describes the characters on~$H_1(\Sigma_2(-K);\Z)$.

\begin{remark}
  \label{rem:ReverseMirror}
  By definition of the reverse mirror image, there is an orientation reversing homeomorphism~$h \colon \Sigma_2(K) \to \Sigma_2(-K)$ and, from now on, it will be understood that we identify the characters on~$H_1(\Sigma_2(K);\Z)$
  and with those on~$H_1(\Sigma_2(-K);\Z)$ via this homeomorphism.
  With this convention, the same proof as in \cite[Proposition~3.4]{BCP_Top} shows that~$\Bl_{\alpha(2,\chi)}(-K)=-\Bl_{\alpha(2,\chi)}(K)$.
  Note however that the character we fixed on~$-K$ is \emph{not} the one obtained by combining the second and third items of \cite[Proposition~3.4]{BCP_Top}.
\end{remark}

\subsection{A concrete example}
\label{sub:ConcreteExample}
Denote by \(T_{\ell,d;r,s}\) the \((r,s)\)-cable of the \((\ell,d)\)-torus knot.
From now on, we consider the following algebraic knot which was thoroughly studied by Hedden, Kirk and Livingston~\cite{HeddenKirkLivingston}:
\begin{align}
  \label{eq:HKLKnot}
  K &= T_{2,3;2,13} \# T_{2,15} \# -T_{2,3;2,15} \# -T_{2,13} \\
    &= K_{1} \# K_{2} \# K_{3} \# K_{4}. \nonumber
\end{align}
Our goal is to study metabelian Blanchfield pairings of the form~$\Bl_{\alpha(2,\chi)}(K)$.
We start by discussing characters on~$H_1(\Sigma_2(K);\Z)$.
Using the decomposition of~$K$ as~$K_{1} \# K_{2} \# K_{3} \# K_{4}$, we obtain the direct sum decomposition~$H_1(\Sigma_2(K);\Z)=H_1(\Sigma_2(K_1);\Z) \oplus \cdots \oplus  H_1(\Sigma_2(K_4);\Z)$.
Furthermore, by Proposition~\ref{prop:HKLCable}, we have the isomorphisms~$H_1(\Sigma_2(K_i);\Z)\cong H_1(\Sigma_2(T_{2,13});\Z)$ for~$i=1,4$ and~$H_1(\Sigma_2(K_i);\Z) \cong H_1(\Sigma_2(T_{2,15});\Z)$ for~$i=2,3$.
Since these isomorphisms identify the corresponding characters, we have obtained the following lemma.

\begin{lemma}
  \label{lem:CharactersOfK}
  Let~$\ell$ be an odd prime.
  For the knot~$K$ described in~\eqref{eq:HKLKnot}, every character~$\chi \colon H_1(\Sigma_2(K);\Z) \to \Z_{\ell}$ can be written as~$\chi= \chi_{1}+\chi_{2}+\chi_{3}+\chi_{4}$ with~$\chi_i:=\chi_{\theta_i} \colon H_1(\Sigma_2(K_i);\Z) \to~\Z_{\ell}$, where \(0 \leq \theta_{1},\theta_{4}\leq 12\) and \(0 \leq \theta_{2},\theta_{3} \leq 14\).
\end{lemma}

\begin{remark}
  \label{rem:HalfCharacters}
  To study the metabelian 
signatures, it is enough to consider the cases \(0 \leq \theta_{1},\theta_{4}\leq 6\) and \(0 \leq \theta_{2},\theta_{3} \leq 7\), indeed this follows from the fact that the representations~\(\alpha(2,\chi_{\alpha})\) and \(\alpha(2,\chi_{-\alpha})\) are equivalent.
  To be more precise, if we set~\(A = \bsm 0&1\\ t&0 \esm \), then we~get~\(A \alpha(2,\chi_{\alpha}) A^{-1} =~\alpha(2,\chi_{-\alpha})\).
\end{remark}

The next proposition describes the Witt class of the metabelian Blanchfield pairing~$ \Bl_{\alpha(2,\chi)}(K)~$.

\begin{proposition}
  \label{prop:MetabelianBlanchfieldHKLKnot}
  Let~$K$ be the knot described in~\eqref{eq:HKLKnot} and let~$\chi \colon H_1(\Sigma_2(K);\Z) \to \Z_{\ell}~$ be a character.
  Write~$\chi=\chi_1+\chi_2+\chi_3+\chi_4$ as in Lemma~\ref{lem:CharactersOfK}, where~$\chi_i=\chi_{\theta_i}$ with \(0 \leq \theta_{1},\theta_{4}\leq 6\) and \(0 \leq \theta_{2},\theta_{3} \leq 7\).
  Then the metabelian Blanchfield form~$ \Bl_{\alpha(2,\chi)}(K)~$ is Witt equivalent to 
  \begin{align}
    \label{eq:HKLMetabelianDecompo}
    & \Bl_{\alpha(2,\chi_{1})}(T_{2,13}) \oplus -\Bl_{\alpha(2,\chi_{4})}(T_{2,13}) \\
    &\oplus \Bl_{\alpha(2,\chi_{2})}(T_{2,15}) \oplus -\Bl_{\alpha(2,\chi_{3})}(T_{2,15})  \nonumber \\
    &\oplus \Bl(T_{2,3})(\xi_{13}^{\theta_{1}}t) \oplus \Bl(T_{2,3})(\xi_{13}^{-\theta_{1}}t) \oplus -\Bl(T_{2,3})(\xi_{15}^{\theta_{3}}t) \oplus -\Bl(T_{2,3})(\xi_{15}^{-\theta_{3}}t). \nonumber
  \end{align}
\end{proposition}
\begin{proof}
  Since we know from Corollary~\ref{cor:MetabelianConnectedSum} that metabelian Blanchfield pairings are additive, up to Witt equivalence, we need only study the metabelian Blanchfield pairing of~$(2,d)$-cables of~$(2,2k+1)$-torus knots (here, we also used Remark~\ref{rem:ReverseMirror}).
  The proposition will follow from the claim that given a~$(2,2k+1)$-torus knot \(K'\) and a character~$\chi=\chi_\theta$ on~$H_1(L_2(K'_{2,d});\Z)$, there is an isometry
  $$\Bl_{\alpha(2,\chi_{\theta})}(K'_{2,d}) \cong \Bl_{\alpha(2,\chi_{\theta})}(T_{2,d}) \oplus \Bl(K')(\xi_{\ell}^{-\theta}t) \oplus \Bl(K')(\xi_{\ell}^{\theta}t).$$
  Using the notation from Section~\ref{sub:cabling},~$K'_{2,d}$ is a satellite knot with pattern \(P = T_{2,d}\), companion~$K'$, and the infection curve \(\eta\) is in fact the curve which was denoted by \(a\) in Subsection~\ref{sub:blanchf-forms-twist}.
  Since the winding number is~$w=2$, the first assertion of Theorem~\ref{thm:metabelian-cabling-formula} implies that~$\alpha(2,\chi)$ is~$\eta$-regular and therefore restricts to a representation~$\alpha(2,\chi)_{K'}$ on~$\pi_1(M_{K'})$.
  Since~$n=2$ divides~$w=2$, the representation~$\alpha(2,\chi)_{K'}$ is abelian.
  As the curve~$a$ is a generator of~$H_1(M_{K'};\Z)$, we see that~$\alpha(2,\chi)_{K'}$ is determined by~$\alpha(2,\chi_{\theta})(a) =
  \bsm
  t \xi_{\ell}^{-\theta} & 0 \\
  0                 & t \xi_{\ell}^{\theta} \\
  \esm$.
  The claim (and thus the proposition) now follow by applying Theorem~\ref{thm:metabelian-cabling-formula}.
\end{proof}

Next, we determine the conditions under which~$\Bl_{\alpha(2,\chi)}(K)$ is metabolic.

\begin{proposition}
  \label{prop:MetabolicCriterion}
  Let~$K$ be the knot described in~\eqref{eq:HKLKnot} and let~$\chi \colon H_1(\Sigma_2(K);\Z) \to \Z_{\ell}~$ be a character.
  Write~$\chi=\chi_1+\chi_2+\chi_3+\chi_4$ as in Lemma~\ref{lem:CharactersOfK} where~$\chi_i=\chi_{\theta_i}$ with \(0 \leq \theta_{1},\theta_{4}\leq 6\) and \(0 \leq \theta_{2},\theta_{3} \leq 7\).
  Then the metabelian Blanchfield pairing \(\Bl_{\alpha(2,\chi)}(K)\) is metabolic if and only if \(\theta_{1}=\theta_{2}=\theta_{3}=\theta_{4}=0\).
\end{proposition}
\begin{proof}
  First, recall from Theorem~\ref{thm:witt_equivalence} that if a linking form is metabolic, then its signature jumps vanish.
  Substituting~$\theta_{1}=\theta_{2}=\theta_{3}=\theta_{4}=0$ into~\eqref{eq:HKLMetabelianDecompo} shows that~$\Bl_{\alpha(2,\chi)}(K)$ is metabolic.
  
  We now prove the converse in two steps.
  Firstly, we show that if~$\theta_1 \neq \theta_4$ (or~$\theta_2 \neq \theta_3$), then~$\Bl_{\alpha(2,\chi)}(K)$ is not metabolic.
  Secondly, we show that when~$\theta_1=\theta_4$ and~$\theta_3=\theta_2$, the metabelian Blanchfield form~$\Bl_{\alpha(2,\chi)}(K)$ is metabolic if and only if~$\theta_{1} = \theta_{3} = 0$.

  Assume that \(\theta_{1} \neq \theta_{4}\).
  We assert that the signature jump of \(\Bl_{\alpha(2,\chi)}(K)\) at \(\xi_{13}^{\theta_{1}}\) is \(\pm1\).
  To see this, first note that Theorem~\ref{thm:DecompoTwistedBlanchfield} (or a glance at the twisted Alexander polynomial) implies that in~\eqref{eq:HKLMetabelianDecompo}, only the summand \(\Bl_{\alpha(2,\chi_{1})}(T_{2,13}) \oplus -\Bl_{\alpha(2,\chi_{4})}(T_{2,13})\) can contribute a non-trivial signature jump at \(\xi_{13}^{\theta_{1}}\): indeed Theorem~\ref{thm:DecompoTwistedBlanchfield} shows that~$\Bl_{\alpha(2,\chi_{2})}(T_{2,15}) \oplus -\Bl_{\alpha(2,\chi_{3})}(T_{2,15})$ can only jump at powers of~$\xi_{15}$.
  Since the untwisted terms in~\eqref{eq:HKLMetabelianDecompo} do not contribute to the signature jump either, we focus on~$\Bl_{\alpha(2,\chi_{1})}(T_{2,13}) \oplus -\Bl_{\alpha(2,\chi_{4})}(T_{2,13})$.
Theorem~\ref{thm:DecompoTwistedBlanchfield} implies that the signature function of~\(\Bl_{\alpha(2,\chi_{1}) }(T_{2,13})\) jumps at~$\xi_{13}^e$ when~$e \neq \theta_1$.
  Therefore, the signature jump of \(\Bl_{\alpha(2,\chi_{1})}(T_{2,13})\) at~\(\xi_{13}^{\theta_{1}}\) is trivial and, since~$\theta_4 \neq \theta_1$, the signature jump of \(-\Bl_{\alpha(2,\chi_{4})}(T_{2,13})\) at~\(\xi_{13}^{\theta_{1}}\) is~\(\pm1\).
  This concludes the proof of the assertion.
  Using this assertion and Theorem~\ref{thm:witt_equivalence}, we deduce that~\(\Bl_{\alpha(2,\chi)}(K)\) is not metabolic.
  The case where \(\theta_{2} \neq \theta_{3}\) is treated analogously.

  Next, we assume that \(\theta_{1} = \theta_{4}\) and \(\theta_{2}=\theta_{3}\).
  Using~\eqref{eq:HKLMetabelianDecompo}, this assumption implies that the metabelian Blanchfield form~$\Bl_{\alpha(2,\chi)}(K)~$ is Witt equivalent to 
  \begin{equation}
    \label{eq:UntwistedBlanchfieldTorusKnots}
    \Bl(T_{2,3})(\xi_{13}^{\theta_{1}}t) \oplus \Bl(T_{2,3})(\xi_{13}^{-\theta_{1}}t) \oplus -\Bl(T_{2,3})(\xi_{15}^{\theta_{3}}t) \oplus -\Bl(T_{2,3})(\xi_{15}^{-\theta_{3}}t).
  \end{equation}
  To determine whether~$\Bl_{\alpha(2,\chi)}(K)~$ is metabolic, Theorem~\ref{thm:witt_equivalence} implies that we must study the jumps of the signature function of the linking form in~\eqref{eq:UntwistedBlanchfieldTorusKnots}.
  Since we are dealing with untwisted Blanchfield forms, these jumps are the signature jumps of the corresponding Levine-Tristram signature function; see Remark~\ref{rem:LevineTristram}.
  The proof of \cite[Theorem 7.1]{HeddenKirkLivingston} shows that for distinct~$a_i$, the jumps (as~$\omega$ varies along~$S^1$) of the Levine-Tristram signatures~$\sigma_{T_{m,n}}(\xi_\ell^{a_i} \omega)$ occur at distinct points (here~$m,n$ are prime).
  Consequently, we deduce that \(\Bl_{\alpha(2,\chi)}(K)\) is metabolic if and only if \(\theta_{1} = \theta_{3} = 0\).
  This concludes proof of the theorem.
\end{proof}

We recover a result of Hedden, Kirk and Livingston~\cite{HeddenKirkLivingston}.

\begin{theorem}
  \label{thm:HKLNotSlice}
  The knot \(K\) from~\eqref{eq:HKLKnot} is algebraically slice but not slice.
\end{theorem}
\begin{proof}
  Hedden-Kirk-Livingston show that \(K\) is algebraically slice~\cite[Lemma 2.1]{HeddenKirkLivingston}.
  By means of contradiction, assume that~$K$ is slice.
  Theorem~\ref{thm:cg_obs} implies that for any prime power~$\ell$, there exists a metaboliser~$P$ of~$\lambda_{\ell}$ such that for any prime power~$q^a$, and any non-trivial character~$\chi \colon  H_1(L_{\ell}(K);\Z) \to \Z_{q^a}$ vanishing on~$P$, we have some~$b \geq a$ such that the metabelian Blanchfield pairing~$\Bl_{\alpha(n,\chi_b)}(K)$ is metabolic.
  As Proposition~\ref{prop:MetabolicCriterion} shows that the Blanchfield pairing~$\Bl_{\alpha(n,\chi_b)}(K)$ is metabolic if and only if the character is trivial, we obtain the desired~contradiction.
  This concludes the proof of the theorem.
\end{proof}



\bibliographystyle{plain}
\def\MR#1{}
\bibliography{research}
\end{document}

%% file: torus-knot.pdf_tex
\begingroup%
  \makeatletter%
  \providecommand\color[2][]{%
    \errmessage{(Inkscape) Color is used for the text in Inkscape, but the package 'color.sty' is not loaded}%
    \renewcommand\color[2][]{}%
  }%
  \providecommand\transparent[1]{%
    \errmessage{(Inkscape) Transparency is used (non-zero) for the text in Inkscape, but the package 'transparent.sty' is not loaded}%
    \renewcommand\transparent[1]{}%
  }%
  \providecommand\rotatebox[2]{#2}%
  \newcommand*\fsize{\dimexpr\f@size pt\relax}%
  \newcommand*\lineheight[1]{\fontsize{\fsize}{#1\fsize}\selectfont}%
  \ifx\svgwidth\undefined%
    \setlength{\unitlength}{359.33895676bp}%
    \ifx\svgscale\undefined%
      \relax%
    \else%
      \setlength{\unitlength}{\unitlength * \real{\svgscale}}%
    \fi%
  \else%
    \setlength{\unitlength}{\svgwidth}%
  \fi%
  \global\let\svgwidth\undefined%
  \global\let\svgscale\undefined%
  \makeatother%
  \begin{picture}(1,0.43437189)%
    \lineheight{1}%
    \setlength\tabcolsep{0pt}%
    \put(0.00166743,0.32687811){\color[rgb]{0,0,0}\makebox(0,0)[lt]{\begin{minipage}{0.15840097\unitlength}\raggedright \end{minipage}}}%
    \put(0,0){\includegraphics[width=\unitlength,page=1]{torus-knot.pdf}}%
    \put(0.15447777,0.0656098){\color[rgb]{0,0,0}\makebox(0,0)[lt]{\lineheight{1.25}\smash{\begin{tabular}[t]{l}\(x_1\)\end{tabular}}}}%
    \put(0.14702359,0.32836901){\color[rgb]{0,0,0}\makebox(0,0)[lt]{\lineheight{1.25}\smash{\begin{tabular}[t]{l}\(x_2\)\end{tabular}}}}%
    \put(0.37951084,0.05573341){\color[rgb]{0,0,0}\makebox(0,0)[lt]{\lineheight{1.25}\smash{\begin{tabular}[t]{l}\(x_2\)\end{tabular}}}}%
    \put(0.37830095,0.33269278){\color[rgb]{0,0,0}\makebox(0,0)[lt]{\lineheight{1.25}\smash{\begin{tabular}[t]{l}\(x_3\)\end{tabular}}}}%
    \put(0,0){\includegraphics[width=\unitlength,page=2]{torus-knot.pdf}}%
    \put(0.5857012,0.06262802){\color[rgb]{0,0,0}\makebox(0,0)[lt]{\lineheight{1.25}\smash{\begin{tabular}[t]{l}\(x_{2k+1}\)\end{tabular}}}}%
    \put(0.60396383,0.32203303){\color[rgb]{0,0,0}\makebox(0,0)[lt]{\lineheight{1.25}\smash{\begin{tabular}[t]{l}\(x_1\)\end{tabular}}}}%
    \put(0.86635043,0.33209599){\color[rgb]{0,0,0}\makebox(0,0)[lt]{\lineheight{1.25}\smash{\begin{tabular}[t]{l}\(x_2\)\end{tabular}}}}%
    \put(0.87380452,0.06411891){\color[rgb]{0,0,0}\makebox(0,0)[lt]{\lineheight{1.25}\smash{\begin{tabular}[t]{l}\(x_1\)\end{tabular}}}}%
    \put(0,0){\includegraphics[width=\unitlength,page=3]{torus-knot.pdf}}%
    \put(0.75677422,0.33321422){\color[rgb]{0,0,0}\makebox(0,0)[lt]{\lineheight{1.25}\smash{\begin{tabular}[t]{l}\(a\)\end{tabular}}}}%
  \end{picture}%
\endgroup%

%% file: Heegard-decomposition.pdf_tex
\begingroup%
  \makeatletter%
  \providecommand\color[2][]{%
    \errmessage{(Inkscape) Color is used for the text in Inkscape, but the package 'color.sty' is not loaded}%
    \renewcommand\color[2][]{}%
  }%
  \providecommand\transparent[1]{%
    \errmessage{(Inkscape) Transparency is used (non-zero) for the text in Inkscape, but the package 'transparent.sty' is not loaded}%
    \renewcommand\transparent[1]{}%
  }%
  \providecommand\rotatebox[2]{#2}%
  \newcommand*\fsize{\dimexpr\f@size pt\relax}%
  \newcommand*\lineheight[1]{\fontsize{\fsize}{#1\fsize}\selectfont}%
  \ifx\svgwidth\undefined%
    \setlength{\unitlength}{336.47744768bp}%
    \ifx\svgscale\undefined%
      \relax%
    \else%
      \setlength{\unitlength}{\unitlength * \real{\svgscale}}%
    \fi%
  \else%
    \setlength{\unitlength}{\svgwidth}%
  \fi%
  \global\let\svgwidth\undefined%
  \global\let\svgscale\undefined%
  \makeatother%
  \begin{picture}(1,0.97280405)%
    \lineheight{1}%
    \setlength\tabcolsep{0pt}%
    \put(0,0){\includegraphics[width=\unitlength,page=1]{Heegard-decomposition.pdf}}%
    \put(0.56624401,0.54945807){\color[rgb]{0,0,0}\makebox(0,0)[lt]{\lineheight{1.25}\smash{\begin{tabular}[t]{l} \(H_{1}\)\end{tabular}}}}%
    \put(0.00188523,0.8790315){\color[rgb]{0,0,0}\makebox(0,0)[lt]{\lineheight{1.25}\smash{\begin{tabular}[t]{l} \(H_2\)\end{tabular}}}}%
    \put(0.90580837,0.77010436){\color[rgb]{0,0,0}\makebox(0,0)[lt]{\lineheight{1.25}\smash{\begin{tabular}[t]{l} \(T\)\end{tabular}}}}%
    \put(0,0){\includegraphics[width=\unitlength,page=2]{Heegard-decomposition.pdf}}%
  \end{picture}%
\endgroup%

%% file: Curves-on-the-torus.pdf_tex
\begingroup%
  \makeatletter%
  \providecommand\color[2][]{%
    \errmessage{(Inkscape) Color is used for the text in Inkscape, but the package 'color.sty' is not loaded}%
    \renewcommand\color[2][]{}%
  }%
  \providecommand\transparent[1]{%
    \errmessage{(Inkscape) Transparency is used (non-zero) for the text in Inkscape, but the package 'transparent.sty' is not loaded}%
    \renewcommand\transparent[1]{}%
  }%
  \providecommand\rotatebox[2]{#2}%
  \newcommand*\fsize{\dimexpr\f@size pt\relax}%
  \newcommand*\lineheight[1]{\fontsize{\fsize}{#1\fsize}\selectfont}%
  \ifx\svgwidth\undefined%
    \setlength{\unitlength}{386.6655bp}%
    \ifx\svgscale\undefined%
      \relax%
    \else%
      \setlength{\unitlength}{\unitlength * \real{\svgscale}}%
    \fi%
  \else%
    \setlength{\unitlength}{\svgwidth}%
  \fi%
  \global\let\svgwidth\undefined%
  \global\let\svgscale\undefined%
  \makeatother%
  \begin{picture}(1,0.99969683)%
    \lineheight{1}%
    \setlength\tabcolsep{0pt}%
    \put(0,0){\includegraphics[width=\unitlength,page=1]{Curves-on-the-torus.pdf}}%
  \end{picture}%
\endgroup%

%% file: Curves-on-the-torus-with-nbhds.pdf_tex
\begingroup%
  \makeatletter%
  \providecommand\color[2][]{%
    \errmessage{(Inkscape) Color is used for the text in Inkscape, but the package 'color.sty' is not loaded}%
    \renewcommand\color[2][]{}%
  }%
  \providecommand\transparent[1]{%
    \errmessage{(Inkscape) Transparency is used (non-zero) for the text in Inkscape, but the package 'transparent.sty' is not loaded}%
    \renewcommand\transparent[1]{}%
  }%
  \providecommand\rotatebox[2]{#2}%
  \newcommand*\fsize{\dimexpr\f@size pt\relax}%
  \newcommand*\lineheight[1]{\fontsize{\fsize}{#1\fsize}\selectfont}%
  \ifx\svgwidth\undefined%
    \setlength{\unitlength}{392.84626007bp}%
    \ifx\svgscale\undefined%
      \relax%
    \else%
      \setlength{\unitlength}{\unitlength * \real{\svgscale}}%
    \fi%
  \else%
    \setlength{\unitlength}{\svgwidth}%
  \fi%
  \global\let\svgwidth\undefined%
  \global\let\svgscale\undefined%
  \makeatother%
  \begin{picture}(1,0.99920561)%
    \lineheight{1}%
    \setlength\tabcolsep{0pt}%
    \put(0,0){\includegraphics[width=\unitlength,page=1]{Curves-on-the-torus-with-nbhds.pdf}}%
  \end{picture}%
\endgroup%

%% file: Splitting-into-handles.pdf_tex
\begingroup%
  \makeatletter%
  \providecommand\color[2][]{%
    \errmessage{(Inkscape) Color is used for the text in Inkscape, but the package 'color.sty' is not loaded}%
    \renewcommand\color[2][]{}%
  }%
  \providecommand\transparent[1]{%
    \errmessage{(Inkscape) Transparency is used (non-zero) for the text in Inkscape, but the package 'transparent.sty' is not loaded}%
    \renewcommand\transparent[1]{}%
  }%
  \providecommand\rotatebox[2]{#2}%
  \newcommand*\fsize{\dimexpr\f@size pt\relax}%
  \newcommand*\lineheight[1]{\fontsize{\fsize}{#1\fsize}\selectfont}%
  \ifx\svgwidth\undefined%
    \setlength{\unitlength}{557.26741208bp}%
    \ifx\svgscale\undefined%
      \relax%
    \else%
      \setlength{\unitlength}{\unitlength * \real{\svgscale}}%
    \fi%
  \else%
    \setlength{\unitlength}{\svgwidth}%
  \fi%
  \global\let\svgwidth\undefined%
  \global\let\svgscale\undefined%
  \makeatother%
  \begin{picture}(1,0.60425367)%
    \lineheight{1}%
    \setlength\tabcolsep{0pt}%
    \put(0,0){\includegraphics[width=\unitlength,page=1]{Splitting-into-handles.pdf}}%
    \put(0.00108825,0.24954873){\color[rgb]{0,0,0}\makebox(0,0)[lt]{\lineheight{1.25}\smash{\begin{tabular}[t]{l} \(B_{2}\)\end{tabular}}}}%
    \put(0.48800312,0.20565786){\color[rgb]{0,0,0}\makebox(0,0)[lt]{\lineheight{1.25}\smash{\begin{tabular}[t]{l} \(B_{1}\)\end{tabular}}}}%
    \put(0.33645995,0.03768914){\color[rgb]{0,0,0}\makebox(0,0)[lt]{\lineheight{1.25}\smash{\begin{tabular}[t]{l} \(\partial_{+,1} B_{1}\)\end{tabular}}}}%
    \put(0.67492558,0.31430821){\color[rgb]{0,0,0}\makebox(0,0)[lt]{\lineheight{1.25}\smash{\begin{tabular}[t]{l} \(\partial_{+,2} B_{1}\)\end{tabular}}}}%
    \put(0.73656505,0.58562569){\color[rgb]{0,0,0}\makebox(0,0)[lt]{\lineheight{1.25}\smash{\begin{tabular}[t]{l} \(\partial_{-}B_{1}\)\end{tabular}}}}%
    \put(0,0){\includegraphics[width=\unitlength,page=2]{Splitting-into-handles.pdf}}%
  \end{picture}%
\endgroup%